\documentclass[11pt,oneside,english,reqno]{amsart}

\usepackage[T1]{fontenc} 
\usepackage[latin9]{inputenc} 
\usepackage{babel} 

\usepackage{amsthm}
\usepackage{amsmath}
\usepackage{amstext}
\usepackage{amssymb}
\usepackage{bbm}
\usepackage{mathrsfs} 
\usepackage{cancel}

\usepackage{color} 
\usepackage[usenames,dvipsnames]{xcolor} 
\usepackage{setspace} 
\usepackage{graphicx} 
\usepackage[unicode=true, pdfusetitle, bookmarks=true, bookmarksnumbered=false, bookmarksopen=false,
 breaklinks=false, pdfborder={0 0 0}, backref=false,colorlinks=true]{hyperref}
\definecolor{myblue}{rgb}{0.09,0.32,0.44} 
\hypersetup{pdfborder={0 0 0}, pdfborderstyle={},colorlinks=true, linkcolor=myblue, citecolor=myblue, urlcolor=blue}
\usepackage{verbatim} 
\usepackage[toc,page]{appendix}
\usepackage{enumitem}
\usepackage{accents}
\allowdisplaybreaks

\usepackage{perpage}
\MakePerPage{footnote}
\gdef\SetFigFontNFSS#1#2#3#4#5{} 
\gdef\SetFigFont#1#2#3#4#5{} 
\def\clap#1{\hbox to 0pt{\hss#1\hss}}

\calclayout
\usepackage{etoolbox}
\patchcmd{\section}{\scshape}{\bfseries}{}{}
\makeatletter
\renewcommand{\@secnumfont}{\bfseries}
\makeatother

\usepackage{ulem}

\begin{document}

\newenvironment{subproof}[1][\proofname]
{\renewcommand{\qedsymbol}{$\square$}\begin{proof}[#1]}{\end{proof}}
\numberwithin{equation}{section}
\renewcommand{\theequation}{\thesection.\arabic{equation}}
\newtheorem{thm}{Theorem}[section]
\newtheorem{lem}[thm]{Lemma}
\newtheorem{prop}[thm]{Proposition}
\newtheorem{cor}[thm]{Corollary}
\newtheorem{conj}[thm]{Conjecture}
\newtheorem{definition}[thm]{Definition}
\newtheorem{prob}[thm]{Problem}
\newtheorem{rmk}[thm]{\normalfont\textit{Remark}}
\newtheorem{prethm}{Theorem}
\renewcommand{\theprethm}{\Alph{prethm}}

\newcommand{\norm}[1]{\left\lVert #1 \right\rVert}
\newcommand\numberthis{\addtocounter{equation}{1}\tag{\theequation}}
\newcommand{\beq}{\begin{equation}} 
\newcommand{\eeq}{\end{equation}} 
\newcommand{\euc}[1]{|{#1}|}
\newcommand{\wt}{\widetilde}
\newcommand{\supp}{\text{\rm supp}}
\newcommand{\rdep}{M}
\newcommand{\deuc}{\text{\rm \bf dist}}
\newcommand{\Eb}{\mathbf{E}}
\newcommand{\Pb}{\mathbf{P}}
\newcommand{\Nb}{\mathbb{N}}
\newcommand{\Rb}{\mathbb{R}}
\newcommand{\Rc}{\mathcal{R}}
\newcommand{\Tc}{\mathcal{T}}
\newcommand{\hW}{\hat{W}}
\newcommand{\Zb}{\mathbb{Z}}
\newcommand{\eps}{\epsilon}
\newcommand{\zero}{\mathbf{0}}
\newcommand{\one}{{\mathbbm{1}}}
\newcommand{\Ac}{\mathcal{A}}
\newcommand{\Kc}{\mathcal{K}}
\newcommand{\Pc}{\mathcal{P}}
\newcommand{\Ic}{\mathcal{I}}
\newcommand{\Jc}{\mathcal{J}}
\newcommand{\Wc}{\mathcal{W}}
\newcommand{\Xc}{\mathcal{X}}
\newcommand{\Xct}{\widetilde{\mathcal{X}}}
\renewcommand{\Mc}{\mathcal{M}}
\newcommand{\scr}{\scriptstyle}
\newcommand{\sss}{\scriptscriptstyle}
\renewcommand{\ker}{\lambda}
\newcommand{\temp}{c}
\newcommand{\sep}{\text{\rm sep}}
\newcommand{\dzero}{\delta_{\mathbf{0}}}
\newcommand{\Euc}{\mathrm{Euc}}
\begin{abstract}
The first main goal of this article is to give a new metrization of the Mukherjee--Varadhan topology, recently introduced as a translation-invariant compactification of the space of probability measures on Euclidean spaces. This new metrization allows us to achieve our second goal which is to extend the recent program of Bates and Chatterjee on localization for the endpoint distribution of discrete directed polymers to polymers based on general random walks in Euclidean spaces. 
Following their strategy, we study the asymptotic behavior of the endpoint distribution update map and study the set of its distributional fixed points satisfying a variational principle. We show that the distributdion concentrated on the zero measure is a unique element in this set if and only if the system is in the high temperature regime. 
This enables us to prove that the asymptotic clustering (a natural continuous analogue of the asymptotic pure atomicity property) holds in the low temperature regime and that the endpoint distribution is geometrically localized with positive density if and only if the system is in the low temperature regime.
\end{abstract}

\title{Localization of directed polymers in continuous space}
\author{Yuri Bakhtin, Donghyun Seo}
\address{Courant Institute of Mathematical Sciences\\
New York University\\
251 Mercer Street\\
New York, N.Y. 10012-1185}
\email{bakhtin@cims.nyu.edu, seo@cims.nyu.edu}
\maketitle

\section{Introduction} \label{sec: introduction}
The directed polymer model was introduced in the physics literature \cite{HH85}, \cite{HHF85}, \cite{Kar85}, \cite{MN85}, \cite{KZ87} and mathematically formulated by Imbrie and Spencer~\cite{IS88}.
 Since then, many models of directed polymers in random environment were studied in the literature over last several decades, 
see, e.g. books \cite{Szn98}, \cite{Gia07}, \cite{dH09}, \cite{Com17} and multiple references therein.
The common feature of these models is that they are based on Gibbs distributions on paths with the reference measure usually describing a process with independent increments (random walks, if the time is discrete) and the energy of the interaction between the path and the environment is given by a space-time random potential (with some decorrelation properties) accumulated along the path.

One of the intriguing phenomena that these models exhibit is the transition of dynamics of directed polymers between
high/low temperature regimes. 
In the high temperature regime, directed polymers have diffusive behavior
which is similar to that of the classical random walks and the endpoint distributions of polymer paths of length~$n$ are typically spread over 
domains of size of the order of~$n^{1/2}$ (see \cite{Bol89}, \cite{SZ96}, \cite{AZ96}). 
On the other hand, in the low temperature regime, 
they are super-diffusive, i.e. the typical transverse displacement of polymer paths is of the order of $n^{\xi}$ with $\xi>1/2$.
In particular, it has been conjectured that $\xi = 2/3$ for $d=1$, based on two following observations: (\romannumeral 1) when $\beta = +\infty$, the directed polymer models coincide with the last passage percolation (LPP) models; (\romannumeral 2) Integrable LPP models have shown the spatial fluctuation of order $n^{2/3}$ and the fluctuation of passage times of order $n^{1/3}$ placing LPP in the KPZ universality class \cite{Cor12}. This has been proved in some integrable models, see \cite{Sep12}, \cite{BCF14}. 
Besides the super-diffusive behavior, it is known that polymer measures are mostly concentrated within a relatively small region in the low temperature regime, see \cite{CV06},  \cite{Var07}, \cite{Lac10}, \cite{BK10}. 
Such localization phenomenon of directed polymers is closely related to the intermittency of the solution of stochastic heat equation, see \cite{CM94}, \cite{BC95}, \cite{Kho14}.
It is believed that the size of the small region is $\mathcal{O}(1)$ but this has been proved only for integrable models, see \cite{CN16}.
It is also conjectured that a similar picture holds for generalized directed polymers, see \cite{BK18}.

While many integrable models for (1+1)-dimensional directed polymers have been extensively studied (see \cite{MO07}, \cite{ACQ11}, \cite{MFQR13}, \cite{AKQ14}), the results on higher dimensions are rather restricted. In~\cite{BC16} and its improved version~\cite{Bat18},
a novel machinery was suggested to study localization  of directed polymers that are discrete in space and time.
This approach is based on another recent achievement, a compactification of the space of probability measures on $\Rb^d$ with respect to the weak convergence~\cite{MV16} (we will refer to this compactification as the \textit{MV topology} in this paper).  In~\cite{BC16}, the authors introduce a simple metrization of the MV topology induced on the space of measures concentrated on~$\Zb^d$ and they were able to obtain localization results for discrete directed polymers by using the metric.

The first goal of this paper is to develop a new metrization of the MV topology that will be useful for space-continuous polymer models. 
Our new metrization is inspired by the one used in the discrete setting in~\cite{BC16} and is based on coupling in optimal transport. Its relation to the metric given in~\cite{MV16}  resembles the equivalence between the definitions of the Kantorovich--Wasserstein distance via 
optimal coupling~\eqref{eq: Wasserstein metric} and via Lipschitz test functions~\eqref{eq: Kantorovich duality} known as the Kantorovich duality.

The second goal of this paper is to introduce a broad family of time-discrete and 
space-continuous polymer models where polymers are understood as discrete sequences of points in $\Rb^d$, and to generalize the entire program of~\cite{BC16} to these models with the help of our new metrization of the MV topology.

As this paper was being prepared we learned that similar results were obtained in~\cite{BM18} for a specific model 
where the reference measure is Brownian and the random potential is the space-time white noise mollified with respect to the space variable. 
We stress that the only assumption we need on the reference measure for polymers is that it defines a random walk, with no restriction on the distribution of \textit{i.i.d.}\ steps in contrast to a concrete model of~\cite{BM18}.

Due to the absence of assumptions on the random walk steps, we can say that  our results generalize those of \cite{BC16} and~\cite{Bat18}  that are restricted to lattice
random walks (except that a moment assumption on the potential is slightly weaker in~\cite{Bat18})  since one can embed any \textit{i.i.d.}\ random potential indexed by $\Zb^d$ into a stationary potential on $\Rb^d$ with a small dependence range.

In addition, we give a new result that goes beyond the asymptotic pure atomicity results of~\cite{BC16} and~\cite{BM18}. 
Under the assumption that the reference measure is absolutely continuous with respect to the Lebesgue measure, several forms of asymptotic clustering property hold for the random density of the polymer endpoint distribution in low temperature regime. An important feature of our work is that our results are based on the new metrization of the MV topology which is of independent interest. However, a similar program was executed in~\cite{BM18} using the original metrization.

The article is organized as follows: 
In the remaining part of Section~\ref{sec: introduction}, we introduce our general model of directed polymers, review the results in discrete setting, 
and state our results for localization/delocalization of directed polymers. 
In Section~\ref{sec: MV}, we review the MV topology and 
introduce a new metric which is equivalent to the original MV metric and useful for our analysis of polymer measures.
In Sections~\ref{sec: update map},~\ref{sec: convergence of empirical measure}, and~\ref{sec: characterization of high/low temperature regime}, we develop a program parallel to \cite{BC16}, proving the continuity of the update map that maps the law of the endpoint distribution to the one of the next step endpoint distribution and proving that 
the empirical measure of the endpoint distribution of directed polymers converges to the set of free energy minimizers which is a subset of the set of fixed points of the update map.
We will also see how the set of free energy minimizers can characterize the high/low temperature regimes.
In Section~\ref{sec: clustering}, we introduce an asymptotic clustering property that is an analogue of the asymptotic pure atomicity studied  in~\cite{Var07},~\cite{BC16}
for discrete directed polymers, and prove that it holds for the endpoint distribution in the low temperature regime. 
In Section~\ref{sec: geometric localization}, we show that the endpoint distribution of directed polymer is asymptotically geometrically localized with positive density.

\textbf{Acknowledgements.} We are grateful to Erik Bates, Chiranjeeb Mukherjee, and Raghu Varadhan for stimulating discussions. YB thanks NSF for partial support via grant DMS-1811444. 

\subsection{The model of directed polymers in stationary environment} \label{sec: model}
We begin with a Markov chain  $\big((\omega_n)_{n \in \Nb}, \{P^x \}_{x \in \Rb^d}\big)$ on $\Rb^d$, defined on a measurable space
$(\Omega_p, \mathscr{F})$, where
\begin{enumerate}[label=$\bullet$, nosep]
	\item $\Omega_p =(\Rb^d) ^{\Nb} = \{\omega = (\omega_n)_{n \geq 0}  :\,  \omega_n \in \Rb^d\}$, 
	\item $\mathscr{F}$ is the cylindrical $\sigma$-algebra on $\Omega_p$, 
	\item For each $x\in \Rb^d$,  $P^x$ is the unique probability measure such that $(\omega_{n+1}-\omega_n)_{n \geq 0}$ are \textit{i.i.d.} and
		\beq \label{def: single step measure}
		P^x(\omega_0=x)=1, \quad P^x(\omega_{n+1}-\omega_n \in dy) =\ker(dy)
		\eeq
		for any nondegenerate Borel probability measure $\ker$ on $\Rb^d$.
\end{enumerate}
We stress that unlike the existing papers on directed polymers, we do not require $\ker$ to be a lattice distribution. In fact, for most of the paper, we do not impose any restrictions on $\ker$ at all. Thus $\ker$ may be an arbitrary mixture of  Lebesgue absolutely continuous, singular, and atomic distributions, and, if atomic, it does not have to be concentrated on any lattice (we only exclude the trivial case where~$\lambda$ is a Dirac mass).
We denote expectation with respect to $P^x$ by $E^x$. We also write $P$ and $E$ for $P^0$ and $E^0$.

The \textit{random environment} that we will consider is a real-valued, non-constant random field $\big(X(n, x)\big)_{n\in\Nb, x\in\Rb^d}$, 
defined on a probability space $(\Omega_e, \mathscr{G}, \Pb)$ such that 
\begin{itemize}
	\item $\big(X(n,\cdot\,)\big)_{n \in \Nb}$ are independent and identically distributed,
	\item $\big(X(1, x)\big)_{x \in \Rb^d}$ is stationary and  $\rdep$-dependent for some finite number $\rdep$, 
	i.e., for any subset $A, B \subset \Rb^d$ with $\deuc (A, B):=  \inf \{ |x-y| : x \in A,\ y \in B\} >\rdep$, 
	$(X(1, x))_{x \in A}$ and $(X(1, x))_{x \in B}$ are independent of each other. 
	\item $X(1, \cdot)$ has continuous trajectories, i.e., the mapping $x \mapsto X(1, x)$ is $\mathbf{P}$-a.s.\ continuous.
\end{itemize}
The continuity condition can actually be weakened, see Remark~\ref{rem:regular_conditional_prob}.
We will write $\Eb$ for expectation with respect to $\mathbf{P}$. $X(1, x)$ will be sometimes shortened to $X(x)$ for convenience.
We denote by $\beta \geq 0$ the \textit{inverse temperature} parameter and will assume
\beq\label{def: logarithmic moment generating function}
	c(\kappa) := \log \Eb \Big[\exp\big(\kappa X(0)\big)\Big] < \infty \quad \text{for}\,\, \kappa \in [-2\beta, 2\beta].
\eeq
For given $n \in \Nb, x \in \Rb^d$ , we define the \textit{point-to-point quenched polymer measure}, starting from~$x$, at time $n$ as
	\[ \mathscr{M}_n ^x (d\omega)=\frac{1}{Z_{n, x}} \exp\Big(\beta\sum\limits_{k=1}^{n} X(k, \omega_k)\Big) P^x (d\omega), \]
where 
	\[ Z_{n, x} = E^x \bigg[\exp\Big(\beta\sum\limits_{k=1}^{n}X(k, \omega_k)\Big)\bigg] \]
is called the \textit{point-to-line partition function}. Let $\mathscr{M}_n$ and $Z_n$ denote the polymer measure and the partition function corresponding to $P$, of length $n$.
Notice that $(\mathscr{M}_n)_{n \geq 0}, (Z_n)_{n \geq 0}$ are random processes adapted to the filtration $(\mathscr{G}_n)_{n \geq 0}$ given by
	\[ \mathscr{G}_n = \sigma(X(k, x): 1 \leq k \leq n, x \in \Rb^d). \]

\subsection{An outline of existing results in discrete setting}
Directed polymer models have been largely studied on the lattice $\Zb^d$.
In this section, we recall the well-known results in the discrete setting, which will be extended to the continuous model in this paper.
To stress the similarity with our model, we will use the same notation here as for our continuous setting. 
That is, in this section, we let $\mathscr{M}_n$ be the quenched polymer measure on paths of length $n$  defined on $(\Zb^d)^\Nb$ by
	\[ \mathscr{M}_n (d\omega) = \frac{1}{Z_n} \exp\Big(\beta \sum\limits_{k=1}^{n} X (k, \omega_k)\Big) P (d\omega), \]
where
\begin{enumerate}[label=$\bullet$, nosep]
	\item $P$ is the distribution of the $d$-dimensional simple random walk starting at $0$,
	\item the random environment $\big(X(k, x)\big)_{k \in \Nb, x \in \Zb^d}$ is given by a collection of non-constant, \text{i.i.d.} random variables
	defined on some probability space $(\Omega_e, \mathscr{G}, \mathbf{P})$ and
	\item $Z_n = E \Big[\exp\big(\beta \sum\limits_{k=1}^n  X (k, \omega_k)\big)\Big]$ is the partition function.
\end{enumerate}

Most of the mathematical results on directed polymers were obtained mainly by analyzing the asymptotic behavior of the partition function $Z_n$. 
One of the interesting quantities, called the \textit{quenched free energy}, is given by
	\[ F_n = \frac{1}{n} \log Z_n. \]
It turned out that the phase transition in directed polymer model is characterized by the discrepancy between the quenched free energy and the annealed free energy, which is
	\[ \frac{1}{n} \log \Eb[Z_n] = c (\beta). \]
Applying a superadditivity argument developed in \cite{CH02}, we see that the limit
\beq \label{eq: convergence of E F_n}
	\lim\limits_{n \rightarrow \infty} \Eb F_n = \sup\limits_{n \geq 1 } \Eb F_n := p(\beta)
\eeq
is well-defined. The following exponential concentration inequality enables us to make \eqref{eq: convergence of E F_n} stronger:
\begin{prethm}[Theorem 1.4 in \cite{LW09}, for $Q=1$] \label{thm: concentration inequality for W_n}
Let $\beta > 0$ be fixed such $\Eb e^{\beta |X(1, 0)|} < \infty$ 
Then, there is a constant $a>0$, depending only on $\beta$ and the law of $X$, such that
	\[ \mathbf{P} \left( \frac{1}{n} \Big| \log Z_n - \Eb\log Z_n \Big| > x \right) \leq
	\begin{cases}
		2e^{-nax^2}   & \text{if} \quad 0 \leq x \leq 1, \\
		2e^{-nax} & \text{if} \quad  x>1.
	\end{cases} \]
In particular, 
\beq\label{eq: convergence of quenched free energy}
	\lim\limits_{n \rightarrow \infty} F_n(\beta) = p(\beta) \quad \text{a.s.\!\ and\,\,} L_p \text{\,\,for all\,\,} p \in [1, \infty)
\eeq
\end{prethm}

We remark that Theorem~\ref{thm: concentration inequality for W_n} was proved for discrete setting but the proof can be easily adapted to our space-continuous setting.
Therefore, we will use \eqref{eq: convergence of quenched free energy} later without further proof.

The \textit{Lyapunov exponent} of the system is defined as
\beq\label{eq: Lyapunov exponent}
	\Lambda(\beta) := c(\beta) - p(\beta) \geq 0,
\eeq
where the inequality follows from Jensen's inequality. Before describing the phase transition of directed polymers, we give a statement for the existence of critical temperature.
\begin{prethm}[Theorem 3.2 in \cite{CY06}, Proposition 2.4 in \cite{Bat18}] \label{thm: critical beta}
$\Lambda(\beta)$ is non-decreasing in $\beta$. In particular, there is a critical inverse temperature $\beta_c = \beta_c(d) \in [0, \infty]$ such that
\begin{align*}
	0 \leq &\beta \leq \beta_c \quad \Rightarrow \quad \Lambda(\beta)=0,
	\\& \beta >\beta_c \quad \Rightarrow \quad \Lambda(\beta)>0.
\end{align*}
\end{prethm}

Theorem~\ref{thm: critical beta} was first proved in \cite{CY06} when the reference measure is the simple random walk and $c(\kappa)$ exists for all $\kappa \in \Rb$.
\cite{Bat18} enhanced this by extending to reference measures given by arbitrary random walks on $\Zb^d$ and weakening the moment condition of random environment. Extending this result to general random walks on $\Rb^d$ is straightforward.

We now collect three statements which describe how the Lyapunov exponent identifies the phase transition of directed polymers.
We denote by
	\[ \rho_i (\cdot) = \mathscr{M}_i (\omega_i \in \cdot) \]
the endpoint distribution of directed polymer of length $i$.

\begin{prethm}[Corollary 2.2 and Theorem 2.3 (a) in \cite{CSY03}] \label{thm: localization of partial mass}
	\[ \Lambda(\beta) > 0 \quad  \Leftrightarrow 
	\quad \exists\,  c>0 \ \, \text{\rm s.t.} \,\,
	\liminf\limits_{n \rightarrow \infty} \frac{1}{n} \sum\limits_{i=0}^{n-1} \max\limits_{x \in \Zb^d} \rho_i \big(\{x\}\big) \geq c 
	\quad \mathbf{P} \text{-} a.s. \]
\end{prethm}

Theorem~\ref{thm: localization of partial mass} tells that the endpoint distribution can localize partial mass in the low temperature regime. 
Vargas proposed in \cite{Var07} the notion of \textit{``asymptotic pure atomicity''}, which describes the localization of the entire mass of the endpoint distribution.
For any $i\geq 0, \epsilon >0$, let
	\[ \Ac_i^{\epsilon} = \{x \in \Zb^d: \rho_i \big(\{x\}\big) >\epsilon\}. \]
Then, $(\rho_i)_{i \geq 0}$ is called \textit{asymptotically purely atomic} if for every sequence $(\epsilon_i)_{i \geq 0}$ tending to $0$, we have
	\[ \lim\limits_{n \rightarrow \infty} \frac{1}{n} \sum\limits_{i=0}^{n-1} \rho_i(\Ac_i^{\epsilon_i}) =1 \quad \mathbf{P} \text{-} a.s. \]
Convergence in probability was used in \cite{Var07} and the author proved that if $c(\beta) = \infty$, then $(\rho_i)_{i \geq 0}$ is asymptotically purely atomic. Bates and Chatterjee replaced it with almost sure convergence and proved the following:

\begin{prethm} [Theorem 6.3 in \cite{BC16}, Theorem 5.3 in \cite{Bat18} ] \label{thm: asymptotic pure atomicty on Z^d}
	\[ \Lambda(\beta) > 0 \quad \Leftrightarrow \quad (\rho_i)_{i \geq 0} \text{ is asymptotically purely atomic.} \]
\end{prethm}

Theorem~\ref{thm: geometric localization on Z^d} illustrates how the favorable sites, which localize mass in the endpoint distribution of directed polymers, cluster together.
For $\delta >0$ and $K >0$, let $\mathcal{G}_{\delta, K}$ be the collection of probability measures on $\Zb^d$ that assign mass greater than $1-\delta$ to some subset of $\Zb^d$ having diameter at most~$K$. 
(We use the $l_1$ distance here.) We say that $(\rho_i)_{i \geq 0}$ is \textit{geometrically localized with positive density} 
if for every $\delta>0$, there exist $K > 0$ and $\theta >0$ such that
	\[ \liminf\limits_{n \rightarrow \infty} \frac{1}{n} \sum\limits_{i=0}^{n-1} \one_{\{\rho_i \in \mathcal{G}_{\delta, K}\}} \geq \theta \quad \mathbf{P} \text{-}a.s. \]
	
\begin{prethm}[Theorem 7.3 (a), (c) in \cite{BC16}, Theorem 5.4 in \cite{Bat18}] \label{thm: geometric localization on Z^d}
	\[ \Lambda(\beta) > 0 \quad \Leftrightarrow \quad (\rho_i)_{i \geq 0} \text{ is geometrically localized with positive density.} \]
\end{prethm}	

\subsection{Main results of this paper}\label{sec: main results}
The first main result of this paper is the development of a new metrization of the translation-invariant compactification of the space of probability measures. 
The structure of the metric and relevant background are provided in Section~\ref{sec: MV}.
As an application of the theory developed in Section~\ref{sec: MV}, we prove analogues of Theorems~\ref{thm: asymptotic pure atomicty on Z^d} and~\ref{thm: geometric localization on Z^d} for our model of directed polymers in the continuous space.
Before stating our results, we denote the \textit{quenched endpoint distribution} for the polymer of length $n$ by 
	\[ \rho_n (dx)=\mathscr{M}_n (\omega_n \in dx). \]

We extend the notion of asymptotic pure atomicity applicable in discrete case to the continuous case in three ways. We introduce three related notions of clustering: we define asymptotic clustering at level $r>0$ in Definition~\ref{def:cluster-at-level} (this notion is also considered in~\cite{BM18}), the notion of asymptotic local clustering in Definition~\ref{def:asymptotic-local-clustering}, and the notion of asymptotic clustering of densities in Definition~\ref{def:clustering-densities}.  For a sequence of absolutely continuous measures, the asymptotic local
clustering is equivalent to the asymptotic clustering of densities, see Remark~\ref{rem:clustering-densities}. 

The following results concerning the notions of  asymptotic clustering at positive levels and  asymptotic local clustering
(analogues of Theorem~\ref{thm: asymptotic pure atomicty on Z^d} on asymptotic pure atomicity) are proved in Section~\ref{sec: clustering}, see Theorems~\ref{thm: asymptotic clustering} and~\ref{thm: asymptotic local clustering}: 

\begin{thm}\label{thm: asymptotic clursterization, intro}
For $r>0$, $\eps>0$, and $i \geq 0$, let us define
	\[ \Ac_i ^{\epsilon} (r) := \{ x \in \Rb^d : \rho_i (B_r (x)) > \epsilon V_d r^d\}, \quad
	\Ac_i ^{\epsilon} = \Big\{ x \in \Rb^d : \liminf\limits_{r \downarrow 0} \frac{\rho_i (B_r (x))}{V_d r^d} > \epsilon\Big\},\]
where $V_d$ is the volume of the unit ball in $\Rb^d$.
\begin{enumerate}[label=\rm(\alph*), nosep]
	\item If $\beta > \beta_c$, then for every $r>0$ and every sequence $(\epsilon_i)_{i \geq 0}$ tending to $0$,
\beq\label{eq: asymptotic clustering, intro}
	\lim_{n\rightarrow \infty}  \frac{1}{n}  \sum_{i=0}^{n-1} \rho_i \big(\Ac_i ^{\epsilon_i} (r)\big) = 1 \quad \mathbf{P}\text{-a.s.},
\eeq
and
\beq \label{eq: asymptotic local clustering, intro}
	\lim_{n\rightarrow \infty}  \frac{1}{n}  \sum_{i=0}^{n-1} \rho_i (\Ac_i ^{\epsilon_i}) = 1 \quad \mathbf{P}\text{-a.s.}
\eeq

	\item If $\beta \leq \beta_c$, then for every $r>0$, there is a sequence $(\epsilon_i)_{i \geq 0}$ tending to $0$, such that
\beq\label{eq: endpoint is not asymptotically clusterized in high temperature, intro}
	\lim\limits_{n \rightarrow \infty} \sum\limits_{i=0}^{n-1} \rho_i (\Ac_i ^{\epsilon_i} (r) ) = 0 \qquad \mathbf{P} \text{-}a.s.
\eeq
\end{enumerate}
\end{thm}

The following localization result (an analogue of Theorem~\ref{thm: geometric localization on Z^d}) is proved in Section \ref{sec: geometric localization}, see Theorem~\ref{thm: localization}:

\begin{thm}\label{thm: geometric localization, intro}
For $\delta>0$ and $K>0$, let us define a set
	\[ \mathcal{G}_{\delta, K} = \{\alpha \in \Mc_1 : \max\limits_{x \in \Rb^d} \alpha(B_K (x)) > 1 - \delta \},\]
where $\Mc_1$ is the collection of probability measures on $\Rb^d$.
\begin{enumerate}[label=\rm(\alph*), nosep]
	\item If $\beta > \beta_c$, then for all $\delta >0$, there exist $K < \infty$ and $\theta > 0$ such that
	\[ \liminf\limits_{n \rightarrow \infty} \frac{1}{n} \sum\limits_{i=0}^{n-1} \one_{\{\rho_i \in \mathcal{G}_{\delta, K}\}} \geq \theta \quad \mathbf{P} \text{-}a.s.,\]

	\item If $\beta \leq \beta_c$, then for all $\delta \in (0, 1)$ and $K>0$,
	\[ \lim\limits_{n \rightarrow \infty} \frac{1}{n}\sum\limits_{i=0}^{n-1} \one_{\{\rho_i \in \mathcal{G}_{\delta, K}\}} = 0 \quad \mathbf{P} \text{-} a.s. \]
\end{enumerate}
\end{thm}

\section{Compactification of a space of probability measures} \label{sec: MV}
In \cite{BC16}, the authors pointed out that the usual topologies of weak/vague convergence of probability measures are inadequate to capture the localization phenomenon of directe polymers. To tackle the issue, they used an analogue of the compact metric space $(\Xct, \mathbf{D})$ constructed in the work of Mukherjee and Varadhan \cite{MV16}.

The idea behind the MV topology is that two measures are considered close to each other if one can find several well-separated regions of high concentration for each of them such that the restrictions of the measures to these regions are close to being spatial translations of each other. To encode this, it is natural to work with an extension
of the space of measures on $\Rb^d$ to the space of measures on $\Nb\times \Rb^d$ where multiple layers (copies of $\Rb^d$)  correspond to multiple domains of concentration. In this approach, it is natural not to distinguish between two measures in one layer if they are obtained by a translation of each other, and the order of the layers is not important either. 
Now all measures on $\Rb^d$ that can be approximated as a sum of translations of the measures in the layers without much overlap can be viewed as being close to each other.

We will discuss two formalizations of these ideas in this section. 
While the Mukherjee--Varadhan (MV) topology was originally defined through test functions, Bates and Chatterjee introduced a different form of metric on the space of sub-probablity distributions on $\Nb\times \Zb^d$ in \cite{BC16} and showed that their metric space is equivalent to the discrete version of the MV topology.
We recall the original MV definition first and then construct a metrization of the MV topology that is similar to the metric introduced in \cite{BC16}.

\medskip

Before we begin, we give a brief guide to the notations that we use throughout the paper. 
\begin{itemize}
\item $x, y$ are used for elements of $\Rb^d$ and $u, v$ for elements of $\Nb \times \Rb^d$.
\item $\alpha, \gamma, \lambda$ are used for subprobabiltiy measures on $\Rb^d$.
\item $\mu, \nu, \eta, \tau$ denote elements of $\Xc$ and $\Xct$, the spaces defined in Section~\ref{sec:original-MV-topology}. 
\item $\xi, \zeta$ denote elements of $\Pc(\Xct)$, the space of probability measures on $\Xct$ introduced in Section~\ref{sec: update map}.
\item Functionals on $\Xct$ are usually denoted by capital letters, such as $T, R$ and $I_r$, while those on $\Pc(\Xct)$ are denoted in calligraphic fonts, e.g., $\Tc$ and $\Rc$.
\end{itemize}

\subsection{Mukherjee--Varadhan topology}\label{sec:original-MV-topology}

For any $a \geq 0$, we denote by $\Mc_a = \Mc_a (\Rb^d)$ $(\Mc_{\leq a})$ 
the space of measures on $\Rb^d$ with mass $a$ (less than or equal to $a$) 
and by $\wt{\Mc}_a = \Mc_a / \sim$ 
the quotient space of $\Mc_a$ under spatial shifts on~$\Rb^d$. For any $\alpha \in \Mc_a$, its \textit{orbit} is defined by 
	\[ \tilde{\alpha} = \{\alpha * \delta_x : x\in \Rb^d\} \in \wt{\Mc}_a, \]
where $\alpha_1 * \alpha_2$ denotes the convolution of $\alpha_1$ and $\alpha_2$ in $\Mc_{\leq a}$, i.e., for any measurable set $A$ in $\Rb^d$,
	\[ \alpha_1 * \alpha_2 (A) = \int_{(\Rb^d)^2} \one_A (x+y) \alpha_1(dx) \alpha_2 (dy). \]
In particular, if $\alpha_2 (dx) = f(x) dx$, then $\alpha_1 * \alpha_2(dx) = \int f(x-y)\alpha_1(dy)dx$.
We denote the zero measure on $\Rb^d$ or $\Nb \times \Rb^d$ by $\mathbf{0}$.

Let us recall the notions of the weak topology and the vague topology on $\Mc_{a}$ and $\Mc_{\leq a}$ which will be used in this paper.
We say that a sequence $(\alpha_n)_{n \in \Nb}$ in $\Mc_{a}(\text{or\,\,}\Mc_{\leq a})$ converges to~$\alpha$ in the \textit{weak topology} and write $\alpha_n \Rightarrow \alpha$ if 

\beq\label{eq: weak convergence}
	\lim\limits_{n \rightarrow \infty} \int_{\Rb^d} f(x) \alpha_n (dx) = \int_{\Rb^d} f(x) \alpha (dx),
\eeq
for all bounded continuous functions $f$ on $\Rb^d$.
We say a sequence $(\alpha_n)_{n \in \Nb}$ in $\Mc_{\leq a}$ converges to $\alpha$
in the \textit{vague topology} and write $\alpha_n \hookrightarrow \alpha$   
if \eqref{eq: weak convergence} holds for all continuous functions with compact support. 
Note that the weak convergence preserves the total mass of measures, while the vague convergence may fail to do so. 

Another distinction between two topologies is that $\Mc_{\leq a}$ is compact in the vague topology, but not in the weak topology.

Throughout the paper, we will work with multisets (sets with multiplicities) $[\tilde\alpha_i]_{i\in I}$ consisting of elements of $\wt{\Mc}_{\leq 1}$.
We define
	\[ \Xct = \Big\{ \mu = [\widetilde{\alpha_i}]_{i \in I} :\ I\subset \Nb,\ 
	 \alpha_i \in \Mc_{\leq 1}  \backslash \{\zero\},\, \sum\limits_{i\in I} \alpha_i (\Rb^d) \leq 1\Big\}\]
to be the space of all empty, finite or countable collections of orbits of subprobability measures on~$\Rb^d$.
For convenience, we slightly depart from the original definition in \cite{MV16} and
do not allow~$\alpha_i$ to be a zero measure. 

Let us introduce an interpretation of $\Xct$ as a quotient space of $\Xc = \Mc_{\leq 1} (\Nb \times \Rb^d)$.
For $\mu \in \Xc$, we can write $\mu(dk,dx)=\sum_{i\in \Nb}\alpha_i(dx) \delta_i(dk)$ and identify $\mu$ with the (ordered) sequence 
$\mu = (\alpha_i)_{i \in \Nb}$,
of subprobability measures on $\Rb^d$,  with
$\sum \|\alpha_i \| \leq 1$. Some of $\alpha_i$ may be equal to $\mathbf{0}$.
We define the $\Nb$-support of $\mu \in \Xc$~by
\beq\label{eq: N-support}
	S_\mu = \{ i \in \Nb : \norm{\alpha_i} >0 \}.
\eeq

The following definition explains when two measures from $\Xc$ are representatives of the same element of  $\Xct$:
\begin{definition}\label{def: equivalence relation}
Let $\mu = (\alpha_i), \nu = (\gamma_i) \in \Xc$.
We write $\mu \sim \nu$ if $|S_\mu| = |S_\nu|$ and there is a bijection $\sigma:  S_\mu \rightarrow S_\nu$ such that $\wt{\alpha}_{i} = \wt{\gamma}_{\sigma(i)}$ for all $i \in S_\mu$.
\end{definition}

Thus  $\mu = [\widetilde{\alpha_i}]_{i \in I}  \in\Xct$ can be represented or viewed as an element of $\Xc$ (a measure on $\Nb\times\Rb^d$), a sequence $(\alpha_i)_{i\in\Nb}$ of measures in $\Rb^d$
by taking $\alpha_i = \zero$ for all $i \notin I$.
 We will often not make a distinction between $\mu\in\Xct$ and its representative. We will write $\zero$ (instead of $\emptyset$) for the empty multiset of $\Xct$
since its sole representative is $\zero$.

\smallskip

In order to define the metric and convergence in $\Xct$, we need to specify test functions.
For an integer $k\geq 2$, let $\mathcal{F}_k$ be the space of continuous functions $f : (\Rb^d)^k \rightarrow \Rb$
which are \textit{translation invariant} and \textit{vanishing at infinity}, i.e.
\beq\label{eq: translation of test function}
	f(x_1+y, \cdots, x_k+ y) = f(x_1, \cdots, x_k) \quad \forall \, x_1, \cdots, x_k, y \in \Rb^d,
\eeq
\beq\nonumber
	\lim\limits_{\max\limits_{i\neq j} |x_i -x_j| \rightarrow \infty} f(x_1, \cdots, x_k) = 0.
\eeq

Note that $\mathcal{F}_k$, equipped with the uniform norm, is separable. Therefore, if we denote $\mathcal{F} = \bigcup\limits_{k \geq 2} \mathcal{F}_k$, we can choose a countable dense subset 
$\{f_r(x_1, \cdots, x_{k_r})\}_{r\in \Nb}$ of $\mathcal{F}$.
We also check that for any $f \in \mathcal{F}_k$ and $\mu=[\wt{\alpha}_i] \in \Xct$ , the functional
\beq \label{eq: functional for metric D}
	\Lambda(f, \mu) := \sum\limits_{i \in I} \int f(x_1, \cdots, x_k) \prod\limits_{j=1}^{k} \alpha_i(dx_j)
\eeq
is well-defined due to \eqref{eq: translation of test function}.
For any $\mu, \nu \in \mathcal{\wt{X}}$, we now define
\beq \label{eq: original metrization D}
	\mathbf{D}(\mu, \nu) = \sum\limits_{r=1}^{\infty} \frac{1}{2^r \big(1+\norm{f_r}\big)} |\Lambda(f_r, \mu) - \Lambda(f_r, \nu)|.
\eeq
Here $\|f\|$ denotes the uniform norm. We state a theorem proved in \cite{MV16}.
\begin{thm}[Theorem 3.1 and Theorem 3.2 in \cite{MV16}]
The metric space $(\Xct, \mathbf{D})$ is a compactification of~$\wt{\Mc}_1$.
\end{thm}

\subsection{Reinterpretation of the MV topology}
Due to the analogy with \cite{BC16}, the compact metric space $(\Xct, \mathbf{D})$ is expected to be suitable for studying localization for directed polymers on $\Nb \times \Rb^d$. However, one might have difficulties in extracting some information on two elements  $\mu=[\wt{\alpha}_i], \nu = [\wt{\gamma}_i] \in \Xct$ close to each other. More precisely, one would expect that if $\mathbf{D} ( \mu, \nu)$ is very small,  
one can match large parts of measures $\alpha_i$ and $\gamma_j$ by applying appropriate
translations to subsets of $\Rb^d$. Motivated by the approach taken in  \cite{BC16}, 
in the present paper, we  attempt at  expressing this idea more explicitly in the definition of an appropriate metric.
Similarly to having two definitions of the Wasserstein distance in terms of Lipschitz test functions and in terms of couplings, it would be natural and helpful to introduce an equivalent metric on $\Xct$ that is based on coupling. Adopting the ideas from \cite{BC16}, we  construct such an equivalent metric which allows us to obtain explicit estimates needed to show continuity of some functionals defined on $\Xct$. 

Before constructing the metric rigorously, we need to introduce some notations. 
We define a distance between two elements  $u=(i, x)$ and $v=(j, y)$ of $\Nb \times \Rb^d$ by
\beq \label{eq: distance on N times R^d}
	\euc{u- v} =\one_{\{i=j\}} \cdot \euc{x-y} + \one_{\{i \neq j\}} \cdot \infty.
\eeq
This definition is natural in the sense that we would like to record two concentrated regions getting away from each other on different copies of $\Rb^d$.
For $r>0$, we denote by $B_r(u)$ the open ball centered at $u$ with radius~$r$ in $\Nb \times \Rb^d$ and similarly by $B_r (x)$ in $\Rb^d$.
Notice that $B_r(u) = \{i\} \times B_r(x)$ by \eqref{eq: distance on N times R^d}.

The right-hand side in \eqref{eq: functional for metric D} can be expressed in terms of functions defined on $\Nb\times \Rb^d$ instead of $\Rb^d$. 
More precisely, for an integer $k \geq 2$, let $\mathcal{G}_k$ be the space of continuous
functions $g : (\Nb \times \Rb^d)^k \rightarrow \Rb$ that are translation-invariant and vanishing at infinity, i.e.
	\[ g(u_1+v, \cdots, u_k+ v) = g(u_1, \cdots, u_k) \quad \forall u_1, \cdots, u_k, u_1+v, \cdots, u_k+v \in \Nb \times \Rb^d, \]
\beq \label{eq: test functions vanish at infinity}
	\lim\limits_{\max\limits_{i\neq j} \euc{u_i -u_j} \rightarrow \infty} g(u_1, \cdots, u_k) = 0.
\eeq
For any $g \in \mathcal{G}_k$, $g \neq 0$ only if all $u_1, u_2, \cdots, u_k$ belong to the same copy of $\Rb^d$ due to \eqref{eq: test functions vanish at infinity}. 
Therefore, there is  a unique $f \in \mathcal{F}_k$ such that
	\[ g(u_1, \cdots, u_k) = 
	\begin{cases}
		f (x_1,  \cdots, x_k) & \text{if} \,\, u_j =(i, x_j) \,\,\text{for some}\,\, i \in \Nb,\\
		\qquad \,\,\, 0 & \text{otherwise.}\\
	\end{cases} \]
In other words, there is a natural bijection
\beq \label{eq: bijection between F_k and G_k}
	\sigma_k : \mathcal{F}_k \rightarrow \mathcal{G}_k.
\eeq
Then, considering $\mu$ as an element of $\Xc$, we have 
	\[ \Lambda(f, \mu) = \int (\sigma_k f) (u_1, u_2, \cdots, u_k) \prod_{j=1}^{k}\mu (du_j). \]
Another remark is that any continuous function $f: (\Rb^d)^{k-1}  \rightarrow \Rb$ vanishing at infinity can be identified with an element of $\mathcal{F}_k$ by mapping it to
\beq \label{eq: a remark for test function}
	\wt{f}(x_1, x_2, \cdots, x_{k}) = f(x_2-x_1, \cdots, x_{k} - x_1).
\eeq

For any $\alpha \in \Mc_{\leq 1} (\text{\,or\,\,} \Xc)$ and non-negative function $f$ which is integrable with respect to $\alpha$, 
we write $\bar{\alpha} = f \alpha$ if $\bar{\alpha}$ is defined as $\bar{\alpha}(A)=\int_A f \alpha(dx)$ 
for each measurable set $A$. Moreover, we say $\bar{\alpha}$ is a submeasure of $\alpha$ (denoted by $\bar{\alpha} \leq \alpha$) if $0\leq f \leq 1$.
For any signed measure $\mu$ on $\Rb^d$ or $\Nb \times \Rb^d$, denote by $\norm{\mu}$ the total variation of $\mu$.

\subsection{The Wasserstein distance} 
In this section, we recall the basics on the Wasserstein distance. 
Similar notions were first introduced to solve the Monge--Kantorovich transportation probem and it turned out that such distances can be used extensively in the variety of fields (see, e.g., \cite{Vil09}).

To any metric $d_{\Euc}$ on $\Rb^d$ generating the Euclidean topology, we can associate a transport distance on measures as follows.
For $\alpha, \gamma \in \Mc_{a}\, (a>0)$, let $\Pi(\alpha, \gamma)$ be the collection of Borel probability measures on $(\Rb^d)^2$ such that
the marginal distribution of the first argument is $\alpha/a$ and of the second argument is $\gamma/a$. 
Then, the Wasserstein distance between $\alpha$ and $\gamma$ is defined by

\beq \label{eq: Wasserstein metric} 
	W(\alpha, \gamma) = a  \inf_{\pi \in \Pi(\alpha, \gamma)} \int_{\Rb^2} d_{\Euc}(x, y) \pi(dx, dy).
\eeq
It is known that the infimum on $\Pi$ is achieved. In this paper, we choose to work with a bounded metric
	\[ d_{\Euc}(x, y) = \euc{x-y} \wedge 1, \]
so that $W$ metrizes the topology of weak convergence of $\Mc_{a}$.

For $\wt{\alpha}, \wt{\gamma} \in \wt{\Mc}_{a}$, we define
	\[ \wt{W}(\wt{\alpha}, \wt{\gamma}) = \inf\limits_{x \in \Rb^d} W(\alpha, \gamma *\delta_x).\]
Since the choice of representatives does not affect the value of $\wt{W}$, it is well-defined. One can check that $\wt{W}$ is a metric on $\wt{\Mc}_a$ and metrizes the weak topology of $\wt{\Mc}_a$. The latter is defined in the following sense: 
	\[ \wt{\alpha}_n \Rightarrow \wt{\alpha} \text{\,\, in \,\,} \wt{\Mc}_a \quad\Longleftrightarrow 
	\quad\exists (x_n)_{n \in \Nb} \text{\,\,in }\Rb^d \,\,\mathrm{such\,\,that\,\,} \alpha_n *\delta_{x_n} \Rightarrow \alpha \text{\,\,in\,\,} \Mc_a. \]
\par A result of \cite{PR14} allows us to apply the Wasserstein distance to $\alpha, \gamma \in \Mc_{\leq a}$ with different masses. 
More precisely, the \textit{generalized Wasserstein distance} $\hW$ can be defined by
\beq\label{def: generalized Wasserstein distance}
	\hW (\alpha, \gamma) 
	= \inf\limits_{\substack{\bar{\alpha} \leq \alpha, \bar{\gamma} \leq \gamma \\ \|\bar{\alpha}\|=\|\bar{\gamma}\|}} \Big(W(\bar{\alpha},
	\bar{\gamma}) + \norm{\alpha - \bar{\alpha}} + \norm{\gamma- \bar{\gamma}}\Big)
\eeq
and it is proved in \cite{PR14} that the infimum on the right hand side is achieved.

The result known as the \textit{Kantorovich duality} states that for any $\alpha, \gamma \in \Mc_{\leq 1}$ with the same mass,
\beq\label{eq: Kantorovich duality}
	W(\alpha, \gamma) = \sup\limits_{f} \Big(\int_{\Rb^d} f(x) \alpha(dx) - \int_{\Rb^d} f(y) \gamma(dy)\Big),
\eeq
where the supremum is taken over all 1-Lipschitz continuous functions $f:(\Rb^d, d_{\Euc}) \rightarrow (\Rb^d, |\cdot|)$, i.e.,
	\[ |f(x)-f(y)| \leq \euc{x-y}, \quad x, y \in \Rb^d, \qquad \sup f - \inf f \leq 1. \]
It follows from~\eqref{eq: Kantorovich duality} that for for any measures $\mu=\mu_1+\mu_2$ and $\nu=\nu_1+\nu_2$ 
with $\norm{\mu}=\norm{\nu}$, $\norm{\mu_1}=\norm{\nu_1}$ and $\norm{\mu_2}=\norm{\nu_2}$, one has
\beq\label{ineq: Wsplit}
	W(\mu,\nu)\le W(\mu_1,\nu_1)+W(\mu_2,\nu_2).
\eeq

\subsection{Construction of a metric on $\Xct$}
We are now ready to define a metric on $\Xct$. 
From now on, for any $\mu = [\wt{\alpha}_i]_{i \in I} \in \Xct$, we will abuse the notation $\mu$ and use it for both the element of $\Xct$ 
and representatives chosen from $\Xc$. 
When $\mu$ is used in integration, we mean that an explicit representative, such as $(\alpha_i)_{i \in \Nb}$, is chosen where $\alpha_i = \mathbf{0}$ for all $i \notin I$.

Let $\mu=[\wt{\alpha}_i], \nu = [\wt{\gamma}_i] \in \mathcal{\wt{X}}$ be given.
We first introduce a family of functionals estimating the mass of the heaviest region for a measure in $\Xc$.
For $r \geq 0$, we define a function $I_r$ on $\Xc$ by
\beq \label{eq: concentration function}
	I_r(\mu) = \sup\limits_{i \in \Nb,  x \in \Rb^d} \int_{\Rb^d} f_r (x-y) \alpha_i (dy)
	=\sup_{ u \in \Nb \times \Rb^d} \int_{\Nb \times \Rb^d} g_r(u-v) \mu(dv),
\eeq
where 
	\[ f_{r} (x)= 
	\begin{cases} 
		\qquad 1, &  \euc{x} \le  r  \\
   		\qquad 0, & \euc{x} > r+1, \\
   		r+1- \euc{x}, &  \euc{x}\in(r, r+1], 
	\end{cases} \]
and $\wt{g}_r = \sigma_2 \wt{f}_r$ (See~\eqref{eq: bijection between F_k and G_k} and \eqref{eq: a remark for test function}).
Note that $f_r$ is 1-Lipschitz continuous with respect to~$d_{\Euc}$.
We collect some useful properties of $I_r$:
\begin{enumerate}[label=$\bullet$, nosep]
	\item $I_r(\mu)$ is comparable with the mass of the heaviest ball of radius $r$ under $\mu$, i.e.,
		\beq\label{eq: relation between I_r and the heaviest ball}
			\sup\limits_{u \in \Nb \times \Rb^d} \mu \big(B_{r}(u)\big) \leq I_r (\mu) \leq \sup\limits_{u \in \Nb \times \Rb^d}  \mu \big(B_{r+1}(u)\big).
		\eeq
	\item $I_r$ is sub-additive, i.e., $I_r(\mu+\nu) \leq I_r(\mu)+ I_r(\nu)$.
	\item $I_r$ is monotone, i.e., if $\mu \leq \nu$, then $I_r(\mu) \leq I_r(\nu)$.
	\item Since $\Mc_{\leq 1}$ is naturally embedded in $\Xc$, we can define $I_r (\alpha)$ for $\alpha \in \Mc_{\leq 1}$ in the same way. 
For any $\alpha, \gamma \in \Mc_{\leq 1}$ with the same mass, \eqref{eq: Kantorovich duality} implies
\beq\label{eq: relation between concentration function and Wasserstein metric}
	|I_r(\alpha)- I_r(\gamma)| \leq \sup\limits_{x \in \Rb^d} \Big|\int f_{r} (x, y) d\alpha (y) - \int f_{r} (x, y) d\gamma (y)\Big| \leq W(\alpha, \gamma).
\eeq
\end{enumerate}

One can check that the choice of the representative of an element in $\Xct$ does not change the value of $I_r (\mu)$ 
so $I_r$ is also well-defined on $\Xct$.

\begin{definition}\label{def: matching}
\topsep=0pt
For any $\mu=(\alpha_i), \nu=(\gamma_i) \in \Xc$,
let $P_{\mu, \nu}$ be the collection of sets $ \{(\mu_k, \nu_k)\}_{k=1}^{n}$  of pairs of submeasures of $\mu, \nu$ such that
\begin{enumerate}[label=\rm(\arabic*), nosep]
	\item For each $k$, \,$\norm{\mu_k}= \norm{\nu_k}>0$.
	\item For each $k$, $|S_{\mu_k}|= |S_{\nu_k}| = 1$, i.e. each $\mu_k$ and $\nu_k$ has exactly one layer of $\Rb^d$ with positive mass. 
	(See \eqref{eq: N-support} for the definition of $S_{\mu}$).
	\item Collections $\big\{\supp(\mu_k)\big\}_{k=1}^{n}$ and $\big\{\supp(\nu_k)\big\}_{k=1}^{n}$ are each composed of mutually disjoint sets.
\end{enumerate}
Then, an element in $P_{\mu, \nu}$ is called a \textit{$(\mu, \nu)$-matching} (or simply a matching when there is no confusion).
We have the empty matching $\emptyset$ included in any~$P_{\mu, \nu}$. 

For $\phi  =  \{(\mu_k, \nu_k)\}_{k=1}^{n} \in P_{\mu, \nu}$, we define
	\[ \sep(\phi) := \inf_ {1\leq k_1 < k_2 \leq n }\Big\{ \deuc \big(\supp(\mu_{k_1}), \supp(\mu_{k_2})\big) \}
	\wedge  \inf_ {1\leq k_1 < k_2 \leq n} \Big\{ \deuc \big(\supp(\nu_{k_1}), \supp(\nu_{k_2})\big) \Big\},\]
where $\deuc (A, B) = \inf \{ |u-v| : u \in A,\ v \in B\}$ for $A, B \subset \Nb \times \Rb^d$. We set $\sep(\emptyset) = \infty$.
\end{definition}

\begin{rmk} \rm
From condition (2) in Definition~\ref{def: matching}, we can identify $\mu_k$ and $\nu_k$ as subprobability measures on $\Rb^d$ if needed.
The quantitity $\sep (\phi)$ is the degree of separation among the supports of submeasures in the matching. 
We see that $\deuc \big(\supp(\mu_{k_1}), \supp(\mu_{k_2})\big) < \infty$ only if 
$\mu_{k_1}$ and $\mu_{k_2}$ belong to the same layer of $\mu$.
If the supports of distinct $\mu_k$ belong to different layers of $\mu$
(i.e., $\mu_{k_1} \leq \alpha_{j_1}$ and $\mu_{k_2} \leq \alpha_{j_2}$ imply $j_1 \neq j_2$ for any $k_1 \neq k_2$) and the same holds for $\nu$, then we have $\sep (\phi)= \infty$ due to \eqref{eq: distance on N times R^d}.
\end{rmk}

\begin{definition}\label{def: triples}
Let $\mu, \nu \in \Xc$. A triple  $(r, \phi, \vec{x})$ is called a $(\mu, \nu)$-\textit{triple} if
$r \geq 0$, $\phi =  \{(\mu_k, \nu_k)\}_{k=1}^{n}  \in P_{\mu, \nu}$, $\sep(\phi) > 2r$, and $\vec{x}= (x_1, \cdots, x_n) \in (\Rb^d)^n$.
For any $(\mu, \nu)$-triple $(r, \phi, \vec{x})$, we define
\beq\label{eq: psudo metric in mvspace}
	d_{r, \phi, \vec{x}} (\mu, \nu) = \sum\limits_{k=1}^{n} W (\mu_k, \nu_k *\delta_{x_k}) 
	+ I_r\Big(\mu - \sum\limits_{k=1}^{n} \mu_k\Big)+  I_r\Big(\nu- \sum\limits_{k=1}^{n} \nu_k\Big)+ 2^{-r}.
\eeq
\end{definition}

\begin{rmk}\label{rmk:interpretation of mu_k and nu_k} \rm
We see that, for the empty matching,
	$d_{r, \emptyset, \vec{x}} (\mu, \nu) = I_r(\mu)+  I_r(\nu)+ 2^{-r}$
does not depend on $\vec{x}$.
For any non-empty matching, $\mu_k$ and $\nu_k$ are interpreted in two different ways in the right-hand side of \eqref{eq: psudo metric in mvspace}.

While they are treated as elements of $\Mc_{\leq 1}$ in the Wasserstein metric term, they are viewed as submeasures of $\mu$ and $\nu$, respectively, in $\Xc = \Mc_{\leq 1} (\Nb \times \Rb^d)$ in the $I_r$ terms. 

Let us see how this works in a specific example. For simplicity, we assume $d=1$.
Let $\mu = (\alpha_1, \alpha_2, \alpha_3, \mathbf{0}, \mathbf{0}, \cdots)$ and $\nu = (\gamma_1, \gamma_2, \gamma_3, \gamma_4, \mathbf{0}, \mathbf{0}, \cdots)$ such that
\[  \alpha_1 =  \frac{1}{8} \delta_0 + \frac{1}{8} U_{(8, 9)}, \quad \alpha_2 = \frac{1}{4} \delta_{1} + \frac{1}{6} N_{(1, 4)}, \quad \alpha_3 = \frac{1}{4}N_{(0, 1)},
\]
\[\gamma_1 = \frac{1}{4}\delta_{-1} + \frac{1}{10}\delta_{10}, \quad \gamma_2 = \frac{1}{6} N_{(2, 1)}, \quad \gamma_3 = \frac{1}{6} U_{(3, 6)}, 
\quad \gamma_4 = \frac{1}{8} N_{(3, 3)},
\]
where $U_{(a, b)}$ is the uniform probability measure on the interval $(a, b)$ and $N_{(a, b)}$ is the Gaussian measure with mean $a$ and variance $b$.
We can choose a $(\mu, \nu)$-matching $\phi$ as follows: 
\[\mu_1 = \frac{1}{10} \delta_0 \text{ in } \alpha_1,\,\, \mu_2 = \frac{1}{8} U_{(8, 9)} \text{ in } \alpha_1,\,\,
\mu_3 = \frac{1}{4} \delta_{1}  \text{ in } \alpha_2,\,\, \mu_4 = \frac{1}{6}N_{(0, 1)}  \text{ in } \alpha_3 ,\]
\[ \nu_1 = \frac{1}{10} \delta_{10} \text{ in }  \gamma_1,\,\, \nu_2 = \frac{1}{8} U_{(3, 6)} \text{ in }  \gamma_3,\,\,
\nu_3 = \frac{1}{4} \delta_{-1} \text{ in }  \gamma_1,\,\, \nu_4 = \frac{1}{6} N_{(2, 1)} \text{ in } \gamma_2. \]
Since $\deuc (\supp(\mu_1), \supp(\mu_2)) = 8, \, \deuc (\supp(\nu_1), \supp(\nu_3)) = 11$ and the distances of any other pairs are infinity, 
we obtain $\sep(\phi) = 8$. So we can let $r$ to be any number less than $4$, let say $r=3$, so that they meet the condition of a triple.
With the choice of $x_1 = -10, x_2 = 4, x_3=2, x_4 =-2$, one can check $W (\mu_k, \nu_k *\delta_{x_k}) = 0$ for $k = 1, 3, 4$ and hence
	\[\sum\limits_{k=1}^{4} W (\mu_k, \nu_k *\delta_{x_k}) = W \Big(\frac{1}{8} U_{(8, 9)}, \frac{1}{8} U_{(7, 10)} \Big).\]
When it comes to the $I_r$ terms, we view measures $\mu_k$ and $\nu_k$ as elements of $\Xc$:
	\[\mu_1 = (\frac{1}{10} \delta_{0}, \mathbf{0}, \mathbf{0}, \cdots), \,\, \mu_2 = (\frac{1}{8} U_{(8, 9)}, \mathbf{0}, \mathbf{0}, \cdots),
	\,\, \mu_3 = (\mathbf{0}, \frac{1}{4} \delta_{1}, \mathbf{0}, \cdots), \,\, \mu_4 = (\mathbf{0}, \mathbf{0}, \frac{1}{6} N_{(0, 1)}, \mathbf{0}, \cdots), \]
	\[\nu_1 = (\frac{1}{10} \delta_{10}, \mathbf{0}, \mathbf{0}, \cdots), \,\, \nu_2 = (\mathbf{0}, \mathbf{0}, \frac{1}{8} U_{(3, 6)}, \mathbf{0}, \cdots),
	\,\, \nu_3 = (\frac{1}{4} \delta_{-1}, \mathbf{0}, \mathbf{0},  \cdots), \,\,
	\nu_4 = (\mathbf{0}, \frac{1}{6} N_{(2, 1)}, \mathbf{0}, \cdots). \]
Therefore, we have
	\[\mu - \sum\limits_{k=1}^{4} \mu_k  = (\frac{1}{40} \delta_{0}, \frac{1}{6} N_{(1, 4)}, \frac{1}{24} N_{(0, 1)}, \mathbf{0}, \mathbf{0}, \cdots ), \quad
	\nu - \sum\limits_{k=1}^{4} \nu_k = (\mathbf{0}, \mathbf{0}, \frac{1}{24} U_{(3, 6)}, \frac{1}{8} N_{3, 3}, \mathbf{0}, \mathbf{0}, \cdots).\]
\end{rmk}

\bigskip

We can now define
\beq\label{def: new metrization on X}
d(\mu, \nu) = \inf\limits_{r, \phi, \vec{x}} d_{r, \phi, \vec{x}} (\mu, \nu),
\eeq
where the infimum is taken over all $(\mu, \nu)$-triples. One can check that the choice of representatives of $\mu,\nu \in \Xct$ does not affect the value of $d(\mu, \nu)$ so it is well-defined in $\Xct$.
One can readily check that 
\beq \label{eq: upper bound of d}
	 d(\mu, \nu) \leq 2,\quad \mu, \nu \in \Xct, 
\eeq
by choosing the empty matching and letting $r \rightarrow \infty$ in a $(\mu, \nu)$-triple.

Let $\phi^{-1} := \{(\nu_k, \mu_k)\}_{k=1}^{n} \in P_{\nu, \mu}$. Then, we see that $\sep(\phi)=\sep(\phi^{-1})$ and  hence
	\[ d_{r, \phi, \vec{x}} (\mu, \nu) = d_{r, \phi^{-1}, -\vec{x}} (\nu, \mu), \]
which implies that $d$ is symmetric. With two propositions below, we prove that $d$ is a metric on $\Xct$.

\begin{prop}
		$d(\mu, \nu) =0$ if and only if $\mu = \nu$.
\end{prop}

\begin{proof}
Since the \textit{``if''} part is obvious, it suffices to prove the \textit{``only if''} part.
Let $d(\mu, \nu) =0$ and $(\alpha_i)_{i \in \Nb}$, $(\gamma_i)_{i \in \Nb}$ be representatives of $\mu, \nu$, respectively. 
We may assume $\norm{\alpha_i} \geq \norm{\alpha_{i+1}}$ and $\norm{\gamma_i} \geq \norm{\gamma_{i+1}}$ for all $i$ by rearranging the order if needed.
For each $m \in \Nb$, there is a $(\mu, \nu)$-triple $\big(r_m, \phi_m = \{(\mu_{m, k},\nu_{m, k})\}_{k=1}^{n_m}, \vec{x}_m = (x_{m, 1}, \cdots, x_{m, n_m})\big)$ such that 
	\[ a_m:=d_{r_m, \phi_m, \vec{x}_m}(\mu, \nu)  < \frac{1}{m}. \]
Note that $r_m \rightarrow \infty$.

Suppose $\alpha_1 = \mathbf{0}$ (i.e. $\mu = \mathbf{0}$). If $\norm{\gamma_1} >\delta>0$, since the empty matching is the only option, we have
	\[ a_m > I_{r_m}(\nu) \geq I_{r_m}(\gamma_1) > \delta \]
for all sufficiently large $m$, which is a contradiction. Hence, $\norm{\gamma_1}=0$ and $\mu = \nu = \mathbf{0} $.

Now suppose $\norm{\alpha_1}>0$. By the same argument as above, we have $\norm{\gamma_1}>0$. 
We may assume 
\beq
\label{eq: alpha_1_ge_gamma_1_in_norm}
\norm{\alpha_1} \geq \norm{\gamma_1},
\eeq
and let $p \in \Nb$ be an integer 
such that $\norm{\gamma_1}= \cdots = \norm{\gamma_p} >\norm{\gamma_{p+1}}$.

Since $a_m$ converges to 0, there is at least one integer $l=l(m)$ such that $\mu_{m, l} \leq \alpha_1$ for all sufficiently large $m$. In fact, if this does not hold, then
	\[ a_m > I_{r_m}\Big(\mu-\sum\limits_{k} \mu_{m, k}\Big) \geq I_{r_m}(\alpha_1) \rightarrow \norm{\alpha_1}>0, \]
a contradiction. 
By rearranging the order of pairs in $\phi_m$, we may assume that $\mu_{m, 1} \leq \alpha_1$ and it has the biggest mass among $(\mu_{m, k}: \mu_{m, k} \leq  \alpha_1 )_k$.

\par For any $\epsilon\in(0,\norm{\alpha_1}/4)$, let us choose $R=R(\epsilon)$ such that $\alpha_1(B_R(0)^c) <\epsilon$. Then, for all $m$ satisfying $r_m > R$, there is at most one sub-measure $\mu_{m, j(m)} \leq \alpha_1$ whose support has an overlap with $B_R(0)$ since $\sep(\phi_m)>2r_m>2R$. 
If $\mu_{m, j(m)}(B_R(0)) \leq \alpha_1(B_R(0))- \epsilon$ for infinitely many $m$, then, for these $m$ we have
	\[ a_m \geq I_{r_m}\Big(\mu - \sum\limits_{k} \mu_{m, k}\Big) \geq I_{r_m}\big(\one_{B_R(0)} \alpha_1 - \one_{B_R(0)}\mu_{m, j(m)}\big) \geq \epsilon, \]
which implies that $a_m$ does not converge to 0. Therefore, 
\beq \label{eq: lower bound of submeasure in matching}
	\|\mu_{m, j(m)}\| >\alpha_1 (B_R(0)) - \epsilon > \|\alpha_1\| - 2\epsilon
\eeq
for all sufficiently large $m$, and for such $m$, $j(m)=1$ by the definition of $\mu_{m ,1}$.

\par We claim that there is $q \in \Nb$ such that $\nu_{m, 1} \leq \gamma_q$ for infinitely many $m$. 
To see this, let $q_m$ be an integer such that $\nu_{m, 1} \leq \gamma_{q_m}$. If there is no such an integer $q$ as claimed above, 
we have $q_m \rightarrow \infty$ as $m \rightarrow \infty$.
It follows that $\norm{\gamma_{q_m}} \rightarrow 0$. On the other hand, for all sufficiently large $m$, 
$\|\gamma_{q_m}\| \geq \|\nu_{m, 1}\| = \|\mu_{m, 1}\| > \|\alpha_1\| -2 \epsilon $ by \eqref{eq: lower bound of submeasure in matching}, 
which is a contradiction. Hence, the claim is proved and, moreover, we obtain 
	\[ \|\gamma_{q}\| > \|\alpha_1\| -2 \epsilon. \]
Here, $q=q(\epsilon)$ may depend on $\epsilon$.
However, since $\|\gamma_{q(\epsilon)}\| > \|\alpha_1\| -2 \epsilon \geq \|\alpha_1\|/2$ for all $\epsilon\in(0,\|\alpha_1\|/4)$ and, given $\nu$, 
there are at most $\big[\frac{2}{\norm{\alpha_1}}\big]$ (Here, $[\,\cdot\,]$ denotes the integer part) indices $i$ such that
$\|\gamma_i\|\ge\norm{\alpha_1}/2$,
there is $q \in \Nb$, independent of $\epsilon$, such that $q = q(\frac{1}{n})$ for infinitely many $n$. 
For such $q$, we have $\| \gamma_{q} \| \geq \| \alpha_1 \|$.
Combining this with \eqref{eq: alpha_1_ge_gamma_1_in_norm} we obtain $\| \alpha_1 \| = \| \gamma_q \|$, so $q \leq p$.
By interchanging $\gamma_1$ and $\gamma_q$, we may assume $q=1$.
\par Let small $\epsilon >0$ be given. We choose $R$ as above and $R'=R'(\epsilon)$ such that $\gamma_1(B(0, R')^c) <\epsilon$. 
We can obtain $\| \nu_{m, 1} \| > \| \gamma_1 \|- 2\epsilon $ for all sufficiently large $m$ by applying the same argument used for~$\alpha_1$.
Then, for all sufficiently large $m$, 
\begin{align*}
		\wt{W}(\wt{\alpha}_1, \wt{\gamma}_1) \leq W(\mu_{m, 1}, \nu_{m, 1}*\delta_{x_{m, 1}}) + W\big(\alpha_1 -\mu_{m, 1}, (\gamma_1 -\nu_{m ,1})*\delta_{x_{m, 1}}\big) 
		\leq a_m +2 \epsilon.
\end{align*}
We used \eqref{ineq: Wsplit} in the first inequality.
Letting $m\to\infty$ first and then $\epsilon \downarrow 0$, we have $\wt{W}(\wt{\alpha}_1, \wt{\gamma}_1)=0$, i.e. $\wt{\alpha}_1 = \wt{\gamma}_1$.
Peeling off $\alpha_1$ and $\gamma_1$ from $\mu$ and $\nu$ and repeating the same process to obtain $\wt{\alpha}_i = \wt{\gamma}_i$ for all $i$, we complete the proof.
\end{proof}

\begin{prop}
		$d(\mu, \nu) \leq d(\mu, \eta) + d(\eta, \nu)$.
\end{prop}
	
\begin{proof}
Throughout the proof, we will identify $\mu, \eta$ and $\nu$ as elements of $\Xc$ by choosing their representatives.
Let $\epsilon > 0$ be given. We can choose triples
	\[ \Big(r_1, \phi_1= \big\{(\mu_k, \eta_{1, k})\big\}_{k=1}^{n_1}, \vec{x}_1 = \big(x_{1, k} \big)_{k=1}^{n_1}\Big),\ 
	\Big(r_2, \phi_2 = \big\{ (\eta_{2, k}, \nu_k)\big\}_{k=1} ^{n_2}, \vec{x}_2= \big(x_{2,k}\big)_{k=1}^{n_2}\Big) \]
such that 
	\[ d_{r_1, \phi_1, \vec{x}_1}(\mu, \eta) < d(\mu, \eta)+ \epsilon, 
	\quad d_{r_2, \phi_2, \vec{x}_2}(\eta, \nu) < d(\eta, \nu)+ \epsilon. \]

We say that $\eta_{1, k}$ and $\eta_{2, l}$ overlap,
if the measure $\bar{\eta}_{k, l} := \min(f, g) \eta$ has non-zero mass, where
$f$ and $g$ are the Radon-Nikodym derivatives of $\eta_{1, k}$ and $\eta_{2, l}$ with respect to $\eta$.
We collect such overlap measures between $\{\eta_{1, k}\}$ and  $\{\eta_{2, k}\}$ and relabel them as $\{\bar{\eta}_k\}_{k=1} ^{n}$. 

For each $l\in\{1,2,\ldots,n\}$, there are $j_1=j_1 (l)$ and $j_2=j_2(l)$ such that $\bar{\eta}_l = f_{j_1}\eta_{1, j_1}$ and $\bar{\eta}_l = f_{j_2}\eta_{2, j_2}$ for some measurable $f_{j_1}$ and $f_{j_2}$ with $0 \leq f_{j_1}, f_{j_2} \leq 1$.
In other words, $\bar{\eta}_l$ is an overlap measure of $\eta_{1, j_1(l)}$ and $\eta_{2, j_2(l)}$.
We remark that $j_a$ can be understood as a function which maps $\{1, 2, \cdots, n\}$ to $\{1, 2, \cdots, n_a\}$ for $a = 1, 2$.

Let us fix $l\in\{1,2,\ldots,n\}$ so that we can shorten $j_1(l)$ as $j_1$ for the moment. 
Let us denote by~$\pi_{j_1}$ the optimal coupling between $\mu_{j_1}$ and $\eta_{1, j_1} *\delta_{x_{1, j_1}}$: 
	\[ W\big(\mu_{j_1}, \eta_{j_1}*\delta_{x_{1, j_1}}\big) = \norm{\mu_{j_1}} \int |x-y|\wedge 1 \,\, \pi_{j_1} (dx, dy). \]
Then, there is a submeasure $\bar{\mu}_l$ of $\mu_{j_1}$ such that $\bar{\mu}_l$ is coupled to 
$\bar{\eta}_l*\delta_{x_{1, j_1}}$ (which is a submeasure of $\eta_{1, j_1} *\delta_{x_{1, j_1}}$) by $\pi_{j_1}$.
More precisely, we can define $\bar{\pi}_l (dx, dy) : = f_{j_1} (y) \pi_{j_1}(dx, dy)$, 
notice that $\bar{\eta}_l*\delta_{x_{1, j_1}}(\cdot)=\norm{\mu_{j_1}} \int \bar{\pi}_l (dx,\cdot)$, and define
$\bar{\mu}_l(dy)=\norm{\mu_{j_1}} \int \bar{\pi}_l (\cdot,dy)$.
The identity
\beq\label{eq: split of Wasserstein metric}
	W\big(\mu_{j_1}, \eta_{1, j_1}*\delta_{x_{1, j_1}}\big) 
	= W\big(\mu_{j_1} - \bar{\mu}_l, (\eta_{1, j_1} -\bar{\eta}_l)*\delta_{x_{1, j_1}}\big) 
	+ W\big(\bar{\mu}_l, \bar{\eta}_l*\delta_{x_{1, j_1}}\big)
\eeq
is a specific case of the following lemma:
\begin{lem} 
Let $\alpha$ and $\gamma$ be measures on $\Rb^d$ with equal total masses and let $\pi$ be the optimal coupling between
them. That is, 
	\[ \alpha(\cdot)=\|\alpha\|\int \pi(\cdot,dy), \,\, \gamma(\cdot)=\|\alpha\|\int \pi(dx,\cdot), \,\,  
	W(\alpha, \gamma) = \|\alpha\| \int |x-y|\wedge 1\,\, \pi(dx, dy). \]
Let $\bar \gamma = f \gamma$ for some measurable $f$ satisfying $0\le f(y)\le 1$ for all $y$.
We define then $\bar \pi(dx,dy)=f(y)\pi(dx,dy)$, notice that $\bar \gamma (\cdot)= \|\alpha\| \int \bar\pi(dx,\cdot)$,
and define $\bar \alpha(\cdot)=\|\alpha\|\int \bar\pi(\cdot,dy)$. Then, we have
	\[ W(\alpha, \gamma)=W(\alpha-\bar{\alpha}, \gamma-\bar{\gamma})+ W(\bar{\alpha}, \bar{\gamma}). \]
\end{lem}

\begin{subproof}[Proof of lemma]
Choosing couplings
	\[ \frac{\norm{\alpha}}{\norm{\alpha-\bar{\alpha}}}(1-f(y)) \pi (dx, dy), \qquad \frac{\norm{\alpha}}{\norm{\bar{\alpha}}} f(y) \pi (dx, dy) \] 
for ($\alpha-\bar{\alpha}, \gamma-\bar{\gamma})$ and $(\bar{\alpha}, \bar{\gamma})$, respectively, gives
	\[ W(\alpha, \gamma) \geq W(\alpha-\bar{\alpha}, \gamma-\bar{\gamma})+ W(\bar{\alpha}, \bar{\gamma}). \]
The reverse inequality follows from \eqref{ineq: Wsplit}.
\end{subproof}

For each $1 \leq k \leq n_1$, applying \eqref{eq: split of Wasserstein metric} inductively to all $l \in j_1 ^{-1}(k) = \{m : j_1(m) = k \}$, we obtain
\beq \label{eq: split of Wasserstein metric, second version}
	W\big(\mu_k, \eta_{1, k}*\delta_{x_{1, k}}\big) = 
	\sum_{\sss{l \in j_1 ^{-1}(k)}} W\big(\bar{\mu}_l, \bar{\eta}_l*\delta_{x_{1, k}}\big)+
	W\bigg(\mu_k - \sum_{\sss{l \in j_1 ^{-1}(k)}} \bar{\mu}_l,\  
	\Big(\eta_{1, k} -\sum_{\sss{l \in j_1 ^{-1}(k)}}\bar{\eta}_l\Big)*\delta_{x_{1, k}}\bigg)
\eeq
so that the optimal coupling between $\mu_k$ and $\eta_{1, k} *\delta_{x_{1, k}}$ can be split into overlapping parts of $\eta_{1, k}$ and the remaining part.
Repeating the same process for $\nu$ in place of $\mu$, we can define $\bar{\nu}_k$.

Now let us define $r =\min(r_1, r_2)$, $\phi = \{(\bar{\mu}_k, \bar{\nu}_k)\}_{k=1}^{n}$, and $\vec{x} = (x_{1, j_1(k)}  + x_{2, j_2(k)}) \in \Rb^n$. Since
 \[\sep(\phi) \geq \min(\sep(\phi_1),\sep(\phi_2)) > \min(2r_1, 2r_2)=2r,\]
$(r, \phi, \vec{x})$ is a $(\mu, \nu)$-triple.

From the subadditivity property of $I_r$ and the following inequality
\begin{align*}
	W \big(\bar{\mu}_k, \bar{\nu}_k  * \delta_{x_k}\big)  
	&= W \big(\bar{\mu}_k * \delta_{-x_{1, j_{1}(k)}}, \bar{\nu}_k  * \delta_{x_{2, j_{2}(k)}}\big)
	\\& \leq W \big(\bar{\mu}_k * \delta_{-x_{1, j_{1}(k)}}, \bar{\eta}_k\big) + W \Big(\bar{\eta}_k, \bar{\nu}_k * \delta_{x_{2, j_{2}(k)}}\Big)
	\\& = W \big(\bar{\mu}_k , \bar{\eta}_k * \delta_{x_{1, j_{1}(k)}}\big) +  W \Big(\bar{\eta}_k, \bar{\nu}_k  * \delta_{x_{2, j_{2}(k)}}\Big),
\end{align*}
we have that
\begin{align*}
	d(\mu, \nu) &\leq d_{r, \phi, \vec{x}} (\mu, \nu) = \sum_{k=1} ^{n} W \big(\bar{\mu}_k, \bar{\nu}_k  * \delta_{x_k}\big) 
	+ I_r\Big(\mu-\sum_{k=1} ^{n} \bar{\mu}_k\Big) + I_r\Big(\nu-\sum_{k=1} ^{n} \bar{\nu}_k\Big) + 2^{-r}
	\\& \leq \sum\limits_{k=1} ^{n} W \big(\bar{\mu}_k, \bar{\eta}_k*\delta_{x_{1, j_1(k)}}\big) 
	+ I_r\Big(\mu-\sum_{k=1} ^{n_1} \mu_k\Big) + I_r\Big(\sum_{k=1} ^{n_1} \mu_k - \sum_{k=1} ^{n} \bar {\mu}_k\Big)
	\\& \,\,+ \sum\limits_{k=1} ^{n} W \big(\bar{\eta}_k, \bar{\nu}_k * \delta_{x_ {2, j_2(k)}}\big)+  I_r\Big(\nu-\sum_{k=1} ^{n_2} \nu_k\Big) 
		+ I_r\Big(\sum_{k=1} ^{n_2} \nu_k - \sum_{k=1} ^{n} \bar {\nu}_k\Big)+2^{-r}.
		\numberthis\label{eq: division of d(mu, nu) into seven parts}
\end{align*}

We claim that 
\beq \label{eq: inclusion-exclusion for overlapping submeasures}
	\sum_{k=1} ^{n_1} \eta_{1, k} - \sum_{k=1} ^{n} \bar {\eta}_k \leq \eta - \sum_{k=1} ^{n_2} \eta_{2, k}.
\eeq 
To see this, let $f_k, g_k$ be the Radon-Nikodym derivatives of  $\eta_{1, k}, \eta_{2, k}$ with respect to $\eta$. 
Since $\big\{\supp(\eta_{1, k})\big\}$, $\big\{\supp(\eta_{2, k} )\big\}$ are disjoint, respectively, 
we have $\sum f_k \leq 1$ and $\sum g_k \leq 1$ pointwise. Therefore,
	\[ \sum_{k=1} ^{n_1} \eta_{1, k} + \sum_{k=1} ^{n_2} \eta_{2, k}- \sum_{k=1} ^{n} \bar {\eta}_k 
	= \Big(\sum_{k=1} ^{n_1} f_k \vee \sum_{k=1} ^{n_2} g_k\Big) \eta \leq \eta,\]
which proves the claim. Note that \eqref{eq: inclusion-exclusion for overlapping submeasures} can be rewritten as 
$\sum \eta_{2, k}  - \sum \bar {\eta}_k \leq \eta - \sum \eta_{1, k}$.

We next observe that
\begin{align}
	 &I_r\Big(\sum\limits_{k=1}^{n_1} \mu_k - \sum\limits_{k=1}^{n} \bar {\mu}_k\Big) 
	 =  I_r\bigg(\sum\limits_{k=1}^{n_1} \Big(\mu_k - \sum\limits_{\sss{l \in j_1 ^{-1} (k)}} \bar {\mu}_l\Big)\bigg)
	 \label{eq: concentration function evaluated at non-overlapping part of mu}
	 \\& =\sup\limits_{\scr{u \in \Nb\times \Rb^d}} \sum\limits_{k=1}^{n_1}  \int g_r(u, v) \,\,\Big(\mu_k - \sum\limits_{\sss{l \in j_1 ^{-1} (k)}} \bar {\mu}_l\Big)(dv)
	 = \sup_{1 \leq k \leq n_1} I_r\Big(\mu_k - \sum\limits_{\sss{l \in j_1 ^{-1} (k)}} \bar {\mu}_l\Big).\nonumber
\end{align}
In the last identity, we used the fact that since $\sep(\phi_1) > 2 r_1  \geq  2r$, 
the support of $g_r(u, \cdot)$ cannot intersect with $\supp(\mu_k)$ and $\supp(\mu_m)$ at the same time for any $k \neq m$.
Similarly, we have 
\beq \label{eq: concentration function evaluated at nonoverlapping part of eta}
	I_r \Big(\sum_{k=1}^{n_1} \eta_{1, k} - \sum_{k=1}^{n} \bar {\eta}_k\Big) 
	= \sup\limits_k \,  I_r\Big(\eta_{1, k}  - \sum\limits_{\sss{l \in j_1 ^{-1} (k)}} \bar {\eta}_l\Big).
\eeq

We now see that
\begin{align*}
	 I_r\Big(\sum_{k=1} ^{n_1} \mu_k & - \sum_{k=1} ^{n} \bar{\mu}_k\Big)
	 \leq I_r\Big(\sum_{k=1}^{n_1} \mu_k-\sum_{k=1}^{n} \bar {\mu}_k\Big) 
	- I_r\Big(\sum_{k=1}^{n_1} \eta_{1, k}-\sum_{k=1}^{n} \bar {\eta}_k\Big) + I_r \Big(\eta-\sum_{k=1}^{n_2} \eta_{2, k} \Big)
	\\& = \sup_{1 \leq k \leq n_1} I_r\Big(\mu_k - \sum_{\sss{l \in j_1 ^{-1} (k)}} \bar {\mu}_l\Big)
		-\sup_{1 \leq k \leq n_1} I_r\Big(\eta_{1, k}-\sum_{\sss{l \in j_1 ^{-1} (k)}} \bar {\eta}_l\Big) 
		+ I_r \Big( \eta - \sum_{k=1}^{n_2} \eta_{2, k}\Big)
	\\& \leq  \sup_{1 \leq k \leq n_1} \bigg[I_r\Big(\mu_k-\sum\limits_{\sss{l \in j_1 ^{-1} (k)}} \bar {\mu_l}\Big) 
		-I_r\Big(\eta_{1, k} -\sum_{\sss{l \in j_1 ^{-1} (k)}} \bar {\eta}_l\Big)\bigg]+I_r \Big(\eta-\sum_{k=1}^{n_2} \eta_{2, k} \Big)
	\\& \leq \sup_{1 \leq k \leq n_1} W\bigg(\mu_k - \sum_{\sss{l \in j_1 ^{-1} (k)}} \bar {\mu}_l, 
		\Big(\eta_{1, k} - \sum_{\sss{l \in j_1 ^{-1} (k)}} \bar {\eta}_l\Big)*\delta_{x_{1, k}}\bigg) 
		+ I_r \Big(\eta - \sum_{k=1}^{n_2} \eta_{2, k} \Big)
	\\& = \sup_{1 \leq k \leq n_1} \Big[W\big(\mu_k, \eta_{1, k}*\delta_{x_{1,k}}\big) 
		- \sum_{\sss{l \in j_1 ^{-1} (k)}} W\big(\bar {\mu}_l, \bar {\eta}_l *\delta_{x_{1, k}}\big)\Big] 
		+ I_r \Big(\eta - \sum_{k=1}^{n_2} \eta_{2, k}\Big)
	\\& \leq \sum_{k=1}^{n_1} W\big(\mu_k, \eta_{1, k}*\delta_{x_{1, k}}\big) 
		- \sum_{k=1}^{n} W\big(\bar{\mu}_k, \bar{\eta}_k *\delta_{x_{1, j_1(k)}}\big) + I_{r_2} \Big(\eta - \sum_{k=1}^{n_2} \eta_{2, k} \Big), 
	\numberthis\label{eq: conversion of nonoverlapping match of mu into (mu, eta) distance terms} 
\end{align*}
where, along with monotonicity of $I_r$, we used \eqref{eq: inclusion-exclusion for overlapping submeasures} in the first line,~\eqref{eq: concentration function evaluated at non-overlapping part of mu} and~\eqref{eq: concentration function evaluated at nonoverlapping part of eta} in the second line, 
shift-invariance of $I_r$ and~\eqref{eq: relation between concentration function and Wasserstein metric} in line 4, \eqref{eq: split of Wasserstein metric, second version} in line 5, 
and in the last line we replaced the maximal term by the sum of all terms.
For the same reason, we obtain
\beq \label{eq: conversion of nonoverlapping match of nu into (eta, nu) distance terms}
	\sum_{k=1}^{n} W \big(\bar{\eta}_k, \bar{\nu}_k *\delta_{x_{2, j_2(k)}}\big)
	+ I_r\Big(\sum_{k=1}^{n_2} \nu_k - \sum_{k=1}^{n} \bar {\nu}_k\Big) 
	\leq \sum_{k=1}^{n_2} W\big(\eta_{2, k}, \nu_k *\delta_{x_{2, k}}\big) +  I_{r_1} \Big(\eta - \sum_{k=1}^{n_1} \eta_{1, k} \Big).
\eeq
Collecting \eqref{eq: division of d(mu, nu) into seven parts},~\eqref{eq: conversion of nonoverlapping match of mu into (mu, eta) distance terms} 
and~\eqref{eq: conversion of nonoverlapping match of nu into (eta, nu) distance terms}, we have 
\begin{align*}
	d(\mu, \nu) &\leq \sum_{k=1}^{n_1} W\big(\mu_k, \eta_{1, k} *\delta_{x_{1, k}}\big) 
	+ I_{r_1}\Big(\mu -\sum_{k=1}^{n_1} \mu_k\Big)+ I_{r_1} \Big(\eta - \sum_{k=1}^{n_1} \eta_{1, k} \Big)
	\\&  + \sum_{k=1}^{n_2} W\big(\eta_{2, k}, \nu_k*\delta_{x_{2, k}}\big)+ I_{r_2}\Big(\nu -\sum_{k=1}^{n_2} \nu_k\Big) 
	+ I_{r_2} \Big(\eta  - \sum_{k=1}^{n_2} \eta_{2, k} \Big) + 2^{-r}
	\\& \leq d_{r_1, \phi_1, \vec{x}_1} (\mu, \eta) + d_{r_2, \phi_2, \vec{x}_2 } (\eta, \nu) < d(\mu, \eta) + d(\eta, \nu) +2\epsilon.
\end{align*} 
Letting $\epsilon \downarrow 0$ completes the proof.
\end{proof}

Having proved that $d$ is a metric on $\Xct$, we can now study properties of the metric space $(\Xct, d)$.

\subsection{Compactness and equivalence to the MV topology}
In this section, we prove that  $(\Xct,d)$ is compact and equivalent to the original MV space $(\Xct, \mathbf{D})$.
We recall that $\wt{\Mc}_1$ is naturally embedded in $\Xct$ since we can identify any $\wt{\alpha} \in \wt{\Mc}_1$ 
with the element $[\wt{\alpha}] \in \Xct$ having representative $(\alpha, \zero, \zero, \cdots)$.

\begin{thm} \label{prop:compactification}
 The space $(\Xct, d)$ is a compactification of $\wt{\Mc}_1$, i.e.,
\begin{enumerate}[label=\rm(\alph*), nolistsep]
 	\item the collection of orbits $\wt{\Mc}_1$ is dense in $(\Xct, d)$;
 	\item for any sequence $(\mu_n)_{n \in \Nb}$ in $\wt{\Mc}_1$, there is a subsequence convergent in $(\Xct,d)$.
\end{enumerate}
\end{thm}

\begin{proof}
This proof is similar to that of Theorem 3.2 in \cite{MV16}.

(a) Let $\mu \in \Xct$ be given and $(\alpha_i)_{i \in \Nb} \in \Xc$ be a representative of $\mu$. 
We may assume $\norm{\alpha_i} \geq \norm{\alpha_{i+1}}$ for all $i \geq 1$.
For each $m \in \Nb$, there are $n=n(m)$ and $R=R(m)$ such that 
	\[\|\alpha_{n+1}\| < \frac{1}{m}, \quad \sum_{j =1}^{n} \alpha_j \big(B_R(0)^c \big) < \frac{1}{m}. \]
We may assume $2^{-R} <1/m$.
We denote by $\ker_{N}$ the product measure on $\Rb^d$ with centered Gaussian marginals of variance $N$.
We can choose $N=N(m)$ and $\vec{x} = \vec{x} (m) =(x_1, \cdots, x_n) \in (\Rb^d)^n$ such that $I_R(\ker_{N})<1/m$ and 
$\max\limits_{i \neq j} \euc{x_i - x_j} > 4R$.
Let
	\[ \mu_{m}= \mu_{n, N} := \sum_{j=1}^{n} \alpha_j *\delta_{x_j} +\Big(1-\sum_{j=1} ^{n} \norm{\alpha_j}\Big) \ker_{N} \,\in \Mc_1.\]
Let us write $B$ and $B_j$ for $B_R(0)$ and $B_R(x_j)$.
One can consider a $(\mu_m, \mu)$-matching given by
	\[\phi = \{(\mu_{1,j}, \mu_{2, j})\}_{j=1}^{n}, \quad
	\mu_{1, j} = (\one_{B_j}(\alpha_j * \delta_{x_j}), \zero, \zero, \cdots), \,\,
	\mu_{2, j} = (\underbrace{\zero, \cdots, \zero}_{j-1}, \one_{B} \alpha_j, \zero, \zero, \cdots),\]
We observe that 
	\[\deuc \big(\supp(\mu_{1,k}), \supp(\mu_{1, l}) \big) > 2R, \,\,\deuc \big(\supp(\mu_{2,k}), \supp(\mu_{2, l}) \big) = \infty
	\quad \text{for all } k \neq l,\] 
which implies $\sep(\phi) > 2R$. Therefore, $(R, \phi, \vec{x})$ is a $(\mu_m, \mu)$-triple and hence
\begin{align*}
	 &d(\mu_m, \mu) \leq d_{R, \phi, \vec{x}} (\mu_m, \mu) 
	\\& = \sum_{j=1}^{n} W\Big(\one_{B_j} (\alpha_j*\delta_{x_j}), \big(\one_{B} \alpha_j \big)*\delta_{x_j} \Big)
	+  I_R\Big(\mu_m - \sum_{j=1}^{n} \mu_{1, k}\Big)
	+ I_R\Big(\mu- \sum_{j=1}^{n} \mu_{2, k} \Big) +2^{-R}.
\end{align*}
Since $\one_{B_j} (\alpha_j*\delta_{x_j}) = \big(\one_{B} \alpha_j \big)*\delta_{x_j}$, all Wasserstein terms above vanish.
We estimate $I_R$ terms.
	\[  I_R\Big(\mu_m - \sum_{j=1}^{n} \mu_{1, k}\Big) 
	\leq I_R\Big(\sum_{j=1}^{n} \one_{B_j ^c}(\alpha_j * \delta_{x_j}) \Big) 
	+ I_R \Big(\big(1-\sum_{j=1} ^{n} \norm{\alpha_j}\big) \ker_{N}\Big) < \frac{2}{m}. \] 
We can decompose $\mu- \sum_{j=1}^{n} \mu_{2, k} $ into two parts:
	\[\mu- \sum_{j=1}^{n} \mu_{2, k}  :=  \mu_1 ^s + \mu_2 ^s, \quad 
	\mu_1 ^s = (\one_{B^c}\alpha_1, \cdots, \one_{B^c}\alpha_n, \zero, \zero, \cdots), 
	\mu_2 ^s = (\underbrace{\zero, \cdots,\zero}_n, \alpha_{n+1}, \alpha_{n+2}, \cdots).\]
It is easy to check that 
	\[I_R \Big(\mu- \sum_{j=1}^{n} \mu_{2, k}  \Big) \leq  I_R (\mu_1 ^s) +  I_R (\mu_2 ^s) < \frac{2}{m}. \] 
Collecting all estimates gives $d(\mu_m, \mu) < 5/m$, which implies that $\mu_m \rightarrow \mu$ as $m \rightarrow \infty$.
\vspace{0.2in}

(b) We now show that for any $(\wt{\mu}_n)_{n \in \Nb}$ in $\wt{\Mc}_1$, there is a subsequence that converges to  $\mu \in \Xct$.
Since $I_r$ is bounded by $1$, by passing to a subsequence, we may assume that for every $r>0$, there is $q_0(r)\ge 0$ such that
	\[ \lim\limits_{n \rightarrow \infty} I_{r} (\mu_n)= q_0 (r). \]

Let $\mu_{n, 0} =\mu_{n}$. For each $m \in \Nb$, 
we can choose inductively a subsequence $(\mu_{n, m})_{n\geq1}$ of  $(\mu_{n, m-1})_{n \geq 1}$ such that the limit
	\[ \lim\limits_{n \rightarrow \infty} I_{m} (\mu_{n, m}) = q_0 (m) \]
exists. Since $I_m(\mu_{n, m}) \leq I_{m+1}(\mu_{n, m})$, $q_0 (m)$ is non-decreasing in $m$. 
Therefore, $q_0 := \lim\limits_{m \rightarrow \infty} q_0 (m)$ is well-defined and one has
	\[\lim\limits_{n \rightarrow \infty} I_{m}(\mu_{n, n}) = q_0 (m), \quad \lim\limits_{r \rightarrow \infty}\lim\limits_{n \rightarrow \infty} I_r(\mu_{n, n}) = q_0. \]
For simplicity of notation, we write $\mu_n$ for $\mu_{n, n}$ from now on.	

If $q_0 =0$ (i.e. $q_0 (r) =0$ for all $r>0$), then for any $r>0$, by choosing the empty matching $\emptyset$, we have
\begin{align*}
	\limsup\limits_{n \rightarrow \infty} d(\mu_n, \mathbf{0}) &\leq \limsup\limits_{n \rightarrow \infty}  d_{r, \emptyset,0} (\mu_n, \mathbf{0})
	= \limsup\limits_{n \rightarrow \infty}I_r (\mu_n) +I_r (\mathbf{0})+2^{-r} = 2^{-r}.
\end{align*}
Letting $r \rightarrow \infty$, we obtain that $\wt{\mu}_n$ converges to $\mathbf{0}$ in $(\Xct, d)$.

If $q_0 > 0$, by choosing a suitable sequence $(a_{n, 1})_{n \in \Nb}$ in $\Rb^d$, we have for some $r>1$, 
	\[ \mu_n * \delta_{a_{n, 1}}\big(B_r(0)\big) \geq I_{r-1} (\mu_n) \geq \frac{q_0}{2} \]
for all sufficiently large $n$. 
Due to the compactness of $\Mc_1$ in the vague topology, by taking a subsequence if needed, 
we may assume 
	\[ \lambda_n:=\mu_n * \delta_{a_{n, 1}} \hookrightarrow \alpha_1. \]
Note that $\norm{\alpha_1} \geq q_0/2$.
By Lemma 2.2 in \cite{MV16}, there is a decomposition $\mu_n= \alpha_{n, 1} + \beta_{n,1}$ such that
	\[ \alpha_{n, 1} *\delta_{a_{n,1}}\Rightarrow \alpha_1, \quad \beta_{n, 1} *\delta_{a_{n, 1}} \hookrightarrow 0. \]
In particular, it is proved in Theorem 3.2 of \cite{MV16} 
that if $q_0=1$, then $\beta_{n,1}$ can be taken to be $\mathbf{0}$, so we have $\wt{\mu}_n \rightarrow \wt{\alpha}_1$ in $(\Xct, d)$.

If $0<q_0<1$, we can repeat this iteratively. More precisely, for each $k \in \Nb$, we can define 
$q_k = \lim\limits_{r \rightarrow \infty} \lim\limits_{n \rightarrow \infty} I_r(\beta_{n, k})$ in the same way as $q_0$. Then, there is a a sequence $(a_{n, k+1})_{n \in \Nb}$ in $\Rb^d$ such that
	\[ \beta_{n, k} = \alpha_{n, k+1} + \beta_{n, k+1} \]
and
	\[ \|\alpha_{k+1}\| \geq \frac{q_k}{2}, \quad \alpha_{n, k+1} *\delta_{a_{n,k+1}}\Rightarrow \alpha_1, \quad \beta_{n, k+1} *\delta_{a_{n, k+1}} \hookrightarrow 0.\]
If there is $k \in \Nb $ such that $ q_k = 0$,  then we have the following decomposition for $\mu_n$:
	\[ \mu_n = \sum\limits_{j=1}^{k} \alpha_{n, j} + \beta_{n, k},\]
where 
 	\[ \lim\limits_{r\rightarrow \infty}\lim\limits_{n\rightarrow \infty} I_{r}(\beta_{n, k}) = 0, \]
	\beq\label{eq: convergence of components consisting of mu-decomposition} 
		\alpha_{n, j} *\delta_{a_{n, j}}\Rightarrow \alpha_j,  \quad
		\beta_{n, j} *\delta_{a_{n, j}}\hookrightarrow 0,\quad 1\leq j \leq k, 
	\eeq
	\beq\label{eq: alpha_i's are far away from each other}
		 \lim\limits_{n \rightarrow \infty} \euc{a_{n,i}- a_{n, j}} = \infty,\quad i \neq j.
	\eeq
To see \eqref{eq: alpha_i's are far away from each other}, let us assume that it is not true. 
By taking a subsequence, we may assume the limit $b :=\lim\limits_{n \rightarrow \infty} (a_{n,i}- a_{n, j})$ exists for some $i>j$. We observe that
\beq\label{ineq: beta_j and alpha_i}
	\beta_{n, i} * \delta_{a_{n, i}} \geq \alpha_{n,j}* \delta_{a_{n, i}} = (\alpha_{n,j}* \delta_{a_{n, j}}) * \delta_{a_{n,i}- a_{n, j}}.
\eeq
Since $\beta_{n, i} *\delta_{a_{n, i}}\hookrightarrow 0$, 
\beq\label{eq: convergence of beta_n, i} 
	\lim\limits_{n \rightarrow \infty} \beta_{n, i} * \delta_{a_{n, i}}(K) = 0
\eeq
for any compact set $K$ in $\Rb^d$.
On the other hand, it follows from $\alpha_{n, j} *\delta_{a_{n, j}}\Rightarrow \alpha_j$ and $\|\alpha_j\| \geq q_{j-1}/2$ that 
there is a compact set $K_j$ such that $\alpha_{n, j} * \delta_{a_{n, j}}(K_j) \geq  q_{j-1} /3$ for all sufficiently large $n$. Let
	\[ K_j ' = \{x \in \Rb^d: x-b \in B_1(y) \text{\, for some\,\,} y \in K_j \}. \]
Then, one can readily check that $\alpha_{n, j} * \delta_{a_{n, i}}(K_j ') \geq  q_{j-1} /3$ for all sufficiently large $n$. 
Combining this with \eqref{eq: convergence of beta_n, i}  gives contradiction to \eqref{ineq: beta_j and alpha_i}.

We claim that $\mu_n \rightarrow \mu= [\wt{\alpha}_1, \cdots , \wt{\alpha}_k ]$ in $(\Xct,d)$. The argument is the same as in the proof of \eqref{eq:mu-tilde-n-conv-to-mu} below, and we omit it here.

If $q_k > 0$ for every $k \in \Nb$, then
there are $(\alpha_{n, j}), (\beta_{n, j})$ in $\Mc_{\leq 1}$ and $(a_{n, j})$ in $\Rb^d$
such that for all $n,k \in \Nb$,  \eqref{eq: convergence of components consisting of mu-decomposition}, \eqref{eq: alpha_i's are far away from each other} hold and
	\[ \mu_n = \sum\limits_{j=1}^{k} \alpha_{n, j} + \beta_{n, k}. \]
Since $\norm{\alpha_j} \geq q_{j-1} /2$ and $\sum\limits_{j \in \Nb} \| \alpha_j \| \leq 1$, we have $q_j \rightarrow 0$.  
We claim that 
\beq \label{eq:mu-tilde-n-conv-to-mu}
	\wt{\mu}_n \rightarrow \mu := [\wt{\alpha}_j]_{j \in \Nb}\quad \text{\rm in}\ (\Xct, d).
\eeq
Let $\epsilon>0$ be given. We first choose $k=k(\epsilon) \in \Nb$ such that $q_k < \epsilon$. 
There is $r=r(\epsilon)$ such that
	\[ \sum\limits_{j=1}^k \alpha_{j}\big(B_r(0)^c\big)<\epsilon, \quad \sum\limits_{j=1}^k \alpha_{n, j}\big(B_r(0)^c\big)<\epsilon \]
for all sufficiently large $n \geq N_1$. We may assume $2^{-r} < \epsilon$.
We can also find $N_2$ such that $\inf\limits_{\substack{i\neq j \\ i, j \leq k}}\euc{a_{n,i}- a_{n, j}}> 2r $ for all $n \geq N_2$. 
Recalling the definition of $f_r$ in~\eqref{eq: concentration function}, we choose a $(\mu, \mu_n)$-matching $\phi = \Big\{\big(f_{r}\alpha_{j}, f_{r}(\cdot +a_{n, j}) \alpha_{n, j}\big) \Big\}_{j=1}^k$
and $\vec{x}_n =(a_{n, 1}, \cdots, a_{n, k})$.
For all $n \geq \max(N_1, N_2)$,
\begin{align*}
	d(\mu, \mu_n) &\leq d_{r, \phi, \vec{x}_n} (\mu, \mu_n) 
	= \sum\limits_{j=1}^{k} W\Big(f_r \alpha_{j}, f_r (\alpha_{n, j}*\delta_{a_{n, j}})\Big)
	\\&+ I_r \Big(\mu - \sum\limits_{j=1}^k  f_r \alpha_j \Big) +I_r \Big(\sum\limits_{j=1}^{k}\big(1-f_r(\cdot + a_{n, j})\big)\alpha_{n, j} + \beta_{n, k} \Big) +2^{-r}.
	\numberthis\label{eq: estimate for d(mu, mu_n)}
\end{align*}
By \eqref{eq: convergence of components consisting of mu-decomposition}, we have  for each $1 \leq j \leq k$,
\beq\label{eq: estimate for W(alpah_j, alpha_n,j)}
	\lim\limits_{n\rightarrow \infty} W\Big(f_r \alpha_{j}, f_r (\alpha_{n, j}*\delta_{a_{n, j}})\Big) = 0.
\eeq
It follows from the subadditivity of $I_r$ that
\begin{align*}
	I_r \Big(\mu - \sum\limits_{j=1}^k  f_r \alpha_j \Big) &\leq \sum\limits_{j=1}^{k} I_r \Big((1-f_r)\alpha_j \Big) + I_r \Big([\wt{\alpha}_{j+k}]_{j \in \Nb}\Big) 
	\\& \leq \sum\limits_{j=1}^k \alpha_{n, j}\big(B_r(0)^c\big) + q_k < 2\epsilon.\numberthis\label{eq: estimate for the first I_r term} 
\end{align*}
Similarly, we can also obtain
\beq\label{eq: estimate for the second I_r term} 
	I_r \Big(\sum\limits_{j=1}^{k}\big(1-f_r(\cdot + a_{n, j})\big)\alpha_{n, j} + \beta_{n, k} \Big) 
	\le \sum\limits_{j=1}^{k} I_r \Big(\big(1-f_r(\cdot + a_{n, j})\big)\alpha_{n, j}\Big) +I_r( \beta_{n, k} )
	< 2 \epsilon.
\eeq
Plugging~\eqref{eq: estimate for W(alpah_j, alpha_n,j)}, \eqref{eq: estimate for the first I_r term} and \eqref{eq: estimate for the second I_r term} 
into~\eqref{eq: estimate for d(mu, mu_n)}, one has $\limsup\limits_{n \rightarrow \infty} d(\mu, \mu_n) < 5 \epsilon$, completing the proof.
\end{proof}

In our proof of the equivalence between the MV topology and the topology defined by our metric~$d$,
we will use the following theorem:

\begin{thm}[Theorem 26.6 in \cite{Mun00}] \label{thm: homeomorphism}
Let $X,Y$ be two topological spaces and let $f : X \rightarrow Y$ be a bijective continuous function. If $X$ is compact and $Y$ is Hausdorff, then $f$ is homeomorphism.
\end{thm}

Before proving the equivalence statement, we recall the original MV metrization $\mathbf{D}$ defined by~\eqref{eq: original metrization D} 
and the functional $\Lambda$ defined by \eqref{eq: functional for metric D}.

\begin{prop}
	$(\Xct, d)$ is equivalent to $(\Xct, \mathbf{D})$.
\end{prop}

\begin{proof}
We fix $k \geq 2$. Since $(\Xct, d)$ is compact by Proposition~\ref{prop:compactification} and $(\Xct, \mathbf{D})$ is Hausdorff being a metric space, 
it suffices, due to Theorem~\ref{thm: homeomorphism}, to show the continuity of the identity map $e: (\Xct, d) \rightarrow (\Xct, \mathbf{D})$.

By the Portmanteau Theorem, $\Lambda (f, \mu^{(n)}) \rightarrow \Lambda (f, \mu)$ for all $f \in \mathcal{F}_k$ is equivalent to 
$\Lambda (f, \mu^{(n)}) \rightarrow \Lambda (f, \mu)$ for all bounded, Lipschitz continuous $f \in \mathcal{F}_k$. 
Therefore, it suffices to show that for given $\epsilon>0$ and bounded Lipschitz continuous $f \in \mathcal{F}_k$, there is $\delta = \delta(\epsilon, f)>0$
such that
	\[ d(\mu, \nu) < \delta \quad \Rightarrow \quad  |\Lambda(f, \mu) - \Lambda(f, \nu)| < \epsilon. \]
We may assume $0\leq f \leq 1$ and $f$ is 1-Lipschitz continuous. Let us choose $M=M(\epsilon, f)>0$ such that 
$f(x_1, \cdots, x_k) < \epsilon/4$ whenever $\max\limits_{i\neq j} \euc{x_i -x_j} \geq M$
and let
	\[ \delta = \min \Big(\frac{\epsilon}{4k}, 2^{-M} \Big). \]
Let us assume that $d(\mu,\nu)<\delta$. Then there is a $(\mu,\nu)$-triple 
\beq\label{eq:mu-nu-triple}
	\Big(r, \phi = \{(\mu_j, \nu_j)\}_{j=1}^{n}, \vec{y} = (y_1, \cdots, y_n)\Big)
\eeq
such that $d_{r, \phi, \vec{y}} (\mu, \nu) < \delta$.
 Let 
	\[ \mu^{s} = (\alpha^{s}_j ) = \mu - \sum_{j=1}^{n} \mu_j , \quad \nu^{s} =(\gamma^{s}_j) = \nu - \sum_{j=1}^{n} \nu_j. \]
Notice that since $r > M$, we have $\sup\limits_{u \in \Nb\times \Rb^d} \mu^{s} \big(B_M(u)\big) \leq I_r (\mu^{s}) < \delta$.

 Let $\alpha_j ^{c} = \alpha_j - \alpha_j ^{s} = \sum_{l: \mu_l \leq \alpha_j} \mu_l$ for each $j$. 
We divide each term of $\Lambda(f,\mu)$ into a core part and a sparse part:
 \begin{align*}
	&\int f(x_1, \cdots, x_k) \prod_{i=1}^{k} \alpha_j(dx_i) = \int f(x_1, \cdots, x_k) \prod_{i=1}^{k} (\alpha^{c}_j + \alpha^{s}_j) (dx_i)
	\\& = \int f(x_1, \cdots, x_k) \prod_{i=1}^{k} \alpha^{c}_j(dx_i) +  \sum_{\vec{t} \in \{c, \,s\}^k \setminus \{c\}^k}\int f(x_1, \cdots, x_k)  \prod_{i=1}^{k} \alpha^{t_i}_j(dx_i).
	\numberthis\label{eq: split of the integration into core part and sparse part}
 \end{align*}
 
Let $A = \{x \in \Rb^d : \max\limits_{i\neq j} \euc{x_i -x_j} \geq M\}$ and $|t|$ be the number of occurrences of $s$ in $\vec{t}$. 
 For any $\vec{t} \in \{c, \,s\}^k \setminus \{c\}^k$,
 there is a number $l$ such that $t_l = s$, so
\begin{align*}
	&\int f(x_1, \cdots, x_k)  \prod \alpha^{t_i}_j(dx_i)	 = \int_{A} f(x) \prod \alpha^{t_i}_j(dx_i) +  \int_{A^c} f(x) \prod \alpha^{t_i}_j(dx_i)
	\\& < \frac{\epsilon}{4}  \|\alpha^{c}_j\|^{k-|t|} \|\alpha^{s}_j\|^{|t|} + \int_{|x_l - x_{l-1}| < M} \prod \alpha^{t_i}_j(dx_i) 
	\\& \leq \frac{\epsilon}{4} \|\alpha^{c}_j\|^{k-|t|}\|\alpha^{s}_j\|^{|t|} 
	+ \sup\limits_{x \in \Rb^d} \alpha^{s}_l \big(B_M(x)\big) \|\alpha^{c}_j\|^{k-|t|} \|\alpha^{s}_j\|^{|t|-1}
	\\& \leq \frac{\epsilon}{4} \|\alpha^{c}_j\|^{k-|t|}\|\alpha^{s}_j\|^{|t|} + \delta \|\alpha^{c}_j\|^{k-|t|} \|\alpha^{s}_j\|^{|t|-1}.
	\numberthis\label{eq: split of the sparse part into two subparts}
\end{align*}
From the binomial theorem, we have
\beq\label{eq: estimate for sparse subpart1}
	\sum\limits_{\vec{t} \in \{c, \,s\}^k \setminus \{c\}^k} \|\alpha^{c}_j\|^{k-|t|}\|\alpha^{s}_j\|^{|t|} 
	\leq \sum\limits_{\vec{t} \in \{c, \,s\}^k} \|\alpha^{c}_j\|^{k-|t|}\|\alpha^{s}_j\|^{|t|} 
	 \leq \|\alpha_j\|^k \leq  \|\alpha_j\|
\eeq
and, by the mean value theorem,
\beq\label{eq: estimate for sparse subpart2}
	\sum\limits_{\vec{t} \in \{c, \,s\}^k \setminus \{c\}^k} \|\alpha^{c}_j\|^{k-|t|}\|\alpha^{s}_j\|^{|t|-1}
	= \frac{\big(\|\alpha^{c}_j\| + \|\alpha^{s}_j\|\big)^k - \|\alpha^{c}_j\|^{k}}{\|\alpha^{s}_j\|} = kp^{k-1} \leq k \|\alpha_j\|
\eeq
for some $p\in \big[\|\alpha^{c}_j\|, \|\alpha_j\|\big]$. Combining \eqref{eq: estimate for sparse subpart1} and \eqref{eq: estimate for sparse subpart2} 
with~\eqref{eq: split of the sparse part into two subparts} gives
\beq \label{eq: estimate for the entire sparse part}
	\sum_{\vec{t} \in \{c, \,s\}^k \setminus \{c\}^k}\int f(x_1, \cdots, x_k)  \prod_{i=1}^{k} \alpha^{t_i}_j(dx_i) \leq \big(\frac{\epsilon}{4}+ k\delta \big) \|\alpha_j\|.
\eeq
For the core part, 
\begin{align*}
	&\int f(x_1, \cdots, x_k) \prod_{i=1}^{k} \alpha^{c}_j(dx_i) 
	= \int f(x_1, \cdots, x_k) \prod_{i=1}^{k} \Big(\sum\limits_{l: \mu_l \leq \alpha_j} \mu_l \Big)(dx_i)
	\\& \leq \sum\limits_{l: \mu_l \leq \alpha_j} \int f(x_1, \cdots, x_k) \prod_{i=1}^{k} \mu_l (dx_i)  + \frac{\epsilon}{4} \| \alpha_j ^c \|^k,
	\numberthis\label{eq: estimate for the core part}
\end{align*}
where we used in the inequality the fact that $\sep(\phi) \geq 2r >M$, so $|f|<\epsilon/4$ on the support of the off-diagonal products of $\mu_l$'s.
Substituting~\eqref{eq: estimate for the entire sparse part} and~\eqref{eq: estimate for the core part} 
into~\eqref{eq: split of the integration into core part and sparse part} 
and summing over all $j$, we obtain
\beq \label{eq: upper bound for Lambda(f, mu)}
	\Lambda(f, \mu)  \leq \sum_{j=1}^{n} \Lambda(f, \mu_j) + \frac{\epsilon}{2} + k \delta.
\eeq
On the other hand, it follows from the non-negativity of $f$ that
\beq\label{eq: lower bound for Lambda(f, mu)}
	\Lambda(f, \mu) \geq \Lambda(f, \sum\limits_{j=1}^{n} \mu_j) \geq \sum\limits_{j=1}^{n} \Lambda(f, \mu_j).
\eeq
Estimates similar to~\eqref{eq: upper bound for Lambda(f, mu)} and~\eqref{eq: lower bound for Lambda(f, mu)} also hold true for $\Lambda(f, \nu)$.

We now give an upper bound for $W(\alpha^{\otimes k}, \gamma^{\otimes k})$ in terms of $W(\alpha, \gamma)$.
Let $\pi$ be the optimal coupling of $(\alpha, \gamma)$.
Then, $\pi^{\otimes k}$ is a coupling of $(\alpha^{\otimes k}, \gamma^{\otimes k})$ and
\begin{align*}
	W(\alpha^{\otimes k}, \gamma^{\otimes k}) &\leq \int (\euc{\vec{x} - \vec{y}}\wedge 1) \pi^{\otimes k} (d\vec{x}, d\vec{y})
	\leq \int \sum\limits_{j=1}^{k} (\euc{x_j - y_j} \wedge 1) \pi^{\otimes k} (d\vec{x}, d\vec{y}) 
	\\& = \sum\limits_{j=1}^{k} \int (\euc{x_j - y_j} \wedge 1)  \pi(dx_j, dy_j) = kW(\alpha, \gamma).\numberthis\label{eq: Wasserstein metric between product measures}
\end{align*}
Combining~\eqref{eq: upper bound for Lambda(f, mu)},~\eqref{eq: lower bound for Lambda(f, mu)} and~\eqref{eq: Wasserstein metric between product measures}, 
we conclude that
\begin{align*}
	|\Lambda(f, \mu) - \Lambda(f, \nu)| &< \sum\limits_{j=1}^{n} |\Lambda(f, \mu_j) - \Lambda(f, \nu_j)| +\frac{\epsilon}{2} +  k \delta 
	\\& \leq \sum\limits_{j=1}^{n} W\Big(\mu_j ^{\otimes k}, (\nu_j*\delta_{y_j})^{\otimes k}\Big) +\frac{\epsilon}{2} +  k \delta 
	\\& \leq k \sum\limits_{j=1}^{n} W \big(\mu_j , \nu_j*\delta_{y_j} \big) +\frac{\epsilon}{4} +  k \delta  \leq 2k\delta +\frac{\epsilon}{2} \leq \epsilon,
\end{align*}
where the inequality in the second line follows from the translational invariance of $\Lambda$ and \eqref{eq: Kantorovich duality}
(we recall that $\vec{y} = (y_1, \cdots, y_n)$ was introduced in \eqref{eq:mu-nu-triple} as an element of the $(\mu, \nu)$-triple).
\end{proof}

\section{The update map}\label{sec: update map}
In this section, following~\cite{BC16}, we define an ``update map'' $\Tc$
which maps the law of the polymer endpoint distribution of length $n$ to that of length $n+1$, and prove that $\Tc$ is continuous. 
As in Section~\ref{sec: main results}, the endpoint distribution for the polymer of length $n$ is denoted by 
	\[ \rho_n (dx)=\mathscr{M}_n (\omega_n \in dx). \]
Notice that $\rho_n$ is a random measure on the probability space $(\Omega_e, \mathscr{G}, \Pb)$ of random environment.
We denote by $\Pc(\Xct)$ the space of Borel probability measures on $\Xct$ and endow the space $\Pc{(\Xct})$ with the Wasserstein metric $\Wc$:
\beq\label{def: Wasserstein metric on P(X)}
	\Wc(\xi_1, \xi_2) : = \inf \int_{\Xct\times \Xct } d(\mu, \nu) \pi(d\mu, d\nu),
\eeq
where the infimum is taken over all couplings $\pi$ of $(\xi_1, \xi_2)$.

\subsection{The conditional update map}
In this section, we define a ``conditional'' update map $T:\Xct \rightarrow \Pc(\Xct)$ that maps $\rho_n$ 
to the law of $\rho_{n+1}$ given $\mathscr{G}_n$.
We recall that $P$, $P^x$ and $\lambda$ were defined in~\eqref{def: single step measure} and below it.
For each $n \geq 1$, let us define
\beq\label{def: finite dimensional distribution of P}
	 P_n ^x \,(\text{or } P_n) = \text{ the law of } (\omega_1, \cdots, \omega_n) \text{ under } P^x \,( \text{or } P).
\eeq
We observe that
\begin{align*}
	\rho_{n+1} (dx) &= \frac{1}{Z_{n+1}} E\Big[\exp\Big(\beta \sum\limits_{i=1}^{n+1} X(i, \omega_i)\Big) \one_{\{\omega_{n+1} \in dx\}}\Big]
	\\&= \frac{1}{Z_{n+1}} \int_{(\Rb^d)^{n}} \exp\Big(\beta \sum\limits_{i=1}^{n} X(i, y_i) +\beta X(n+1, x) \Big) P_{n+1}(dy_1, \cdots, dy_n, dx)
	\\&= \frac{1}{Z_{n+1}} \int_{(\Rb^d)^{n}} \exp\Big(\beta \sum\limits_{i=1}^{n} X(i, y_i) +\beta X(n+1, x) \Big) 
	P_{n}(dy_1, \cdots, dy_n)\ker \big(d(x-y_n)\big)
	\\& =\frac{Z_n}{Z_{n+1}} \int_{\Rb^d} e^{\beta X(n+1, x)} \rho_n (dy_n) \ker \big(d(x-y_n)\big)
	\\& = \frac{Z_n}{Z_{n+1}} e^{\beta X(n+1, x)} \rho_n*\ker(dx).
\end{align*}
Integrating over $x$ on the both sides gives
	\[ \frac{Z_{n+1}}{Z_n} =  \int_{\Rb^d} e^{\beta X(n+1, x)} \rho_n*\ker(dx). \]
Since $X(n+1, \cdot \,)$ is independent of $\mathscr{G}_n$, the law of $\rho_{n+1} $ given $\mathscr{G}_n$ is equal to the law of 
\beq\label{eq: law of endpoint distribution of length n+1 given G_n}
	\rho(dx) := \frac{e^{\beta Y(x)} \rho_n*\ker(dx)}
	{\int_{\Rb^d} e^{\beta Y(z)} \rho_n*\ker(dz)},
\eeq
where $Y(\cdot) \overset{d}{=} X(\cdot)$ and $Y$ is independent of $\mathscr{G}_n$.

In general, for $\mu=(\alpha_i) \in \Xc$, we can consider a $\Xc$-valued random variable
$\hat{\mu} = (\hat{\alpha}_i) $ given by
\beq\label{eq: conditional update map}
	\hat{\alpha}_i (dx) : 
		 = \frac{e^{\beta Y(i, x)}\,\alpha_i * \ker (dx)}
		 {\sum\limits_{j=1}^{\infty} \int_{\Rb^d} e^{\beta Y(j, z)} \alpha_j * \ker (dz) + (1 - \norm{\mu})e^{c(\beta)}}, 
\eeq
where $\big(Y(u)\big)_{u \in \Nb \times \Rb^d} \overset{d}{=} \big(X(u)\big)_{u \in \Nb \times \Rb^d}$ 
and $c(\cdot)$ is defined in \eqref{def: logarithmic moment generating function}.

Notice that~\eqref{eq: conditional update map} is a generalization of~\eqref{eq: law of endpoint distribution of length n+1 given G_n} 
because $\|\rho_n\|=1$ for all $n\in\Nb$. 
The additional term in the denominator allows us to define $\hat{\mu}$ when $\mu = \zero $ and we will see later that this term makes the (conditional) update map continuous.

For any $\mu=(\alpha_i)_{i \in \Nb} \in \Xc$ and $\gamma \in \Mc_{\leq 1}$,  we will write
\beq \label{eq: convolution in X}
	\mu * \gamma  := (\alpha_i * \gamma)_{i \in \Nb}.
\eeq
One can check that the convolution is also well-defined for $\mu \in \Xct$.
The measure $\hat{\mu}$ can be now expressed in terms of \eqref{eq: convolution in X} and integration on $\Nb \times \Rb^d$ as follows:
\beq\label{eq: conditional update map, 2nd version}
	\hat{\mu} (du) : = \frac{e^{\beta Y(u)}\, \mu * \ker (du)}
	{ \int_{\Nb\times \Rb^d} e^{\beta Y(w)} \mu *\ker (dw) + (1-\norm{\mu})e^{c(\beta)}}.
\eeq
We now would like to have \eqref{eq: conditional update map, 2nd version} to be well-defined on $\Xct$. 
For a fixed environment though, $\hat{\mu}$ does depend  on the choice of the representative of $\mu$.
However, the next proposition claims that the law of $\hat{\mu}$ is independent of the choice of representative.
We recall the equivalence relation $\sim$ on $\Xc$ introduced in Definition~\ref{def: equivalence relation}.
\begin{prop} \label{prop: T is well-defined}
For $\mu_1, \mu_2 \in \Xc$ with $\mu_1 \sim \mu_2$, define $\hat{\mu}_1$, $\hat{\mu}_2$ as in \eqref{eq: conditional update map, 2nd version}. 
Then, $\hat{\mu}_1 \overset{d}{=} \hat{\mu}_2$ as $\Xct$-valued random variables.
\end{prop}

\begin{proof}
It suffices to find a coupling of $(Y_1, Y_2)$ such that
\begin{enumerate}[label=$\bullet$, nolistsep]
	\item $Y_1, Y_2$ are random fields with the same law as $X$,
	\item $Y_i$ is used to define $\hat{\mu}_i$ in \eqref{eq: conditional update map, 2nd version},
	\item $\hat{\mu}_1 = \hat{\mu}_2$ in $\Xct$.
\end{enumerate}
Let $\mu_1=(\alpha_i)$ and $\mu_2 = (\gamma_i)$. 
From Definition~\ref{def: equivalence relation}, there are a sequence $\{x_i\}$ in $\Rb^d$ and a bijection $\sigma: S_{\mu_2} \rightarrow S_{\mu_1}$ such that
$\gamma_i  = \alpha_{\sigma (i)}* \delta_{x_i}$ for all $i \in S_{\mu_2}$. 
Let $Y_1, W$ be independent random fields with the same law as $X$ and set 
	\[Y_2 (i, x) = 
	\begin{cases}
		Y_1 (\sigma (i), x - x_i) & \text{if}\,\, i \in S_{\mu_2},\\
		W(i, x) & \text{otherwise}.
	\end{cases} \]
Then, we have for any $i \in S_{\mu_2}$,
\begin{align*}
		\hat{\gamma}_i(dx)
		& =\frac{e^{\beta Y_2 (i, x)}\,\gamma_i * \ker (dx)}
		 {\sum\limits_{j=1}^{\infty} \int_{\Rb^d} e^{\beta Y_2(j, z)} \gamma_j * \ker (dz) + (1 - \norm{\mu_2})e^{c(\beta)}}
		\\& =\frac{e^{\beta Y_1(\sigma (i),\,x-x_i)} \alpha_{\sigma (i)} *\delta_{x_i} *\ker(dx)}
		{\sum\limits_{j \in S_{\mu_2}} \int_{\Rb^d} e^{\beta Y_1(\sigma (j), z- x_j)} \alpha_{\sigma (j)} *\delta_{x_j} *\ker(dz) 
		+ (1-\norm{\mu_1})e^{c(\beta)}}
		\\&=\frac{e^{\beta Y_1 (\sigma (i),\, x)} \alpha_{\sigma (i)} *\ker(dx)}
		{\sum\limits_{j=1}^{\infty} \int_{\Rb^d} e^{\beta Y_1 (j, z)} \alpha_j * \ker(dz)
		+ (1-\norm{\mu_1})e^{c(\beta)} }
		\\&=\hat{\alpha}_{\sigma (i)} *\delta_{x_i} (dx),
\end{align*}
	which implies that $\hat{\alpha}_{\sigma (i)}$ and $\hat{\gamma}_i$ belong to the same orbit. 
\end{proof}

Proposition~\ref{prop: T is well-defined} allows us to define the update map $T : \Xct \rightarrow \Pc{(\Xct})$ by
	\[ T: \mu \,\,\mapsto\,\, \text{law of\,\,\,} \hat{\mu}. \]
Since $\wt{\Mc}_1$ is naturally embedded in $\Xct$, we can identify the endpoint distribution $\rho_n$ with a random element of $\Xct$.
As we discussed before, we have
	\[ T \mu (d \nu) = \mathbf{P}(\rho_{n+1} \in d\nu \,|\, \rho_n = \mu), \]
or equivalently,
\beq\label{eq: interpretation of conditional map as a transition kernel}
	T \rho_n (d\nu) \overset{a.s.}{=} \mathbf{P}(\rho_{n+1} \in d\nu \,|\, \rho_n) \overset{a.s.}{=} \mathbf{P}(\rho_{n+1} \in d\nu \,|\, \mathscr{G}_n).
\eeq
Therefore, $T \mu (d \nu):= \Gamma(\mu, d\nu)$ can be understood as a transition kernel for the Markov chain $(\rho_n)_{n \geq 0}$ on $\Xct$. 

\subsection{Construction of $d$ revisited}
Before proceeding to prove the continuity of the conditional update map, we explore an alternative construction 
of the new metric $d$ on $\Xct$, which was defined in~\eqref{def: new metrization on X}.
For any $\mu, \nu \in \Xc$, we call $\varphi := \{(\mu_k, \nu_k)\}_{k=1}^{n} $ a $(\mu, \nu)$-\textit{g-matching} (standing for \textit{generalized matching}) 
if it is a $(\mu, \nu)$-matching for which the first condition of the matching is relaxed (see Definition~\ref{def: matching}).
That is, paired submeasures $\mu_k$ and $\nu_k$ don't need to have the same mass in a g-matching. 
We can define $\sep(\varphi)$ and a \textit{g-triple} $(r, \varphi, \vec{x})$ in the same way as for matchings (see Definition~\ref{def: triples}). We denote the set of all $(\mu, \nu)$-g-matchings by $G_{\mu, \nu}$.
Given a $(\mu, \nu)$-g-triple $(r, \varphi, \vec{x})$, we define
	\[ d_{r, \varphi, \vec{x}} (\mu, \nu) := \sum\limits_{k=1}^{n} \hW (\mu_k, \nu_k *\delta_{x_k}) 
		+ I_r\Big(\mu - \sum\limits_{k=1}^{n} \mu_k\Big)+  I_r\Big(\nu- \sum\limits_{k=1}^{n} \nu_k\Big)+ 2^{-r}, \]
where $\hW$ is the generalized Wasserstein distance, defined in \eqref{def: generalized Wasserstein distance}.
We claim that
	\[ d_g(\mu, \nu) := \inf_{\substack{r, \varphi, \vec{x} \\ \varphi \in G_{\mu, \nu}}} d_{r, \varphi, \vec{x}}(\mu, \nu) 
	= \inf_{\substack{r, \phi, \vec{x}\\\phi \in P_{\mu, \nu}}} d_{r, \phi, \vec{x}} (\mu, \nu)=d(\mu, \nu). \]
The inequality $d(\mu,\nu)\ge d_g(\mu,\nu)$ is obvious because every $(\mu, \nu)$-matching is a $(\mu, \nu)$-g-matching. 
To see the reverse inequality, let us take any $\epsilon>0$ and choose a g-triple $\big(r, \varphi = \{ (\mu_k, \nu_k) \}_{k=1}^{n}, \vec{x}\big)$ such that $d_{r, \varphi , \vec{x}}(\mu, \nu) <d_g(\mu, \nu) +\epsilon$.
For each pair $(\mu_k, \nu_k)$, there exist submeasures $\bar{\mu}_k\leq \mu_k$ and $\bar{\nu}_k \leq \nu_k$ such that they have the same mass and  
	\[ \hW (\mu_k, \nu_k) = W(\bar{\mu}_k, \bar{\nu}_k)+\norm{\mu_k - \bar{\mu}_k} +  \norm{\nu_k - \bar{\nu}_k}. \]
Now we consider a triple $\big(r, \phi = \{(\bar{\mu}_k, \bar{\nu}_k)\}, \vec{x}\big)$. Then we have
\beq\label{eq: extract of matching from g-matching}
	d_{r, \phi, \vec{x}} (\mu, \nu) = \sum\limits_{k=1}^{n} W(\bar{\mu}_k, \bar{\nu}_k) 
	+ I_r \Big(\mu- \sum\limits_{k=1}^{n} \bar{\mu}_k\Big) +  I_r \Big(\nu- \sum\limits_{k=1}^{n} \bar{\nu}_k\Big) + 2^{-r}
\eeq
and by the subadditivity of $I_r$,
\beq\label{eq: estimate for concentration function from g-matching metric}
	I_r \Big(\mu- \sum\limits_{k=1}^{n} \bar{\mu}_k\Big) 
	\leq I_r \Big(\mu- \sum\limits_{k=1}^{n} \mu_k\Big) + I_r \Big(\sum\limits_{k=1}^{n} \mu_k - \bar{\mu}_k\Big) 
	\leq  I_r \Big(\mu- \sum\limits_{k=1}^{n} \mu_k\Big) +\sum\limits_{k=1}^{n} \norm{\mu_k - \bar{\mu}_k}.
\eeq
Plugging~\eqref{eq: estimate for concentration function from g-matching metric} into~\eqref{eq: extract of matching from g-matching}, we obtain that
	\[ d_{r, \phi, \vec{x}} (\mu, \nu) \leq d_{r, \varphi , \vec{x}}(\mu, \nu) < d_g(\mu, \nu) +\epsilon, \]
which completes the proof.

\subsection{Continuity of the conditional update map} \label{sec: continuity-cond-update} 
In this section, we prove that $T:(\Xct, d) \rightarrow \big(\Pc{(\Xct}), \Wc\big)$ is continuous. 
In our proof of continuity, we follow the general strategy used in~\cite{BC16}.
Some elements of our proof in the general continuous setting have appeared in \cite{BM18} for a Gaussian setting.
We begin with a useful lemma.

\begin{lem}\label{lem: estimate for denominator of conditional update map}
Let
\beq \label{eq: denominator of conditional update map}
	A =\int_{\Nb \times \Rb^d}  e^{\beta Y(u)} \mu * \ker (du) + \big(1-\norm{\mu}\big)e^{c(\beta)},
\eeq
where $Y \overset{d}{=} X$. Then, for any $p>0$,
	\[ \Eb A^{-p} \leq 2^p e^{ c(-p \beta)}. \]
\end{lem}

\begin{proof}
We consider two cases. If $\norm{\mu} \leq 1/2$, then by Jensen inequality,
	\[ \Eb A^{-p} \leq  \Big((1-\norm{\mu})e^{c(\beta)}\Big)^{-p} 
	\leq 2^p \big(\Eb e^{\beta Y(1, 0)} \big)^{-p} 
	\leq 2^p \Eb e^{-p \beta Y(1, 0)} = 2^p e^{ c(-p \beta)}. \]
If $\norm{\mu} > 1/2$, then again by Jensen's inequality,
\begin{align*}
	 \Eb A^{-p} &\leq  \Eb \Big(\int_{\Nb \times \Rb^d}  e^{\beta Y(u)} \mu *\ker (du) \Big) ^{-p} 
	\\& \leq \norm{\mu}^{-p} \Eb \int_{\Nb \times \Rb^d}  e^{-p\beta Y(u)} \frac{\mu *\ker}{\norm{\mu}} (du) 
	\\& = \norm{\mu}^{-p} e^{c(-p\beta)} \int_{\Nb \times \Rb^d}\frac{\mu *\ker}{\norm{\mu}} (du)
	 \leq 2^p e^{ c(-p \beta)}.
\end{align*} 
\end{proof}

\begin{prop} \label{prop: continuity of conditional update map}
$T: (\Xct, d) \rightarrow (\Pc(\Xct), \Wc)$ is continuous. That is, 
for any $\epsilon >0$, there is $\delta = \delta(\epsilon)>0$ such that for $\mu, \nu \in \Xct,$
	\[ d(\mu, \nu) <\delta \quad \Rightarrow \quad \Wc (T\mu, T \nu)< \epsilon. \]
\end{prop}

\begin{proof}
Given representatives $\mu = (\alpha_j), \nu = (\gamma_j) \in \Xc$,
let $\hat{\mu} = (\hat{\alpha}_j), \hat{\nu}=(\hat{\gamma})$ be $\Xc$-valued random variables with laws $T\mu, T\nu$, respectively. Since
\beq \label{eq: Wasserstein distance between Tmu Tnu}
	\Wc(T\mu, T\nu) = \min \Eb d(\hat{\mu}, \hat{\nu}),
\eeq
where the minimum is taken over all couplings of $(\hat{\mu}, \hat{\nu})$, our goal is to construct a coupling which makes 
$\Eb d(\hat{\mu}, \hat{\nu}) $ as small as possible.

Let $\eps \in (0, 1)$ be given. We will determine $\delta = \delta(\epsilon)>0$ later.
If $d(\mu, \nu) <\delta$, there is a triple $\big(r, \phi = \{(\mu_k, \nu_k)\}_{k=1}^{n}, \vec{x}= (x_1, \cdots, x_n)\big)$ such that
\beq \label{eq: setting (mu, nu)-triple}
	d_{r, \phi, \vec{x}} (\mu, \nu) <\delta.
\eeq
Recalling that $\rdep$ is the radius of dependence of the random field $X(1, \cdot)$,
we may assume $\delta < 2^{-{\rdep}}$ so that $r>\rdep$. 
\vspace{0.1in}

\noindent \textbf{(Part 1)} We first assume that for each $k$, there is at most one $j =j(k)$ such that $S_{\mu_j} = \{k\}$ and the same condition holds for $\nu$.
By rearranging the order of $(\alpha_j)$, $(\gamma_j)$ and translating each $\gamma_i$ if needed, 
we may assume $S_{\mu_k} = S_{\nu_k} = \{k\}$ for $1 \leq k \leq n$ and $\vec{x} = 0$.

Let us use the same environment $Y$ to define $ \hat{\alpha}_j$, $ \hat{\gamma}_j$ as in~\eqref{eq: conditional update map}.
In addition to $A$ defined in~\eqref{eq: denominator of conditional update map}, we introduce
\beq\label{eq: denominator of nuhat}
	B = \int_{\Nb\times \Rb^d} e^{\beta Y(u)} \nu *\ker (du) 
	+ \big(1-\norm{\nu}\big)e^{c(\beta)},
\eeq
	\[ \alpha^{\ast} _i(dx) = e^{\beta Y(i, x)} \alpha_i *\ker (dx), \]
	\[ \gamma^{\ast} _i(dx) = e^{\beta Y(i, x)} \gamma_i *\ker (dx).\]
Then, we can write $\hat{\alpha}_i = \alpha^{\ast} _i(dx)/A$, $\hat{\gamma}_i = \gamma^{\ast} _i(dx)/B$ due to \eqref{eq: conditional update map}. Similarly, let us define
\beq\label{eq: numerator of muhat}
	 \mu^{\ast} _k(dx) = e^{\beta Y(k, x)}\mu_k * \ker (dx), \quad \hat{\mu}_k  = \mu^{\ast}_k / A,
\eeq
	\[ \nu^{\ast} _k(dx) =  e^{\beta Y(k, x)}\nu_k * \ker (dx), \quad \hat{\nu}_k  = \nu^{\ast}_k / B. \]
Choosing a $(\hat{\mu}, \hat{\nu})$-g-triple $\big(r, \varphi = \{\hat{\mu}_k, \hat{\nu}_k\}_{k=1} ^{n}, 0\big)$, we have 
\beq\label{eq: gmatching of muhat and nuhat}
	\Eb \, d(\hat{\mu}, \hat{\nu}) \leq 
	\Eb \Big[ \sum_{k=1}^{n} \hW(\hat{\mu}_k, \hat{\nu}_k ) + I_r\Big(\hat{\mu}-\sum_{k=1}^{n} \hat{\mu}_k\Big)
	+ I_r \Big(\hat{\nu}-\sum_{k=1}^{n} \hat{\nu}_k \Big) +2^{-r}\Big],
\eeq
\vspace{0.1in}

\noindent \textbf{(Part 1.1)} Upper bound for $\Eb \sum \hW(\hat{\mu}_k, \hat{\nu}_k )$

We first observe the generalized Wasserstein distance terms.
\begin{align*}
	\Eb\hW(\hat{\mu}_k, \hat{\nu}_k) &= \Eb \hW\left(\frac{\mu^{\ast}_k}{A}, \frac{\nu^{\ast}_k}{B}\right)
	 \leq \Eb \,\hW \left( \frac{\mu^{\ast}_k}{A}, \frac{\nu^{\ast}_k}{A} \right) + \Eb \, \hW \left( \frac{\nu^{\ast}_k}{A}, \frac{\nu^{\ast}_k}{B} \right)
	 \\& \leq \Eb \frac{1}{A} \hW(\mu^{\ast}_k, \nu^{\ast}_k) + \Eb \left[\frac{|A-B|}{AB}\norm{\nu^{\ast}_k}\right].
	 \numberthis\label{eq: estimate for W(muhat_k, nuhat_k)}
\end{align*}
Summing over $k$ gives
\begin{align*}
	 \Eb \sum\limits_{k=1}^{n} \hW(\hat{\mu}_k, \hat{\nu}_k ) 
	 &\leq  \Eb\frac{1}{A} \sum\limits_{k=1}^{n}\hat W(\mu ^{\ast} _k, \nu_k ^{\ast}) 
	 +\Eb \left[\frac{|A-B|}{A}\sum\limits_{k=1}^{n} \frac{\norm{\nu^{\ast}_k}}{B}\right]
	 \\& \leq \Eb\frac{1}{A} \sum\limits_{k=1}^{n}\hat W(\mu ^{\ast} _k, \nu_k ^{\ast}) +\Eb\frac{|A-B|}{A}
	 \\& \leq (\Eb A^{-2})^{1/2} \bigg(\sum\limits_{k=1}^{n} \big[\Eb \hW(\mu_k ^{\ast}, \nu_k ^{\ast})^2\big]^{1/2} 
	 +  \big[\Eb (A-B)^2 \big]^{1/2} \bigg).\numberthis\label{eq: estimate for W(muhat_k, nuhat_k), continued}
\end{align*}
We need to estimate all the terms on the right-hand side of~\eqref{eq: estimate for W(muhat_k, nuhat_k), continued}. 
We start with $\Eb \hW(\mu^{\ast}_k, \nu^{\ast}_k )$.

\vspace{0.1in}
\noindent \textbf{(Part 1.1.1)} Upper bound for $\Eb \big[\hW(\mu_k ^{\ast}, \nu_k ^{\ast})^2\big]^{1/2} $

We observe that the stationary process $(e^{2\beta Y(x)})_{x \in \Rb^d}$ is uniformly integrable and by the path continuity, a
$\lim\limits_{\euc{x} \rightarrow 0} e^{\beta Y(x)} = e^{\beta Y(0)}$ almost surely.
It follows that $(e^{\beta Y(x)})_{x \in \Rb^d}$ is $L_2$-continuous, i.e.,
	\[ \lim\limits_{\euc{x} \rightarrow 0} \Eb\big(e^{\beta Y(x)} - e^{\beta Y(0)}\big)^2 = 0.\]
Therefore, there is $\delta_1 = \delta_1 (\epsilon) \in (0, 1)$ such that
	\[ \euc{x} < \delta_1  \quad \Rightarrow \quad \Eb(e^{\beta Y(x)} - e^{\beta Y(0)})^2 <  \epsilon^2. \]
Let  $\mu'_k=\mu_k * \ker$, $\nu'_k=\nu_k *\ker$. Let $\pi_k$ 
be the optimal coupling of $(\mu_k, \nu_k)$ and 
$\pi'_k$ be the optimal coupling of $(\mu'_k,\nu'_k)$. 
One can check that
	\[ W(\mu'_k, \nu'_k) \leq W(\mu_k, \nu_k) \]
by considering the following coupling of $(\mu'_k, \nu'_k)$:
	\[ \pi_k ''(dx,dy)=\int_{w \in \Rb^d} \ker(dw) \pi_k \big(d(x-w), d(y-w)\big). \]
Since $\norm{\mu_k} (e^{\beta Y(x)} \wedge e^{\beta Y(y)}) \pi'_k$ is a unnormalized sub-coupling of $(\mu^{\ast}_k, \nu^{\ast}_k)$, i.e.,
	\[ \norm{\mu_k} \int e^{\beta Y(\cdot)} \wedge e^{\beta Y(y)} \pi'_k (\cdot, dy)
	\leq \mu^{\ast}_k(\cdot), \quad 
	\norm{\mu_k} \int e^{\beta Y(x)} \wedge e^{\beta Y(\cdot)} \pi'_k (dx, \cdot)
	\leq \nu^{\ast}_k(\cdot), \]
we can use it to estimate
\begin{align*}
	&\hW(\mu^{\ast}_k, \nu^{\ast}_k)
	\leq \|\mu_k\| \int d_{\Euc}(x, y) (e^{\beta Y(x)} \wedge e^{\beta Y(y)}) \pi'_k (dx, dy) 
	\\&+ \|\mu_k\| \int \Big(e^{\beta Y(x)} - e^{\beta Y(x)} \wedge e^{\beta Y(y)}\Big) \pi'_k (dx, dy) 
	+ \|\mu_k\| \int \Big(e^{\beta Y(y)} - e^{\beta Y(x)} \wedge e^{\beta Y(y)} \Big) \pi'_k (dx, dy)
	\\& = \|\mu_k\|  \int d_{\Euc}(x, y) (e^{\beta Y(x)} \wedge e^{\beta Y(y)}) \pi'_k (dx, dy) + \|\mu_k\| \int |e^{\beta Y(x)} - e^{\beta Y(y)}| \pi'_k (dx, dy)
	\\& \leq \|\mu_k\| \int d_{\Euc}(x, y) e^{\beta Y(x)} \pi'_k (dx, dy) +  \|\mu_k\| \int |e^{\beta Y(x)} - e^{\beta Y(y)}| \pi'_k (dx, dy).
\end{align*}
Therefore,
\begin{align*}
	&\Eb \hW(\mu^{\ast}_k, \nu^{\ast}_k )^2
	 \\& \leq 2\|\mu_k\|^2 \left[\Eb\left(\int d_{\Euc}(x, y) e^{\beta Y(x)} \pi'_k (dx, dy)\right)^2 
	 + \Eb \left(\int |e^{\beta Y(x)} - e^{\beta Y(y)}| \pi'_k (dx, dy) \right)^2\right]
	\\& \leq 2 \|\mu_k\|^2 \left[ e^{c(2\beta)} \int d_{\Euc}(x, y)^2 \, \pi'_k(dx, dy) + \int \Eb  |e^{\beta Y(x)} - e^{\beta Y(y)}|^2 
	(\one_{\euc{x-y} \geq \delta_1} + \one_{\euc{x-y} <  \delta_1} ) \pi'_k (dx, dy) \right]
	\\& \leq 2 \|\mu_k\|^2 \left [ e^{c(2\beta)} \frac{W(\mu_k, \nu_k)}{ \|\mu_k\|}
	+ 4 e^{c(2\beta)}\frac{W(\mu_k, \nu_k)}{ \|\mu_k\| \delta_1} + \epsilon^2 \right].
\end{align*}
Taking the square root and summing over $k$, we conclude that there is a constant $C_0 = C_0 (\beta)>0$ such that
\beq\label{eq: estimate for sum of W(mu*, nu*)}	
	\sum\limits_{k=1}^{n}[\Eb \hW(\mu^{\ast}_k, \nu^{\ast}_k )^2]^{1/2} 
	\leq C_0 \Big(\sqrt{\frac{\delta}{\delta_1}} +\epsilon \Big).
\eeq

\vspace{0.1in}
\noindent\textbf{(Part 1.1.2)} Upper bound for $\Eb (A-B)^2$

To estimate $(A-B)^2$, let us define
\begin{gather*}
	V(u)= e^{\beta Y(u)} - e^{c(\beta)},\\
	\mu^{s} = (\alpha^{s}_j) := \mu - \sum\limits_{j=1}^{n} \mu_j ,\,\nu^{s} =(\gamma^{s}_j) := \nu - \sum\limits_{j=1}^{n} \nu_j.
	\numberthis\label{eq: sparse part of mu and nu}
\end{gather*}
One can write
	\[ A - e^{c(\beta)} = \int_{\Nb \times \Rb^d} V(u) \mu*\ker (du),\quad 
	B - e^{c(\beta)} = \int_{\Nb \times \Rb^d} V(u) \nu*\ker (du), \]
and therefore
\begin{align*}
	(A-B)^2&
	 =  \left(\int_{\Nb \times \Rb^d} V(u) \mu * \ker (du) 
	- \int_{\Nb \times \Rb^d}V(u) \nu * \ker (du) \right) ^2 
	\\& \leq 3\left[ \sum\limits_{k=1}^{n} \left( \int_{\Rb^d} V(k, x) \mu_k *\ker (dx)
	- \int_{\Rb^d} V(k, x) \nu_k * \ker (dx) \right) \right]^2
	\\&+3\Big(\int_{\Nb \times \Rb^d} V(u) \mu^{s} *\ker (du)\Big)^2
	+3\Big( \int_{\Nb \times \Rb^d} V(u) \nu^{s} *\ker (du)\Big)^2. \numberthis\label{eq: split of (A-B)^2}
\end{align*}
By the independence of $\big\{V(i, \cdot)\big\}_{i\in\Nb}$, we see that
\begin{align*}
	&\Eb \left[\sum\limits_{k=1}^{n} \Big(\int_{\Rb^d} V(k, x) \mu_k *\ker (dx)
	- \int_{\Rb^d} V(k, x) \nu_k * \ker (dx) \Big)\right]^2
	\\& = \sum\limits_{k=1}^{n} \Eb  \left( \int_{\Rb^d} V(k, x) \mu_k *\ker (dx)
	- \int_{\Rb^d} V(k, x) \nu_k * \ker (dx) \right)^2
	\\& =\sum\limits_{k=1}^{n}  \Eb \big(\|\mu^{\ast}_k\|- \|\nu^{\ast}_k \| \big) ^2 
	\leq \sum\limits_{k=1}^{n} \Eb \hW (\mu^{\ast}_k, \nu^{\ast}_k )^2.
	\numberthis\label{eq: estimate for the core part in (A-B)^2}
\end{align*}
Denote $c_{\beta}(x) = \Eb \Big[ \big(e^{\beta X(x)}-e^{c(\beta)} \big) \big(e^{\beta X(0)}-e^{c(\beta)} \big)\Big] $.
Notice that $c_{\beta} (x) \leq e^{c(2 \beta)}$ for all $x \in \Rb^d$.
Let us now give an upper bound for the second term in~\eqref{eq: split of (A-B)^2}:
\begin{align*}
	& \Eb \Big(\int_{\Nb \times \Rb^d}  V(u)\mu^{s} * \ker (du)\Big)^2 
	 = \sum\limits_{j=1}^{\infty} \Eb  \Big(\int_{\Rb^d}  V(j, x)\alpha_j^{s}*\ker (dx) \Big)^2 
	\\& = \sum\limits_{j=1}^{\infty} \Eb  \int_{(\Rb^d)^2}  V(j, x) V(j, \tilde{x}) {(\alpha_j^{s}*\ker)}^{\otimes 2}(dx, d\tilde{x})
	 \\& = \sum\limits_{j=1}^{\infty} \int_{(\Rb^d)^2}  c_{\beta}(x-\tilde{x}) {(\alpha_j^{s}*\ker)}^{\otimes 2}(dx, d\tilde{x}).
	 \numberthis\label{eq: estimate for the sparse part in (A-B)^2}
\end{align*}
One can see that for any $\alpha \in \Mc_{\leq 1}$ and $r \geq 0$,
\begin{align*}
	I_r (\alpha * \ker) &= \sup_{x_0 \in \Rb^d} \iint f_r (x-x_0) \alpha(dy)\ker(d(x-y)) 
	= \sup_{x_0 \in \Rb^d} \iint f_r (x+y-x_0) \alpha(dy)\ker(dx)
	\\& \leq \int \bigg(\sup_{x_0 \in \Rb^d} \int f_r (x+y-x_0) \alpha(dy)\bigg)\ker(dx) = I_r (\alpha).
	\numberthis\label{eq: I_r of convolution measure}
\end{align*}
We recall that $\rdep$ is the radius of dependence of the potential and $r>\rdep$.
Therefore,
\begin{align*}
	 &\int_{(\Rb^d)^2} c_{\beta}(x-\tilde{x}) {(\alpha_j^{s}*\ker)}^{\otimes 2} (dx, d\tilde{x})
	= \int_{\euc{x-\tilde{x}} \leq \rdep} c_{\beta}(x-\tilde{x})  {(\alpha_j^{s}*\ker)}^{\otimes 2}(dx, d\tilde{x})
	\\& \leq  e^{c(2\beta)} \int_{\euc{x-\tilde{x}} < r}  {(\alpha_j^{s}*\ker)}^{\otimes 2}(dx, d\tilde{x})
	=  e^{c(2\beta)} \int_{\Rb^d} \alpha_j^{s}*\ker (B_r (x))\alpha_j^{s}*\ker(dx)
	\\& \leq e^{c(2\beta)} \sup_{x \in \Rb^d} \alpha_j^{s}*\ker (B_r (x)) \int_{\Rb^d} (\alpha_j^{s}*\ker)(dx) 
	\leq e^{c(2\beta)} I_r (\alpha_j^{s}*\ker ) \norm{\alpha_j^s} <  e^{c(2\beta)} \delta \norm{\alpha_j^s}.
	\numberthis\label{eq: estimate for the sparse part in (A-B)^2, continued}
\end{align*}
Combining \eqref{eq: estimate for the sparse part in (A-B)^2} with~\eqref{eq: estimate for the sparse part in (A-B)^2, continued} gives
\beq\label{eq: estimate for the sparse part in (A-B)^2, final}
	\Eb \Big(\int_{\Nb \times \Rb^d}  V(u)\mu^{s} * \ker (du)\Big)^2 
	\leq \sum\limits_{j=1}^{\infty} e^{c(2\beta)} \delta \norm{\alpha_j^s} \leq e^{c(2\beta)}\delta.
\eeq
The same upper bound holds  for the third term in~\eqref{eq: split of (A-B)^2}. 
Plugging \eqref{eq: estimate for sum of W(mu*, nu*)},~\eqref{eq: estimate for the core part in (A-B)^2}, 
and~\eqref{eq: estimate for the sparse part in (A-B)^2, final} into~\eqref{eq: split of (A-B)^2}, we obtain
\begin{align*}
	&\big[\Eb(A-B)^2 \big]^{1/2} 
	\leq \sqrt{3} \Big(\sum\limits_{k=1}^{n} \Eb \hW (\mu^*_k, \nu^*_k)^2 + 2 e^{c(2\beta)} \delta \Big)^{1/2}
	\\& \leq \sqrt{3} \left(\sum\limits_{k=1}^{n} \Big(\Eb \hW (\mu^*_k, \nu^*_k)^2\Big)^{1/2} + e^{c(2\beta)/2} \sqrt{2\delta}\right)
	\leq C_1 \Big( \sqrt{\frac{\delta}{\delta_1}} + \epsilon \Big),\numberthis\label{eq: estimate of E(A-B)^2}
\end{align*}
where $C_1 = C_1(\beta)>0$ is some constant depending only on $\beta$. 

\vspace{0.1in}
\noindent\textbf{(Part 1.1 Conclusion)} Now Lemma~\ref{lem: estimate for denominator of conditional update map} and relations~\eqref{eq: estimate for W(muhat_k, nuhat_k), continued}, 
\eqref{eq: estimate for sum of W(mu*, nu*)}, \eqref{eq: estimate of E(A-B)^2} imply 
there is a constant $C_2 = C_2(\beta)>0$ such that
\begin{align*}
	\Eb \Big[\sum\limits_{k=1}^{n} \hW (\hat{\mu}_k, \hat{\nu}_k ) \Big] 
	& \leq C_2 \Big(\sqrt{\frac{\delta}{\delta_1}} + \epsilon \Big),\numberthis\label{eq: estimate for W(muhat_k, nuhat_k), final}
\end{align*}
which is an estimate of the first term in \eqref{eq: gmatching of muhat and nuhat}.

\vspace{0.1in}
\noindent\textbf{(Part 1.2)} Upper bound for $\Eb  I_r \Big(\hat{\mu} - \sum \hat{\mu}_k \Big)$

Let us estimate the second term on the right-hand side of \eqref{eq: gmatching of muhat and nuhat}. Since $r>\rdep$, we have
\begin{align*}
	&\Eb  I_r \Big(\hat{\mu} - \sum\limits_{k=1}^{n} \hat{\mu}_k \Big) 
	=\Eb\Big[\frac{1}{A}\sup_{u_0} \int_{\Nb \times \Rb^d} g_{r} (u-u_0) e^{ \beta Y(u)} \mu^s *\ker(du)\Big]
	\\& \leq (\Eb A^{-2})^{1/2} 
	\bigg[\Eb \Big(\sup_{u_0} \int_{\Nb \times\Rb^d} g_{r} (u-u_0) e^{\beta Y(u)}  \mu^s *\ker(du)\Big)^2\bigg]^{1/2}
	\\ & \leq 2e^{c(-2\beta)/2} 
	\left[\Eb  \sup_{u_0} \left(\int_{\Nb \times\Rb^d} g_{r} (u-u_0)^2 \mu^s *\ker(du) \int_{\Rb^d} e^{2\beta Y(u)}  \mu^s *\ker(du)\right)\right]^{1/2}
	\\& \leq 2e^{c(-2\beta)/2}  
	\left[\Eb  \int_{\Nb \times\Rb^d} e^{2\beta Y(u)}  \mu^s *\ker(du) \right]^{1/2} \left( \sup_{u_0} \int_{\Nb \times \Rb^d} g_{r} (u-u_0) \mu^s *\ker(du)\right)^{1/2}
	\\& \leq 2e^{(c(-2\beta)+c(2\beta))/2} \, I_r (\mu^s * \ker)^{1/2}  
	\\& \leq 2e^{(c(-2\beta)+c(2\beta))/2} \, I_r (\mu^s)^{1/2}
	< 2e^{(c(-2\beta)+c(2\beta))/2} \sqrt{\delta},
	\numberthis\label{eq: estimate for the second term in W(muhat, nuhat)}
\end{align*}
where we used \eqref{eq: I_r of convolution measure} in line 6.
The same upper bound holds for $\Eb  I_r \big(\hat{\nu} - \sum \hat{\nu}_k \big) $.

\vspace{0.1in}
\noindent\textbf{(Part 1 Conclusion)}
Finally, based on~\eqref{eq: Wasserstein distance between Tmu Tnu}, \eqref{eq: gmatching of muhat and nuhat}, 
\eqref{eq: estimate for W(muhat_k, nuhat_k), final} and \eqref{eq: estimate for the second term in W(muhat, nuhat)}, 
we conclude that there is a constant $C_3 = C_3 (\beta)>0$ such that
\begin{align*}
	\Wc(T\mu, T\nu) \leq \Eb d(\hat{\mu}, \hat{\nu})
	< C_3 \Big(\sqrt{\frac{\delta}{\delta_1}}  + \epsilon \Big).
\end{align*}
Choosing $\delta =\min(\delta_1 \epsilon^2, 2^{-\rdep})$ and replacing $\eps$ with $\eps/(2C_3)$, we complete the proof.

\vspace{0.2in}
\noindent\textbf{(Part 2)} We now relax the assumption that $\alpha_k, \gamma_k$ are minorized by at most one $\mu_j, \nu_j$  for each~$k$, respectively.

Our goal is to reduce the problem to that studied in \textbf{(Part 1)}. To that end, we will find $\mu'$ and~$\nu'$ in $\Xct$ such that
\begin{enumerate}[label = $\bullet$]
\item $T\mu'$ and $T\nu'$ are close to $T\mu$ and $T\nu$, respectively,
\item $\mu'$ and $\nu'$ satisfy the conditions of \textbf{(Part 1)}, i.e., $(\mu_k)$ and $(\nu_k)$ are submeasures (as elements in $\Mc_{\leq 1}$) of mutually exclusive orbits in $\mu'$ and $\nu'$, respectively.
\end{enumerate}

First, we choose $R>0$ such that $\ker(\bar{B}_R(0)^c) < \epsilon$ and decompose $\lambda$ into central and exterior parts:
	\[ \ker= \one_{\bar{B}_R(0)} \ker +  \one_{\bar{B}_R(0)^c} \ker
	:= \ker_1+  \ker_2.\]

Recalling that $n$ and $(\mu_i)_{i=1}^{n}$ were defined just before~\eqref{eq: setting (mu, nu)-triple}, let us consider $\mu' = (\alpha'_i) \in \Xc$
defined as follows:
	\[ \alpha'_i = 
	\begin{cases}
	\quad\quad\quad\quad \mu_i, & i\leq n, \\
	\alpha_{i-n} - \sum\limits_{\sss{j : \mu_j \leq \alpha_{i-n}}} \mu_j,  & i > n, \\
	\end{cases} \]
where we view $\mu_i$ as a subprobability measure on $\Rb^d$ instead of $\Nb \times \Rb^d$. In other words, we set the first $n$ layers $(\alpha_i ')_{i=1}^{n}$ in~$\mu'$ to coincide with~$(\mu_i)_{i=1}^{n}$, while the remaining layers $(\alpha_i ')_{i \geq n+1}$ are obtained from 
  $\mu - \sum_{k=1}^{n} \mu_k$ via a shift by $n$ layers. Regarding the interpretation of $\mu_k$, we refer readers to Remark~\ref{rmk:interpretation of mu_k and nu_k}.
We also define a function $ J:\{1, \cdots, n\} \rightarrow \Nb$ by
	\[\mu_k \leq \alpha_l  \quad \Leftrightarrow \quad J(k) = l.\]
	
We denote by $\hat{\mu}'$ an $\Xc$-valued random variable whose law is $T\mu'$.
To estimate $\Wc (T\mu, T\mu')$, we need to introduce a coupling of $(\hat{\mu}, \hat{\mu}')$.
To this end, let $Y, V$ be independent random fields with the same law as $X$ and let us use $Y$ to define $\hat{\mu}$.
We fix $j$ in the range of $J$. For $k \in J^{-1}(j)$,
	\[ U_k = \{ x \in \Rb^d : x \in \bar{B}_R(y) \text{\,\, for some\,\,} y \in \supp({\mu}_k)\}.\]
If we impose $\delta < 2^{-\max(\rdep, R)-R}$(i.e., $r>\max(\rdep, R)+R$), then by the definition of a $(\mu, \nu)$-triple, for all $k \neq l$,
	\[ \deuc \big(\supp(\mu_k), \supp(\mu_l)\big) \geq \sep(\phi) > 2(\max(\rdep, R) +R), \]
which implies that $\min\limits_{k \neq l} \deuc(U_k , U_l) > 2\max(\rdep, R)> \rdep$.
Due to the $\rdep$-dependence of $Y$, there is a coupling between $Y(j, \cdot)$ and \textit{i.i.d.} copies $(Y^{(k)})_{k\in J^{-1} (j)}$ such that for each $k \in J^{-1} (j)$,
\beq\label{eq: separating coupling}
	Y(j, x) = Y^{(k)} (x)   \quad \forall x \in U_k
\eeq
(we refer to Lemma~\ref{lem:coupling-using-m-dep} in the Appendix for a precise statement.)

We now set 
	\[ Y'(k, \cdot) = 
	\begin{cases}
	Y ^{(k)} & 1 \leq k \leq n, \\
	V(k, \cdot)  & k > n, \\
	\end{cases} \]
and use it to define $\hat{\mu}'$.

Combining this with \eqref{eq: separating coupling}, we have
\beq \label{eq: definition of coupling of mu and mu'}
	Y(J(k), x) = Y'(k, x) \quad \forall x \in U_k.
\eeq
Similarly to \eqref{eq: numerator of muhat} and~\eqref{eq: denominator of conditional update map}, let us define $\mu_k ^{\ast \prime}$, $\hat{\mu}'_k$, $\mu^{s \prime}$ and $A'$ 
with the environment $Y'$ and introduce for $j \in \{1, 2\}$,
	\[ \mu_{k, j} ^{\ast} = e^{\beta Y(J(k), x)}\mu_k * \ker_j(dx), \quad \hat{\mu}_{k, j}  = \mu^{\ast}_{k, j} / A, \]
	\[ \mu_{k, j} ^{\ast \prime} = e^{\beta Y'(k, x)}\mu_k * \ker_j(dx), \quad \hat{\mu}'_{k, j}  = \mu^{\ast \prime}_{k, j} / A', \]
Choosing a $(\hat{\mu}, \hat{\mu}')$-g-triple 
\beq \label{eq: choosing g-triple}
r_0 = \frac{r}{2},\ \varphi = (\hat{\mu}_{k, 1}, \hat{\mu}' _{k, 1})_{k=1}^{n},\  \vec{x} = 0,
\eeq
gives
\beq \label{est: Tmu and Tmu prime}
	\Wc(T\mu, T\mu') \leq \Eb d(\hat{\mu}, \hat{\mu}')
	\leq \Eb \Big[\sum\limits_{k=1}^{n} \hW(\hat{\mu}_{k, 1}, \hat{\mu}'_{k, 1}) + I_{r_0} \big(\hat{\mu}-\sum\limits_{k=1}^{n} \hat{\mu}_{k, 1} \big) 
	+ I_{r_0} \big(\hat{\mu}' - \sum\limits_{k=1}^{n} \hat{\mu}'_{k, 1} \big) +2^{-r_0} \Big].
\eeq

We estimate the first term on the right-hand side in \eqref{est: Tmu and Tmu prime}. 
Similarly to \eqref{eq: estimate for W(muhat_k, nuhat_k)} and \eqref{eq: estimate for W(muhat_k, nuhat_k), continued},
\beq\label{est: mhat and mhat'}
	\Eb\sum\limits_{k=1}^{n} \hW(\hat{\mu}_{k, 1}, \hat{\mu}'_{k, 1}) 
	\leq (\Eb A^{-2})^{1/2} \left(\sum\limits_{k=1}^{n} \big[ \Eb  \hW(\mu^{\ast}_{k, 1}, \mu_{k, 1}^{\ast \prime})^2\big]^{1/2} +\big[\Eb(A-A')^2\big]^{1/2}\right).
\eeq 
It follows from $\supp(\mu_k * \ker_1) \in U_k$ and \eqref{eq: definition of coupling of mu and mu'} that 
	\[ \mu^{\ast}_{k, 1} =\mu_{k, 1}^{\ast \prime} \,\,\, \mathbf{P}\text{-a.s.}, \quad \Rightarrow \quad \hW(\mu^{\ast}_{k, 1}, \mu_{k, 1}^{\ast \prime}) = 0,\]
and 
\begin{align*}
	&\Eb\left(\int \big|e^{\beta Y(J(k), x)} - e^{\beta Y'(k, x)}\big| \mu_k *\ker(dx)\right)^2
	= \Eb \left(\int \big|e^{\beta Y(J(k), x)} - e^{\beta Y'(k, x)}\big| \mu_k *\ker_2(dx)\right)^2
	\\& \leq \norm{\mu_k *\ker_2}\Eb \int \Big(e^{\beta Y(J(k), x)} - e^{\beta Y'(k, x)}\Big)^2 \mu_k *\ker_2(dx)
	\leq 2e^{c(2\beta)} \norm{\mu_k}^2 \epsilon^2.
	\numberthis\label{est: W(m ast, m ast prime)}
\end{align*}
Similarly to~\eqref{eq: split of (A-B)^2}, we have
\begin{align*}
	(A-A')^2 &\leq 3\left[ \sum\limits_{k=1}^{n} \int_{\Rb^d} \big( e^{\beta Y(J(k), x)} - e^{\beta Y'(k, x)} \big)  \mu_k *\ker (dx) \right]^2
	\\&+3\Big(\int_{\Nb \times \Rb^d} V(u) \mu^{s} *\ker (du)\Big)^2
	+3\Big( \int_{\Nb \times \Rb^d} V'(u) \mu^{s \prime}*\ker (du)\Big)^2,
	\numberthis\label{est: A-A'}
\end{align*}
where $V'(u)= e^{\beta Y'(u)} - e^{c(\beta)}$ and $\mu^{s \prime} = \mu' - \sum\mu_k$ as in \eqref{eq: sparse part of mu and nu}.
Plugging \eqref{est: W(m ast, m ast prime)} and~\eqref{eq: estimate for the sparse part in (A-B)^2, final} into \eqref{est: A-A'}, we have
\beq\label{est: E(A-A')}
	\Eb(A-A')^2 \leq 6e^{c(2\beta)}(\epsilon^2 + \delta).
\eeq
Plugging \eqref{est: W(m ast, m ast prime)} and \eqref{est: E(A-A')} into \eqref{est: mhat and mhat'} 
and using Lemma~\ref{lem: estimate for denominator of conditional update map}, we obtain
\beq\label{est: W(m hat 1 and m hat prime 1)}
	\Eb \sum\limits_{k=1}^{n} \hW(\hat{\mu}_{k, 1}, \hat{\mu}'_{k, 1})  \leq 2e^{(c(-2\beta)+c(2\beta))/2} \cdot \sqrt{6}(\epsilon + \sqrt{\delta}).
\eeq
As for the $I_{r_0}$ terms on the right-hand side of \eqref{est: Tmu and Tmu prime}, 
we repeat the computation of \eqref{eq: estimate for the second term in W(muhat, nuhat)} to obtain
\begin{align*}
	\Eb I_{r_0} \Big(\hat{\mu}-\sum\limits_{k=1}^{n} \hat{\mu}_{k, 1} \Big)  
	&\leq  \Eb  I_{r} \Big(\hat{\mu}-\sum\limits_{k=1}^{n} \hat{\mu}_{k} \Big) + \Eb  I_{r}\Big(\sum\limits_{k=1}^{n} \hat{\mu}_{k, 2} \Big)
	\\& \leq 2e^{(c(-2\beta)+c(2\beta))/2} \, \left[ I_r (\mu^s * \ker)^{1/2} + I_r ({\textstyle\sum} \mu_k*\ker_2)^{1/2}\right]
	\\& \leq 2e^{(c(-2\beta)+c(2\beta))/2} \, (\sqrt{\delta} + \sqrt{\epsilon}).
	\numberthis\label{eq: outside matching parts}
\end{align*}

Combining \eqref{est: Tmu and Tmu prime}, \eqref{est: W(m hat 1 and m hat prime 1)}, \eqref{eq: outside matching parts} and recalling that $2^{-r_0} = 2^{-r/2} < \sqrt{\delta}$ due to~\eqref{eq: choosing g-triple} and \eqref{eq: setting (mu, nu)-triple}, we conclude that
there is a constant $C>0$ such that if $\delta < \min(\delta_0 \epsilon^2, 2^{-\max(\rdep, M) - M})$, then
	\[ \Wc(T\mu, T\mu') \leq 2e^{(c(-2\beta)+c(2\beta))/2} 
	\big( \sqrt{6}(\epsilon + \sqrt{\delta}) +  2(\sqrt{\delta} + \sqrt{\epsilon}) \big) +\sqrt{\delta} < C\epsilon. \]
The same result holds for $\nu$ so that there is $\nu'$ such that $\Wc(T\nu, T\nu') < C\epsilon$.
One can check that $d_{r, \phi', \vec{x}} (\mu', \nu') = d_{r, \phi, \vec{x}} (\mu, \nu) < \delta$,
where $\phi'=\{(\alpha'_i,\gamma'_i)\}_{i=1}^n$.
 Thus,
\textbf{(Part 1)} can be now applied to estimate $\Wc(T\mu', T\nu')$. 
Finally, the triangle inequality
	\[ \Wc(T\mu, T\nu) \leq \Wc(T\mu, T\mu') + \Wc(T\mu', T\nu') + \Wc(T\nu', T\nu),\]
completes the proof in the general case.
\end{proof}
 
\subsection{Lifting the update map}
We discussed in~\eqref{eq: interpretation of conditional map as a transition kernel} that 
$T\nu(d\mu) = \Gamma (\nu, d\mu)$ can be understood as a transition kernel for the Markov chain $(\rho_i)_{i \geq 0}$ of the endpoint distributions of random polymers.
Integrating $\Gamma(\nu, d \mu)$ over the initial conditions $\nu$, we can extend~$T$  to an operator  $\Tc$ on $\Pc{(\Xct})$:
\beq\label{def: update map}
	\Tc\xi (d\mu) = \int_{\Xct}T \nu (d\mu) \xi(d\nu).
\eeq
Therefore, $\Tc$ maps the law of $\rho_i$ to the law of $\rho_{i+1}$. 
\begin{prop}\label{prop: continuity of update map}
	$\Tc:\Pc{(\Xct}) \rightarrow \Pc{(\Xct})$ is (uniformly) continuous 
\end{prop}
\begin{proof}
Let $\epsilon > 0$ and $\xi_1, \xi_2 \in \Pc (\Xct)$ be given.
We would like to show that there is $\delta >0$ such that
	\[ \Wc(\xi_1, \xi_2) < \delta \quad \Rightarrow \quad \Wc(\Tc\xi_1, \Tc \xi_2) < \epsilon. \]
By Proposition~\ref{prop: continuity of conditional update map}, there is $\delta_1 > 0$ such that
	\[ d(\mu, \nu) < \delta_1 \quad \Rightarrow \quad \Wc(T\mu, T\nu) < \epsilon/2. \]
Let us assume $\Wc(\xi_1, \xi_2) < \delta:=\delta_1 \epsilon/4$ and
let $\Pi (d\mu, d\nu) \in \Pc(\Xct^2)$ be the optimal coupling of $(\xi_1, \xi_2)$. 
Moreover, for $\mu, \nu \in \Xct$, let $\Pi_{\mu, \nu}$ be the optimal coupling of $(T\mu, T\nu)$. Then, one can check that
	\[ \Pi '(d\eta, d\tau) = \int_{\Xct \times \Xct} \Pi_{\mu, \nu}(d\eta, d\tau) \Pi(d\mu, d\nu) \]
is a coupling of $(\Tc\xi_1, \Tc \xi_2)$. Therefore, 
\begin{align*}
	&\Wc(\Tc\xi_1, \Tc \xi_2) \leq \int d(\eta, \tau) \Pi' (d\eta, d\tau)
	= \iint d(\eta, \tau) \Pi_{\mu, \nu}(d\eta, d\tau) \Pi(d\mu, d\nu)
	\\& =  \int \Wc(T\mu, T\nu) \Pi(d\mu, d\nu) 
	= \int \Wc(T\mu, T\nu) (\one_{d(\mu, \nu) \geq \delta_1} +\one_{d(\mu, \nu) < \delta_1}) \Pi(d\mu, d\nu) 
	\\& \leq \frac{2}{\delta_1} \int d(\mu, \nu) \Pi(d\mu, d\nu) + \frac{\epsilon}{2} = \frac{2}{\delta_1} \Wc(\xi_1, \xi_2) + \frac{\epsilon}{2}  
	< \epsilon.
\end{align*}
In the first inequality above, we used the fact $\Wc(T\mu, T\nu) \leq 2$ which follows from \eqref{eq: upper bound of d} 
and \eqref{def: Wasserstein metric on P(X)}. This completes the proof.
\end{proof}

\begin{rmk}\rm
Based on the definition of $\Tc$, one can readily check that $T\nu = \Tc \delta_{\nu}$.
Therefore, \eqref{def: update map} can be rewritten as 
	\[ \Tc\xi (d\mu) = \int_{\Xct}\Tc \delta_\nu (d\mu) \xi(d\nu), \]
and by iteration, one has that for $\Tc^i=\underbrace{\Tc\circ\ldots\circ\Tc}_{i}$,  $i\ge 1$,
\beq\label{eq: iteration of update map}
	\Tc^{i} \xi (d\mu) =  \int_{\Xct} \Tc^{i} \delta_\nu (d\mu) \xi(d\nu).
\eeq
\end{rmk}

\section{Convergence of empirical measure} \label{sec: convergence of empirical measure}
This section is an adaptation of Section~4 of~\cite{BC16} to our general setting.

\subsection{The convergence to fixed points of the update map}
Let us denote the empirical probability measure of the endpoint distributions on $\Xct$ by
	\[ \psi_n := \frac{1}{n} \sum\limits_{i=0}^{n-1} \delta_{\rho_i} \in \Pc(\Xct). \]
The goal of this section is to  study the asymptotic behavior of $\psi_n$.
We will prove that $\psi_n$ converges to a set $\Kc$ of fixed point of $\Tc$ and  introduce an ``energy functional'' $\Rc$ 
which maps $\psi_{n}$ to a value close to the quenched free energy $F_n$.
The functional $\Rc$ allows us to improve the former result by replacing $\Kc$ with a subset $\Kc_0$ of $\Kc$ with the minimal energy state.

\begin{prop}\label{prop: contraction}
	As $n \rightarrow \infty$, $\Wc (\psi_n, \Tc\psi_n) \rightarrow 0$\,\, $\Pb \text{-a.s.}$
\end{prop}
\begin{proof}
We use martingale analysis similar to that in the proof of Proposition 4.1 in \cite{BC16}.
Let
	\[ \mathcal{L} = \{ h : \Xct \rightarrow \Rb : |h(\mu) - h(\nu)|  \leq d(\mu, \nu) \,\,\text{for all} \,\, \mu, \nu \in \Xct,\,\, h(\mathbf{0})=0  \} \]
	and
	\[ \psi ' _n := \frac{1}{n} \sum\limits_{i=1}^{n} \delta_{\rho_i}. \]
Using $\frac{1}{n} \big(\delta_{(\rho_0, \rho_{n})} + \sum\limits_{i=1}^{n-1} \delta_{(\rho_i, \rho_i)}\big)$ as a coupling of $(\psi_n, \psi'_n)$ 
and applying \eqref{eq: upper bound of d} we conclude that $\Wc(\psi_n,\psi'_n)\le 2/n$.
Therefore, it suffices to prove that  
 \[\Wc(\psi_n ', \Tc \psi_n )\to 0. \]
It follows from \eqref{eq: Kantorovich duality} that
	\[ \Wc(\psi_n ', \Tc \psi_n ) = \sup_{h \in \mathcal{L}} \frac{M_n(h)}{n}, \]
where
	\[ M_n(h) = \sum\limits_{i=0}^{n-1} \Big(h(\rho_{i+1})-\Eb[h(\rho_{i+1})|\mathscr{G}_i]\Big)\]
is a martingale with respect to the filtration $(\mathscr{G}_n)_{n \geq 1}$.
Since $|h|\le 2$, we can apply the Burkholder--Davis--Gundy inequality to see that there is a constant $C>0$ such that
	\[ \Eb M_n(h)^4 \leq C \Eb \left(\sum\limits_{i=0}^{n-1} \Big(h(\rho_{i+1})-\Eb[h(\rho_{i+1})|\mathscr{G}_i]\Big)^2\right)^2 \leq 16Cn^2, \]
which implies $\Eb \Big(M_n (h) / n \Big)^4  \leq 16Cn^{-2}$ and hence, by the Borel--Cantelli lemma, 
\beq\label{eq: pointwise convergence of M_n/n+1}
	\lim\limits_{n \rightarrow \infty} \frac{|M_n (h)|}{n} = 0 \quad \Pb\text{-a.s.}
\eeq
On the other hand, we observe
	\[ \norm{h-h'}_\infty< \epsilon \quad \Rightarrow \quad \left| \frac{M_n (h)}{n} -\frac{M_n (h')}{n} \right| < 2\epsilon,\]
which tells us that $(M_n(\cdot)/n)_{n \geq 1}$ is an equicontinuous sequence of functions on $\mathcal{L}$. 
By the compactness of $\mathcal{L}$ and Arzela--Ascoli theorem, 
the limit in \eqref{eq: pointwise convergence of M_n/n+1} is uniform in $h \in \mathcal{L}$, which completes the proof.
\end{proof}

Proposition~\ref{prop: contraction} suggests that $(\psi_n)_{n \geq 1}$ will be close to the set of fixed points of $\Tc$ as $n$ becomes large. We denote the set of fixed points of $\Tc$ by
	\[ \Kc = \{ \xi \in \Pc(\Xct) : \Tc \xi = \xi \}. \]
Notice that $\Kc$ is nonempty since $\Tc \delta_{\mathbf{0}} = \delta_{\mathbf{0}}$.
By applying the same argument as in Corollary 4.3 and Proposition 4.4 in \cite{BC16}, we can prove the following: 
\begin{prop}\label{prop: convergence of psi to K}
	As $n \rightarrow \infty$, $\Wc(\psi_n, \Kc) := \inf \{\Wc(\psi_n, \xi) : \xi \in \Kc\} \rightarrow 0$\,\, $\mathbf{P}$\rm-a.s.
\end{prop}
\begin{proof}
Suppose that $\Wc(\psi_n, \Kc) \nrightarrow 0$. Then, there is $\epsilon>0$ and a subsequence $(\psi_{n_k})_{k \ge 1}$ such that $\Wc (\psi_{n_k}, \Kc)>\epsilon$ for all $k \ge 1$.
Since $\Pc (\Xct)$ is compact, we may assume $\lim\limits_{k \rightarrow \infty} \psi_{n_k} = \psi$ for some $\psi \in \Pc (\Xct)$, if needed, by choosing a further subsequence.
On the other hand,
	\[ \Wc(\psi, \Tc \psi) \leq  \Wc(\psi, \psi_{n_k}) + \Wc(\psi_{n_k},  \Tc \psi_{n_k}) +\Wc(\Tc \psi_{n_k}, \Tc \psi ), \]
and as $k \rightarrow \infty$, each term in the right-hand side converges to $0$ due to the  convergence of $\psi_{n_k}$, Proposition~\ref{prop: contraction}, and continuity of $\Tc$, respectively.
It follows $\Tc \psi = \psi$, i.e., $\psi \in \Kc$, which contradicts the assumption that $\psi_{n_k}$ is $\epsilon$-away from $\Kc$.
\end{proof}

\begin{prop}\label{prop: 0-1 property of fixed point set K}
	If $ \xi \in \Kc$,\, then $\xi(\{\mu \in \Xct : 0 < \norm{\mu} < 1\}) = 0.$
\end{prop}

\begin{proof}
Let $\xi \in \Kc$ and let $\mu$ be a $\Xct$-valued random variable whose law is $\xi$. Let us suppose, to derive a contradiction, that 
	$\xi(\{\mu \in \Xct : 0 < \norm{\mu} < 1\}) > 0.$
Recalling that $T\mu$ is the law of the $\Xct$-valued random variable $\hat{\mu}$ defined as in~\eqref{eq: conditional update map, 2nd version}, 
we have, by Jensen's inequality applied to the concave function $x\mapsto x/(x+(1-\norm{\mu})e^{c(\beta)})$,
\begin{align*}
	\Eb \norm{\hat{\mu}} &= \Eb \frac{\int_{\Nb\times \Rb^d} e^{\beta Y(u)}\, \mu * \ker (du)}
	{ \int_{\Nb\times \Rb^d} e^{\beta Y(w)} \mu *\ker (dw) + (1-\norm{\mu})e^{c(\beta)}} 
	\\& \leq \frac{\Eb \int_{\Nb\times \Rb^d} e^{\beta Y(u)}\, \mu * \ker (du)}
	{\Eb  \int_{\Nb\times \Rb^d} e^{\beta Y(w)} \mu *\ker (dw) + (1-\norm{\mu})e^{c(\beta)}}
	\leq  \frac{ e^{c(\beta)} \norm{\mu}}{e^{c(\beta)}} = \norm{\mu},
\end{align*}
where identity holds if and only if $\int_{\Nb\times \Rb^d} e^{\beta Y(u)}\, \mu * \ker (du)$ is a constant $\Pb$-a.s. 
However, since $Y$ is non-degenerate, we have a strict inequality. Combining this fact and the assumption that we made, we see that
	\[ \int \norm{\hat{\mu}} \Tc\xi(d\hat{\mu}) = \iint \norm{\hat{\mu}} T\mu (d\hat{\mu}) \xi(d\mu) 
	< \iint \norm{\mu} T\mu (d\hat{\mu}) \xi(d\mu) =\int \norm{\mu} \xi(d\mu),\]
which is a contradiction since $\Tc \xi = \xi$.
\end{proof}

\subsection{Variational formula for the free energy}
We observe that
	\[F_n(\beta) := \frac{1}{n} \log Z_n(\beta) = \frac{1}{n} \sum\limits_{i=0}^{n-1} \log \frac{Z_{i+1}}{Z_i}  
	= \frac{1}{n} \sum\limits_{i=0}^{n-1} \log\Big(\int_{\Rb^d} e^{\beta X(i+1, x)} \rho_i *\ker (dx)\Big). \]
Conditioning the $i$-th term on $\mathscr{G}_i$, we have
\beq \label{eq: exepctation of F_n}
	\Eb F_n = \frac{1}{n} \sum\limits_{i=0}^{n-1} \Eb R(\rho_i),
\eeq
where 
	\[ R(\rho_i):=\Eb \bigg[\log\Big(\int_{\Rb^d} e^{\beta X(i+1, x)} \rho_i *\ker (dx)\Big) \Big| \mathscr{G}_i \bigg]. \]
It is useful to extend this to a functional on $\Xct$ as follows:  
\beq\label{eq: definition of R}
	R(\mu) := \Eb \log\Big(\int_{\Nb \times \Rb^d} e^{\beta Y(u)} \mu * \ker(du) +(1-\norm{\mu})e^{c(\beta)}\Big),\quad \mu \in \Xct,
\eeq
where $Y$ has the same law as $X$.
\begin{prop}
\label{prop: continuity of R} 
	Let $R : (\Xct, d) \rightarrow (\Rb, |\cdot|)$  be defined by \eqref{eq: definition of R}. Then $R$ is well-defined and uniformly continuous. 
\end{prop}

\begin{proof}
It is easy to check that the right-hand side of \eqref{eq: definition of R}  does not depend on the choice of the representative of $\mu$. We need to prove that $R(\mu)$ is finite.
For any positive random variable $K$, one has
\beq\label{eq: log estimate}
	|\Eb \log K | \leq \max \{\log \Eb K, \log \Eb K^{-1}\} = \log \big(\max\{ \Eb K, \Eb K^{-1}\}\big)
\eeq
by Jensen's inequality.
On the other hand, we see that
\beq\label{eq: EK}
	\Eb \Big[\int_{\Nb \times \Rb^d} e^{\beta Y(u)} \mu * \ker(du) +(1-\norm{\mu})e^{c(\beta)} \Big] = e^{c(\beta)}
\eeq
and by Lemma \ref{lem: estimate for denominator of conditional update map},
\beq \label{eq: EK^(-1)}
	\Eb\Big(\int_{\Nb \times \Rb^d} e^{\beta Y(u)} \mu * \ker(du)
	+(1-\norm{\mu})e^{c(\beta)}\Big)^{-1} \leq 2 e^{c(-\beta)}.
\eeq
Combining \eqref{eq: log estimate}, \eqref{eq: EK}, and \eqref{eq: EK^(-1)} gives $|R(\mu)| < \infty$.
\par We now prove the uniform continuity of $R$:
given any $\eps>0$, there is $\delta = \delta(\eps)>0$ such that
	\[ d(\mu, \nu) < \delta  \quad \Rightarrow \quad |R(\mu)- R(\nu)| < N\epsilon,\]
where $N>0$ is a constant depending only on $\beta$ and the law of $X(1, 0)$. \\
Let $\eps>0$ be given.
Let us define $A, B$ by~\eqref{eq: denominator of conditional update map} and~\eqref{eq: denominator of nuhat} and choose $\delta$ used in the proof of Proposition~\ref{prop: continuity of conditional update map} (see the last line of Part 1). Notice that for any $x>0$,
	\[ |\log x| \leq |x-1|+ \Big|\frac{1}{x}-1\Big|. \]
Combining this with Lemma~\ref{lem: estimate for denominator of conditional update map} and~\eqref{eq: estimate of E(A-B)^2}, we have
\begin{align*}
	| R(\mu)-R(\nu)| &= \Big| \Eb \log \frac{A}{B} \Big| \leq \Eb\Big|\frac{A}{B} -1 \Big|+ \Eb \Big|\frac{B}{A} -1 \Big| 
	\\& \leq \big((\Eb B^{-2})^{1/2} + (\Eb A^{-2})^{1/2}\big) \big(\Eb (A-B)^2 \big)^{1/2} \leq e^{c(-2\beta)} C_1 \epsilon,
\end{align*}
 which completes the proof.
\end{proof}

\begin{rmk}\label{rmk: maximum of R}
We can apply Jensen's inequality to \eqref{eq: definition of R} to obtain
	\[ R(\mu) \leq \log \Eb \Big[\int e^{\beta Y(u)} \mu * \ker (du) + (1-\norm{\mu}) e^{c(\beta)} \Big] 
	= \log \int \Eb e^{\beta Y(u)} \, \mu * \lambda(du) + (1-\norm{\mu}) e^{c(\beta)} = c(\beta).\]
Since $Y$ is non-degenerate, the identity holds only if $\norm{\mu}=0$, i.e., $\mu = \zero$.
Hence, the functional $R$ attains its unique maximum at $\zero$.
\end{rmk}

\begin{prop} \label{prop: continuity of mathcal R}
 The map $\Rc: (\Pc(\Xct),\Wc) \to(\Rb, |\cdot|)$ defined by
	\[\Rc (\xi):= \int R(\mu) \xi(d\mu), \quad \xi \in \Pc(\Xct),\]
is uniformly continuous.
\end{prop}

\begin{proof}
Let $\epsilon > 0$ be given and  let $N := \max\limits_{\mu \in \Xct} |R (\mu)|$. It is sufficient to show that there is $\delta>0$ such that
	\[ \Wc(\xi_1, \xi_2) < \delta  \quad \Rightarrow \quad |\Rc(\xi_1)- \Rc(\xi_2)| < (2N+1) \epsilon. \]
By Proposition~\ref{prop: continuity of R}, there is $\delta_1 > 0$ such that
	\[ d(\mu, \nu) < \delta_1 \quad \Rightarrow \quad |R(\mu)- R(\nu)| < \epsilon. \]
Set $\delta = \delta_1 \eps$ and let $\Pi$ be the optimal coupling of $(\xi_1, \xi_2)$. Then, for any $\xi_1, \xi_2 \in \Pc(\Xct)$ with $\Wc(\xi_1, \xi_2) < \delta$, we have
\begin{align*}
	&|\Rc (\xi_1) - \Rc (\xi_2)| \leq \int |R(\mu)- R(\nu)| \Pi (d\mu, d\nu)
	\\& = \int |R(\mu) - R(\nu)| (\one_{d(\mu, \nu) \geq \delta_1} +\one_{d(\mu, \nu) < \delta_1}) \Pi(d\mu, d\nu) 
	\\& \leq \frac{2N}{\delta_1} \int d(\mu, \nu) \Pi(d\mu, d\nu) + \epsilon = \frac{2N}{\delta_1} \Wc(\xi_1, \xi_2) + \epsilon < (2N+1)\epsilon,
\end{align*}
completing the proof.
\end{proof}

We can rewrite \eqref{eq: exepctation of F_n} as

	\[ \Eb F_n = \frac{1}{n} \sum\limits_{i=0}^{n-1} \Eb R(\rho_i) = \Eb \big[ \Rc(\psi_{n})\big]. \]
In fact, not only the expectations but the random variables themselves are close:

\begin{prop}\label{prop: convergence of quenched energy}
	As $n \rightarrow \infty$, $F_n-\Rc(\psi_{n}) \rightarrow 0$ \, $\mathbf{P}$-a.s.
\end{prop}

\begin{proof}
Let
	\[ V_i = \int_{\Rb^d } e^{\beta X(i+1, x)} \rho_i * \ker (dx),  \quad U_i = \log V_i. \]
We have $ \Eb [U_i |\mathscr{G}_i ] = R(\rho_i)$ and therefore
	\[ M_n := n \big(F_n-\Rc(\psi_{n}) \big) = \sum_{i=0}^{n-1} \Big(U_i - \Eb [U_i | \mathscr{G}_i ]\Big) \]
is a martingale.

We claim that $\Eb\big(U_i-\Eb[U_i | \mathscr{G}_i]\big)^4 $ is bounded. It suffices to show that $\Eb U_i^4$ is bounded. To this end, we observe that
\begin{align*}
	\Eb V_i ^2 \leq \Eb\int_{\Rb^d }e^{2\beta X(i+1, x)} \rho_i *\ker (dx)= e^{c(2\beta)}
\end{align*}
Similarly, we obtain $\Eb [V_i ^{-2} ] \leq  e^{c(-2\beta)}$.
Using the inequality $(\log x)^4 < x^2 + 1/x^2$, we have
	\[ \Eb U_i^4 <  \Eb V_i^2 +  \Eb V_i^{-2} \leq e^{c(2\beta)} + e^{c(-2\beta)}. \]
Using the Burkholder--Davis--Gundy inequality, we obtain
	\[ \Eb M_n ^4 \leq C_2 \Eb \Big(\sum_{i=0} ^{n-1} (U_i - \Eb [U_i | \mathscr{G}_i ] ) ^2 \Big)^2 \leq Cn^2, \]
which implies $\Eb (M_n / n)^4  \leq Cn^{-2}$ and hence, by the Borel--Cantelli lemma, 
	\[ \lim_{n \rightarrow \infty} |F_n - \Rc (\psi_{n}) | = \lim_{n \rightarrow \infty} \frac{|M_n|}{n} = 0 \quad \mathbf{P}\text{-a.s.}, \]
	completing the proof.
\end{proof}

\subsection{A representation of convergence of $F_n$ via $\Rc$ and $\Tc$}
In this section, we explore how energy functional $\Rc$ and $\Tc$ are related to the quenched free energy $F_n$. First, we show that the limit of free energy can be understood as the minimal energy state among the set $\Kc$ of fixed points of $\Tc$, see Theorem~\ref{thm: relation between F_n  and R}. This result allows us to describe more precisely the asymptotic behavior of empirical measures $\psi_n$ previously stated in Proposition~\ref{prop: convergence of psi to K}.

\begin{thm}\label{thm: relation between F_n  and R} 
 \[p(\beta) = \lim_{n \rightarrow \infty}  F_n = \inf_{\xi \in \Kc} \Rc(\xi) \quad \Pb\text{-a.s.} \]
\end{thm}

We already discussed in \eqref{eq: convergence of quenched free energy} that $|F_n - \Eb F_n| \rightarrow 0$ almost surely and in $L_p$ for all $p \geq 1$.
Hence, it is sufficient to show 
	\[ \lim\limits_{n \rightarrow \infty} \Eb F_n = \inf\limits_{\xi \in \Kc} \Rc(\xi). \]
To that end, we need two following propositions.

\begin{prop}
 \[\liminf_{n \rightarrow \infty} \Eb F_n \geq \inf_{\xi \in \Kc} \Rc(\xi)\quad \mathbf{P}\text{-a.s.} \]
\end{prop}

\begin{proof}
By Proposition~\ref{prop: convergence of psi to K} and~\ref{prop: continuity of mathcal R}, 
	\[ \liminf _{n \rightarrow \infty} \Rc (\psi_{n}) \geq \inf_{\xi \in \Kc} \Rc(\xi) \quad \mathbf{P}\text{-a.s.} \]
The conclusion follows from Proposition \ref{prop: convergence of quenched energy} and Fatou's Lemma.
\end{proof}

\begin{prop}
	\[\limsup_{n \rightarrow \infty} \Eb F_n \leq \inf_{\xi \in \Kc} \Rc(\xi)\quad \mathbf{P}\text{-a.s.} \]
\end{prop}

\begin{proof}
We claim that for any $\mu^{(0)} \in \Xct$ and $n \in \Nb$,
\beq \label{eq: comparing R to E log Z}
	\sum_{i=0}^{n-1} \Rc \big(\Tc^i \delta_{\mu^{(0)}} \big) \geq \Eb \log Z_n.
\eeq
First, let us use this claim to derive the proposition. For any $\xi \in \Pc(\Xct)$, it follows from \eqref{eq: iteration of update map} that
\begin{align*}
	\Rc(\Tc^{i} \xi) = \int R(\nu) \Tc^{i} \xi (d\nu) = \iint R(\nu) \Tc^{i} \delta_{\mu} (d\nu) \xi(d\mu)	
	 = \int \Rc(\Tc^{i} \delta_{\mu}) \xi(d\mu).
\end{align*}
Moreover, if $\xi \in \Kc$, then \eqref{eq: comparing R to E log Z} implies
	\[ \Rc (\xi) =\frac{1}{n} \sum\limits_{i=0}^{n-1} \Rc(\Tc^i \xi ) = \int \frac{1}{n} \sum\limits_{i=0}^{n-1} \Rc(\Tc^i \delta_{\mu}) \xi(d \mu) 
	\geq  \int \frac{\Eb \log Z_n}{n}  \xi(d \mu) = \Eb F_n. \]
Taking infimum over $\xi \in \Kc$ and $\limsup\limits_{n \rightarrow \infty}$, we complete the proof.

Let us prove the claim~\eqref{eq: comparing R to E log Z}. For $i \geq 1$, let $\mu^{(i)}$ be defined inductively by
\beq\label{eq: mu_i}
	 \mu ^{(i)}(du) = \frac{\int_{\Nb \times \Rb^d}  e^{\beta Y^{(i)} (u)} \mu ^{(i-1)} * \ker (du)}
		{\int_{\Nb \times \Rb^d} e^{\beta Y^{(i)} (v)} \mu ^{(i-1)}* \ker (dv) 
		+ (1-\|\mu^{(i-1)}\|)e^{c(\beta)}},
\eeq
where $(Y^{(i)})$ are \textit{i.i.d.} random fields whose law is the same as $X$. 
By induction, we see that the law of $\mu^{(i)}$ is $\Tc^i \delta_{\mu^{(0)}}$. Hence,
	\[ \Rc(\Tc^i \delta_{\mu^{(0)}}) 
	= \Eb \log\Big(\int_{\Nb \times \Rb^d }e^{\beta Y^{(i+1)}(u)}\mu^{(i)}*\ker (du)
	+(1-\|\mu^{(i)}\|)e^{c(\beta)}\Big)= \Eb \log D_{i+1}, \]
where
\begin{align*}
	D_i &= \int_{\Nb \times \Rb^d} e^{\beta Y^{(i)} (u)}\mu^{(i-1)} *\ker (du)
	+(1-\|\mu^{(i-1)}\|)e^{c(\beta)}
	\\& = \int_{(\Nb \times \Rb^d)^2 }e^{\beta Y^{(i)} (u)} P_1 ^v (du) \mu^{(i-1)} (dv)
	+(1-\|\mu^{(i-1)}\|)e^{c(\beta)},\quad i\ge 0.
	\numberthis\label{eq: D_i}
\end{align*}
Here,  we can consider $P_k ^u$ as a probability measure on $(\Nb\times\Rb^d)^k$ 
by extending $P_k ^x$ defined in \eqref{def: finite dimensional distribution of P} as follows: for any $u_0=(m_0, x_0) \in \Nb \times \Rb^d$,
	\[ P_k^{u_0}(du_1, \cdots, du_k) 
	=\one_{\{m_0 = m_1 = \cdots = m_k\}} P_k^{x_0} (dx_1, \cdots, dx_k), \quad u_j = (m_j, x_j) \in \Nb \times \Rb^d, j =1, 2, \cdots, k.\]
Notice that
	\[ \sum\limits_{i=0}^{n-1} \Rc (\Tc^i \delta_{\mu^{(0)}}) = \Eb \log \prod_{i=1}^{n} D_i . \]
Iterating \eqref{eq: mu_i} for $i=n-1, \cdots, 1$, we have  
\begin{align*}
	& \int_{(\Nb \times \Rb^d)^2 } e^{\beta Y^{(n)}(u_n)} P_1 ^{u_{n-1}} (du_n) \mu^{(n-1)} (du_{n-1})
	\\& = \frac{1}{D_{n-1}} \int_{(\Nb \times \Rb^d)^3 } e^{\beta (Y^{(n)}(u_n)+Y^{(n-1)}(u_{n-1}))}
	P_2 ^{u_{n-2}}(du_{n-1}, du_n) \mu^{(n-2)}(du_{n-2})
	\\ & = \,\, \cdots
	\\& = \frac{1}{\prod_{i=1}^{n-1} D_i} \int_{\Nb \times \Rb^d} \mu^{(0)}(du_0) \int_{(\Nb \times \Rb^d)^{n}} \exp \Big( \sum\limits_{i=1}^{n} \beta Y^{(i)} (u_i) \Big) 
	P_n ^{u_0} (d\vec{u}),
	\numberthis\label{eq: iteration of the first part of D_i}
 \end{align*}
 where $\vec{u} = (u_1, \cdots, u_n)$.
Integrating over $u$ in \eqref{eq: mu_i}, we have
\begin{align*}
1-\|\mu^{(i)}\| =  \frac{\big(1-\|\mu^{(i-1)}\| \big) e^{c(\beta)} }{D_{i}}.
\end{align*}
Iterating this relation for $i= n-1, \cdots, 1$, we obtain
\beq\label{eq: iteration of the second part of D_i}
	\big(1-\|\mu^{(n-1)}\| \big) e^{c(\beta)} 
	= \frac{\big(1-\|\mu^{(0)}\|\big) e^{nc(\beta)} }{\prod\limits_{i=1}^{n-1} D_{i}} 
	= \frac{\big(1-\|\mu^{(0)}\|\big) \Eb Z_n }{\prod\limits_{i=1}^{n-1} D_{i}}.
\eeq
Combining \eqref{eq: D_i} for $i=n$, \eqref{eq: iteration of the first part of D_i}, and \eqref{eq: iteration of the second part of D_i}, we obtain
\begin{align*}
	\prod_{i=1}^{n} D_i =  \int \mu^{(0)} (du_0) \int\exp \Big(\sum\limits_{i=1}^{n} \beta Y^{(i)} (u_i) 
	\Big) P_n ^{u_0} (d\vec{u})
	+ \big(1-\|\mu^{(0)}\|\big) \Eb Z_n,
\end{align*}
and by Jensen's inequality,
\begin{align*}
	\log \prod_{i=1}^{n} D_i 
	 &\geq \int  \mu^{(0)}(du_0) \log \bigg( \int  \exp\Big(\sum\limits_{i=1}^{n} \beta Y^{(i)} (u_i)\Big) P_n ^{u_0} (d\vec{u})\bigg)
	  + \big(1-\|\mu^{(0)}\|\big) \log \Eb Z_n
	\\& \geq  \int \mu^{(0)} (du_0)  \log \bigg( \int  \exp\Big(\sum\limits_{i=1}^{n} \beta Y^{(i)} (u_i)\Big) P_n ^{u_0} (d\vec{u})\bigg)
	+ \big(1-\|\mu^{(0)}\| \big) \Eb \log Z_n
	\\& =  \int \mu^{(0)} (du_0) \log Z_{n, u_0} + \big(1-\|\mu^{(0)}\| \big) \Eb \log Z_n.
	\numberthis\label{eq: log of prod D_i}
\end{align*}
Since $\big(Y^{(i)} (j, x) \big)_{j \geq 1}$ is stationary in $x$,  $ \log Z_{n, u} \overset{d}{=}  \log Z_{n}$ for all $u \in \Nb \times \Rb^d$. In particular,
$\Eb \log Z_{n, u} =  \Eb \log Z_{n}$. Taking expectation on the both sides of \eqref{eq: log of prod D_i}, we obtain
\begin{align*}
	\sum_{i=0}^{n-1} \Rc (\Tc^i \delta_{\mu^{(0)}}) = \Eb \Big[\log \prod_{i=1}^{n} D_i \Big] 
	 \geq \int \mu^{(0)} (du_0) \Eb \log Z_n + \big(1- \|\mu^{(0)}\|\big) \Eb \log Z_n = \Eb \log Z_n,
\end{align*}
completing the proof of \eqref{eq: comparing R to E log Z} and the entire proposition.
\end{proof}

Let us denote
	\[ \Kc_0 = \{\xi_0 \in \Kc : \Rc(\xi_0) = \inf_{\xi \in \Kc} \Rc(\xi)\}. \]
Since $\Rc$ is continuous on the compact space $\Pc(\Xct)$, the infimum is attained and $\Kc_0$ is also compact.
Proposition~\ref{prop: convergence of psi to K}, Proposition~\ref{prop: convergence of quenched energy}, and Theorem~\ref{thm: relation between F_n  and R} suggest that 
one can strengthen Proposition~\ref{prop: convergence of psi to K} by taking a subset $\Kc_0$ of $\Kc$.
\begin{thm}\label{thm: convergence of psi to K0}
	As $n \rightarrow \infty$, $\Wc(\psi_n, \Kc_0) \rightarrow 0$\,\, $\mathbf{P}$-a.s.
\end{thm}
We omit the proof. It is identical to that of Theorem 4.11 in \cite{BC16} and is based on compactness of $\Kc$, continuity of $\Rc$,
Proposition~\ref{prop: convergence of quenched energy}, and Theorem~\ref{thm: relation between F_n  and R}.

\section{Characterization of high/low temperature regimes} \label{sec: characterization of high/low temperature regime}

\subsection{Existence of phase transitions}

We recall that the critical inverse temperature $\beta_c$ was introduced in Theorem~\ref{thm: critical beta}.

\begin{thm}\label{thm: chracterstic of K}
\mbox{}
\begin{enumerate}[label=\rm(\alph*)]
	\item $0 \leq \beta \leq \beta_c$, then $\Kc = \Kc_0 = \{\delta_{\mathbf{0}}\}$.
	\item If $\beta > \beta_c$,  then $\Kc$ has an element other than~$\delta_{\mathbf{0}}$.  
	In this case, $\xi (U) := \xi \big(\{\mu \in \Xct : \norm{\mu} =1\} \big) = 1$ for all $\xi \in \Kc_0$. 
\end{enumerate}
\end{thm}

The proof below follows the proof of Theorem 5.2 in \cite{BC16} closely.

\begin{proof}
It is sufficient to prove the inverses of (a) and (b) because their hypotheses are complementary.
We recall from Theorem~\ref{thm: critical beta} that $0 \leq \beta \leq \beta_c$ is equivalent to $\Lambda (\beta) = c(\beta) - p(\beta) = 0$, 
where $\Lambda$, $c$ and $p$ were defined in \eqref{eq: Lyapunov exponent}, \eqref{eq: convergence of E F_n} and \eqref{def: logarithmic moment generating function}, respectively.

If $\Kc$ has no elements other than $\dzero$, i.e., $\Kc = \Kc_0 = \{\dzero\}$, then Theorem~\ref{thm: relation between F_n  and R} tells us that
	\[ p(\beta) = \lim_{n \to \infty} \Eb F_n  = \inf_{\xi \in \Kc} \Rc (\xi) = \Rc (\dzero) = R(\zero) = c(\beta),\]
which implies $0 \leq \beta \leq \beta_c$.

Let us assume that there is an element $\zeta \in \Kc$ which is different from $\dzero$. 
From Remark~\ref{rmk: maximum of R}, we have $R(\mu) <R(\zero)$ for all $\mu \neq \zero$. 
Combining this with the fact $\zeta(\{\zero\}) <1$, we see that
	\[ p(\beta) = \lim_{n \to \infty} \Eb F_n \leq \Rc (\zeta) = \int R(\mu) \zeta(d\mu) < R(\zero) = c(\beta),\]
which implies $\beta>\beta_c$.

To see that $\xi (U) = 1$ for all $\xi \in \Kc_0$ in (b), fix $\zeta \in \Kc \setminus \{\zero\}$ and let us consider a conditional probability measure on~$\Xct$ given by
	\[ \zeta_U (A)= \frac{\zeta(A \cap U)}{\zeta(U)} \quad \text{ for all Borel } A \subset \Xct.\]
We claim that $\zeta_U \in \Kc$.  To prove this, first, we notice that due to the presence of $(1-\norm{\mu})e^{c(\beta)}$ term in the denominator of~\eqref{eq: conditional update map, 2nd version},
	\[ \mu \in U \,\,\Rightarrow \,\, T\mu (U) = 1, \quad \mu \notin U \,\,\Rightarrow\,\, T\mu(U) =0. \]
Therefore, for any Borel $A \subset \Xct$,
\begin{align*}
\Tc \zeta_U (A) &= \frac{1}{\zeta (U)} \int_U T\mu (A) \,\zeta (d\mu) =\frac{1}{\zeta (U)} \int_U T\mu (A \cap U) \,\zeta (d\mu)
\\&= \frac{1}{\zeta (U)} \int_{\Xct} T\mu (A \cap U) \,\zeta (d\mu) = \frac{\zeta(A \cap U)}{\zeta(U)} = \zeta_U (A),
\end{align*}
which proves the claim.\\
If $\zeta(U)<1$, then
\begin{align*}
	\Rc(\zeta_U) & = \frac{1}{\zeta(U)}\int_{U} R(\mu) \zeta(d\mu) = \int_{U} R(\mu) \zeta(d\mu) + \frac{1-\zeta(U)}{\zeta(U)}\int_{U} R(\mu) \zeta(d\mu)
	\\& < \int_{U} R(\mu) \zeta(d\mu) + (1-\zeta(U)) R(\zero) = \Rc(\zeta),
\end{align*}
which implies that $\zeta \in \Kc_0$ only if $\zeta(U)=1$.
\end{proof}

\begin{lem}\label{lem: continuity of I_r}
For any $r \geq  0 $, the map $I_r: \Xct \rightarrow [0, 1]$, defined in \eqref{eq: concentration function}, is continuous.
\end{lem}

\begin{proof}
Let $\epsilon>0$, $r \geq 0$ and $\mu \in\Xct$ be given. We need to show that there is $\delta(\epsilon, r) > 0$ such that
	\[ d(\mu, \nu) < \delta  \quad \Rightarrow \quad  |I_r (\mu) - I_r (\nu)| < \epsilon. \]
We set $\delta = \min (\epsilon/2, 2^{-r})$ and take any $\nu \in \Xct$ such that $d(\mu, \nu) < \delta$. 
There is a $(\mu, \nu)$-triple $\big(r_0, \phi = \{(\mu_k, \nu_k)\}_{k=1}^{n}, \vec{x} \big)$ such that $d_{r_0, \phi, \vec{x}} (\mu, \nu) <\delta$. 
By the choice of $\delta$, we have $r_0 > r $.
From the subadditivity of $I_r$, we have
	\[ I_r \Big(\sum\limits_{k=1}^{n} \mu_k \Big) \leq  I_r (\mu) \leq I_r \Big(\sum\limits_{k=1}^{n} \mu_k \Big) + I_r \Big(\mu - \sum\limits_{k=1}^{n} \mu_k \Big) \]
	\[ I_r  \Big(\sum\limits_{k=1}^{n} \nu_k \Big) \leq  I_r (\nu) \leq I_r  \Big(\sum\limits_{k=1}^{n} \nu_k \Big) + I_r \Big(\nu - \sum\limits_{k=1}^{n} \nu_k \Big).\]
It follows from $\sep(\phi) > 2r_0 > 2r$ that
	\[ \Big|I_r \Big(\sum\limits_{k=1}^{n} \mu_k \Big) - I_r \Big(\sum\limits_{k=1}^{n} \nu_k \Big) \Big|  = \big|\sup_k I_r (\mu_k) - \sup_k I_r (\nu_k) \big|. \]
Combining all these estimates, we conclude that
\begin{align*}
	 |I_r (\mu) - I_r (\nu)| &\leq \Big| I_r \Big( \sum \mu_k \Big) - I_r \Big(\sum \nu_k \Big)\Big| 
	+ \max \Big\{I_{r} \Big(\mu - \sum\limits_{k=1}^{n} \mu_k \Big), I_{r} \Big(\nu - \sum\limits_{k=1}^{n} \nu_k \Big) \Big\}
	\\& < \big|\sup_k I_r (\mu_k) - \sup_k I_r (\nu_k) \big| +  \delta \leq \sup_k \big|I_r(\mu_k) - I_r (\nu_k) \big| +\delta
	\\&  \leq \sup_k W(\mu_k, \nu_k*\delta_{x_k}) + \delta < 2 \delta \leq \epsilon,
\end{align*}
which completes the proof.
\end{proof}

\begin{thm}\label{thm: localization of directed polymers}
\mbox{}
\begin{enumerate}[label=\rm(\alph*)]
	\item If $0 \leq \beta \leq \beta_c$, then
	\[ \lim_{n\rightarrow \infty} \frac{1}{n} \sum_{i=0}^{n-1}\sup_{x \in \Rb^d} \rho_i \big(B_1(x) \big) = 0 \quad \Pb \text{-a.s.} \]

	\item If $\beta > \beta_c$, then there is $c>0$ such that
	\[	 \liminf_{n\rightarrow \infty} \frac{1}{n} \sum_{i=0}^{n-1}\sup_{x \in \Rb^d} \rho_i \big(B_1 (x) \big) \geq  c \quad \Pb\text{-a.s.}	\]
\end{enumerate}
\end{thm}

\begin{proof} 
By the (uniform) continuity of the map $\xi \mapsto \int I_r (\mu) \xi(d\mu)$ on the compact space $\Pc(\Xct)$, for any $\epsilon>0$, we can choose some $\delta>0$ such that
	\[ \Wc (\zeta, \Kc_0) <\delta \quad  \Rightarrow \quad
	\inf_{\xi \in \Kc_0} \int I_r (\mu) \xi(d \mu) - \epsilon \leq \int I_r (\mu) \zeta(d \mu) \leq \sup_{\xi \in \Kc_0} \int I_r (\mu) \xi(d \mu) + \epsilon. \]
Theorem~\ref{thm: convergence of psi to K0} implies that 
\begin{align*}
	\inf_{\xi \in \Kc_0} \int I_r (\mu) \xi(d \mu) &\leq \liminf_{n \rightarrow \infty} \int I_r (\mu) \psi_{n} (d \mu) 
	\\&\leq \limsup_{n \rightarrow \infty} \int I_r (\mu) \psi_{n} (d \mu) \leq \sup_{\xi \in \Kc_0} \int I_r (\mu) \xi(d \mu).\numberthis\label{eq: squeez}
\end{align*}
If $0 \leq \beta \leq \beta_c$ (or equivalently $\Kc_0 = \{\delta_{\mathbf{0}}\}$),
we have $\sup\limits_{\xi \in \Kc_0} \int I_1 (\mu) \xi(d \mu) = 0$  and together with~\eqref{eq: relation between I_r and the heaviest ball}, we obtain
	\[ \limsup_{n\rightarrow \infty} \frac{1}{n} \sum_{i=0}^{n-1}\sup_{x \in \Rb^d} \rho_i (B_1(x)) 
	\leq \limsup\limits_{n \rightarrow \infty} \int I_{1} (\mu) \psi_{n} (d \mu) = 0. \]
If $\beta > \beta_c$, by Theorem~\ref{thm: chracterstic of K}, every $\xi \in \Kc_0$ is supported on $\{\mu \in \Xct : \norm{\mu}=1\}$. Therefore, we have
	\[ \int I_{0} (\mu) \xi(d\mu) > 0 \text{\quad for all \,\,} \xi \in \Kc_0.\]
This and compactness of $\Kc_0$ imply that there is $c>0$ such that
	\[ \inf_{\xi \in \Kc_0} \int I_{0} (\mu) \xi(d\mu) \geq c. \]
Combining this with \eqref{eq: squeez} and \eqref{eq: relation between I_r and the heaviest ball} completes the proof of (b).
\end{proof}

\section{Asymptotic clustering } \label{sec: clustering}

\subsection{Definitions and sufficient conditions}
\begin{definition}\label{def:cluster-at-level}
The sequence $(\rho_i)_{i \geq 0}$ of the endpoint distributions is said to be \textit{``asymptotically clustered at level $r>0$''} 
if for every sequence $(\epsilon_i)_{i \geq 0}$ tending to $0$, we have
	\[ \lim_{n\rightarrow \infty}  \frac{1}{n}  \sum_{i=0}^{n-1} \rho_i \big(\Ac_{i} ^{\epsilon_i} (r)\big) = 1 \quad \mathbf{P}\text{-a.s.},\]
where $\Ac_i ^{\epsilon} (r) = \{ x \in \Rb^d : \rho_i (B_r (x)) > \epsilon V_d r^d\}$ and $V_d$ is the volume of the unit ball in $\Rb^d$. 
\end{definition}

\begin{definition}\label{def:asymptotic-local-clustering}
We say that $(\rho_i)_{i \geq 0}$ is \textit{``asymptotically locally clustered''} if for every sequence $(\epsilon_i)_{i \geq 0}$ tending to $0$, we have
 	\[\lim_{n\rightarrow \infty}  \frac{1}{n}  \sum_{i=0}^{n-1} \rho_i (\Ac_{i} ^{\epsilon_i}) = 1 \quad \mathbf{P}\text{-a.s.},\]
	where $\Ac_i ^{\epsilon} = \{ x \in \Rb^d : \liminf\limits_{r \downarrow 0} \frac{\rho_i (B_r (x))}{V_d r^d} > \epsilon\}$.
\end{definition}

\begin{definition}\label{def:clustering-densities}
We say that {\it asymptotic  clustering of densities} holds for $(\rho_i)_{i \geq 0}$ 
if every $\rho_i$ is absolutely continuous with respect to the Lebesgue measure and 
for every sequence $(\epsilon_i)_{i \geq 0}$ tending to~$0$, we have
 \[
\lim_{n\rightarrow \infty}  \frac{1}{n}  \sum_{i=0}^{n-1} \rho_i (\mathcal{B}_{i} ^{\epsilon_i}) = 1 \quad \mathbf{P}\text{-a.s.},
\] 
where
$\mathcal{B}_i ^\epsilon=\left\{x\in\Rb^d:\ \frac{d \rho_i }{ dx}(x)>\eps\right\}.$	

\end{definition}

\begin{rmk}\rm \label{rem:clustering-densities}
If every $\rho_i$ is absolutely continuous with respect to the Lebesgue measure, then, due to the Lebesgue differentiation theorem, clustering of densities is equivalent to
asymptotically local clustering.
\end{rmk}

The above definitions are Euclidean space extensions of the notion of 
of asymptotic pure atomicity that was introduced first by Vargas in \cite{Var07}  and modified by Bates and Chatterjee in \cite{BC16} in their studies of endpoint distributions for discrete polymers. Moreover, asymptotic clustering at positive levels was studied in  \cite{BM18} still under the name of asymptotic pure atomicity. 
 We  introduce the new term {\it asymptotic clustering} 
instead of {\it asymptotic  pure atomicity} to avoid a misleading image of convergence of the measures in question  to a purely atomic measure.
Roughly speaking, $(\rho_i)_{i \geq 0}$ is asymptotically clustered at level $r$ if the mass of $\rho_i$ concentrates on few balls of radius $r$ for large~$i$.

We state a sufficient condition for asymptotic clustering
that is simpler to verify because it is stated in terms of fixed $\eps > 0$ instead of sequences $(\eps_i)_{i \geq 0}$.

\begin{lem}[Lemma 6.2 in \cite{BC16}]
Let $r>0$ be given. If for every $c>0$, there is $\epsilon = \epsilon(r,c) >0$ such that
\beq \label{eq: equivalent form of asymptotic clustering}
	\liminf_{n \rightarrow \infty} \frac{1}{n} \sum\limits_{i=0}^{n-1} \rho_i(\Ac_{i}^{\epsilon} (r))>1-c \quad \Pb \text{-} a.s.,
\eeq
then $(\rho_i)_{i \ge 0}$ is asymptotically clustered at level $r$. 
\end{lem}
The proof of this lemma repeats the proof of Lemma 6.2 of \cite{BC16} word for word. The discreteness of $\Zb^d$ plays no role in this argument.

\subsection{Auxiliary functionals}
For any $\epsilon > 0$, let us define $f_\epsilon :  \Rb_+ \rightarrow [0, 1]$ by
\begin{align*}
	f_\epsilon (t) = 
\begin{cases} 
		\qquad 0 & \text{for} \quad 0 \leq t < \epsilon, \\
	 	\frac{1}{\epsilon} (t -\epsilon) & \text{for} \quad  \epsilon \leq t \leq 2\epsilon, \\
		\qquad 1 & \text{otherwise.}
\end{cases}
\end{align*}
One can see that $f_\epsilon$ is $1/\epsilon$-Lipschitz continuous and can be interpreted as an approximation of a step function 
$g_\epsilon (t) = \one_{(2\eps, +\infty)} (t)$  for small $\epsilon$.

For any $\mu = (\alpha_i)_{i \in \Nb} \in \Xc$ and $r>0$, let us define a functional $D_r$ on $\Xc \times (\Nb \times \Rb^d)$ as
	\[ D_r (\mu, u) = \frac{1}{V_d r^d}\int_{\Nb \times \Rb^d} \Big(1-\frac{\euc{u-v}}{r} \Big)^+ \mu(dv) 
	= \frac{1}{V_d r^d} \int_{\Rb^d} \Big(1- \frac{\euc{x-y}}{r}\Big)^+ \alpha_i(dy), \]
where $u=(i, x)$ and $a^+ = \max(a, 0)$.
Comparing this with the definition of $I_r$, one has that
	\[ \sup\limits_{u \in \Nb \times \Rb^d} D_r (\mu, u) \leq  \frac{1}{V_d r^d} I_r(\mu). \]
We also observe that 
\begin{align*}
	| D_r (\mu, u_1) - D_r (\mu, u_2) | 
	&\leq \frac{1}{V_d r^d}  \int \Big| \Big(1-\frac{\euc{v-u_1}}{r}\Big)^+ - \Big(1-\frac{\euc{v-u_2}}{r}\Big)^+ \Big|\, \mu(dv) 
	\\& \leq \frac{1}{V_d r^d} \int \frac{\euc{u_1 -u_2}}{r} \mu(dv) \leq  \frac{\euc{u_1-u_2}}{V_d r^{d+1}},
\end{align*}
so $D_r$ is $1/V_d r^{d+1}$-Lipschitz continuous in $u$. Using the embedding of $\Mc_{\leq 1}$ into $\Xc$, we can naturally define  
$D_r$ on $\Mc_{\le 1} \times \Rb^d$.
Combining \eqref{eq: Kantorovich duality} with the fact that $y\mapsto(1-\euc{x-y}/r)^+ $ is $1/r$-Lipschitz for every $x\in\Rb^d$, we obtain
\beq \label{eq: application of kantorovich}
	| D_r (\alpha, x) - D_r (\gamma, x)| \leq \frac{1}{V_d r^{d+1}} W(\alpha, \gamma),\quad \alpha, \gamma \in \Mc_{1},\ x \in \Rb^d.
\eeq
Let us define a functional $J_{r, \epsilon} : \Xc \rightarrow [0, 1]$ by
	\[ J_{r, \epsilon} (\mu) = \int_{\Nb \times \Rb^d} f_\epsilon \!\circ\! D_{r, \mu} (u)  \mu(du) = \sum_{j \geq 1} \int_{\Rb^d} f_{\epsilon} \!\circ\! D_{r, \alpha_j}(x)\alpha_j(dx), \]
where we denoted  $D_r(\mu, \cdot)$ by $D_{r, \mu} (\cdot)$.
$J_{r, \epsilon}$ is well-defined on $\Xct$ due to the following observation 
	\[ \int f_{\epsilon} \!\circ\! D_{r, \alpha*\delta_{y}}(x) (\alpha*\delta_{y})(dx) = \int f_{\epsilon} \!\circ\! D_{r, \alpha}(x) \alpha (dx), 
	\qquad \alpha \in \Mc_{\leq 1},\ y \in \Rb^d. \]

\begin{prop}
For any $r, \eps >0$, $J_{r, \epsilon} : \Xct \rightarrow [0, 1]$ is continuous.
\end{prop}

\begin{proof}
Let $\mu\in\Xct$ and $ \delta_2 > 0$ be given. We claim that
	\[d(\mu, \nu) < \delta_1:=\min\Big(\frac{\epsilon \delta_2 V_d r^{d} (r \wedge 1)}{6}, 2^{-2r}\Big) 
	\quad \Rightarrow \quad |J_{r, \epsilon} (\mu) - J_{r, \epsilon}(\nu)| < \delta_2.\]
Fix $\nu \in \Xct$ satisfying $ d(\mu, \nu) < \delta_1 $. Then, we can find a triple $(r', \phi= \{(\mu_k, \nu_k)\}_{k=1}^{n}, \vec{x})\,$
such that $r'>2r$ and $d_{r', \phi, \vec{x}}(\mu, \nu) < \delta_1$.
Let us denote
	\[\mu^{s} = \mu - \sum\limits_{j=1}^{n} \mu_j , \quad \nu^{s} = \nu - \sum\limits_{j=1}^{n} \nu_j, \quad \nu_k' = \nu_k * \delta_{x_k}.\]
Then, we have
\begin{align*}
	|J_{r, \epsilon} (\mu) - J_{r, \epsilon} (\nu)|
	& \leq \sum\limits_{k=1}^{n} \Big|\int f_\epsilon \!\circ\! D_{r, \mu} (u)  \mu_k(du) - \int f_\epsilon \!\circ\! D_{r, \nu}  (u)  \nu_k(du)\Big|
	\\& + \int f_\epsilon \!\circ\! D_{r, \mu}  (u)  \mu^{s}(du) + \int f_\epsilon \!\circ\! D_{r, \nu}  (u)  \nu^{s}(du). \numberthis\label{eq: split of J(mu)- J(nu)}
\end{align*}
In order to derive the upper bound for the first term in the right-hand side of \eqref{eq: split of J(mu)- J(nu)}, we observe 
\begin{align*}
	&\left|\int f_\epsilon \!\circ\! D_{r, \mu}  (u)  \mu_k(du) - \int f_\epsilon \!\circ\! D_{r, \mu_k}  (u)  \mu_k(du)\right| 
	\\&  = \left|\int f_\epsilon \big( D_{r, \mu_k}(u) + D_{r, \mu-\mu_k}  (u)\big)  \mu_k(du) - \int f_\epsilon \big(D_{r, \mu_k}  (u) \big)  \mu_k(du)\right|
	\\& \leq \frac{1}{\epsilon } \int  D_{r, \mu-\mu_k}(u) \mu_k(du)
	= \frac{1}{\epsilon } \int  D_{r, \mu^s}(u) \mu_k(du)
	\\&\leq \frac{1}{\epsilon V_d r^d} \int  I_r(\mu^s) \mu_k(du)  \leq \frac{1}{\epsilon V_dr^d} I_r(\mu^s) \|\mu_k\|
	\leq \frac{\delta_1}{\epsilon V_d r^d} \|\mu_k\|.
\end{align*}
In the second equality, we used the fact that $\deuc \big(\supp(\mu_k), \supp(\mu_l)\big) > 2r' > 4r$ for all $l \neq k$, which implies  
$D_{r, \mu-\mu_k} = D_{r, \mu^s}$ on $\supp(\mu_k)$.
Combining this with the triangle inequality, we have
\begin{align*}
	& \sum\limits_{k=1}^{n} \Big|\int f_\epsilon \!\circ\! D_{r, \mu}  (u)  \mu_k(du) - \int f_\epsilon \!\circ\! D_{r, \nu}  (u)  \nu_k(du)\Big| 
	\\& \leq \sum\limits_{k=1}^{n} \Big|\int f_\epsilon \!\circ\! D_{r, \mu_k}  (x)  \mu_k(dx) - \int f_\epsilon \!\circ\! D_{r, \nu_k}  (x)  \nu_k(dx)\Big|  
	+  \frac{2\delta_1}{\epsilon V_d r^d}\numberthis\label{eq: first-1 part of J(mu)-J(nu)}.
\end{align*}
On the other hand, we can use  \eqref{eq: application of kantorovich}, \eqref{eq: Kantorovich duality}, and the $1/\epsilon$-Lipschitz continuity of $f_\epsilon$ to write
\begin{align*}
	& \left|\int f_\epsilon \!\circ\! D_{r, \mu_k}  (x)  \mu_k(dx)\! -\! \int f_\epsilon \!\circ\! D_{r, \nu_k}  (x)  \nu_k(dx) \right|
	=\left|\int f_\epsilon \!\circ\! D_{r, \mu_k}  (x)  \mu_k(dx)\! - \!\int f_\epsilon \!\circ\! D_{r, \nu'_k}  (x)  \nu'_k(dx)\right|
	\\& \le \left|\int f_\epsilon \!\circ\! D_{r, \mu_k}  (x)  \mu_k(dx)\! - \!\int f_\epsilon \!\circ\! D_{r, \nu'_k}  (x)  \mu_k(dx) \right| 
	+ \left|\int f_\epsilon \!\circ\! D_{r, \nu'_k}  (x)  \mu_k(dx) \!-  \!\int f_\epsilon \!\circ\! D_{r, \nu'_k}  (x)  \nu'_k(dx) \right|
	\\& \leq \frac{1}{\epsilon} \int  \big|D_{r, \mu_k} (x) - D_{r, \nu'_k}  (x) \big| \mu_k(dx) + \frac{1}{\epsilon V_d r^{d+1}} W(\mu_k, \nu'_k) 
	\leq \frac{2}{\epsilon V_d r^{d+1}} W(\mu_k, \nu'_k).
\end{align*}
Summing over $k$ on both sides gives
\beq\label{eq: first-2 part of J(mu)-J(nu)}
	 \sum\limits_{k=1}^{n} \Big| \int f_\epsilon \!\circ\! D_{r, \mu_k}  (x)  \mu_k(dx) - \int f_\epsilon \!\circ\! D_{r, \nu_k}  (x)  \nu_k(dx) \Big| 
	 \leq  \frac{2}{\epsilon V_d r^{d+1}}  \sum\limits_{k=1}^{n} W(\mu_k, \nu'_k) < \frac{2 \delta_1}{\epsilon V_d r^{d+1}}.
\eeq

Let us estimate the second and the third terms in the right-hand side of \eqref{eq: split of J(mu)- J(nu)}.
\begin{align*}
	 & \int f_\epsilon \!\circ\! D_{r, \mu}  (u)  \mu^{s}(du) =  \int_{\{u: D_{r, \mu} (u) >\epsilon\}} f_\epsilon \!\circ\! D_{r, \mu}  (u)  \mu^{s}(du) 
	  \\& \leq \mu^{s}\big(\{u \in \Nb \times \Rb^d: D_{r, \mu} (u) >\epsilon\} \big) 
	  \leq \mu^{s} \big( \{u \in \Nb \times \Rb^d : \mu(B_r(u)) >\epsilon V_d r^d \} \big).\numberthis\label{eq: second part of J(mu)-J(nu), intermediate}
\end{align*}
Let $C = \{u \in \Nb \times \Rb^d : \mu(B_r(u)) >\epsilon V_d r^d\}$. Then, we can choose a finite number of disjoint balls
	\[ \mathcal{C} = \{B_r (u_i): u_i \in C, 1 \leq i \leq N \} \] 
such that every ball $B_r (u)$, $u \in C$ has non-empty intersection with $\bigcup\limits_{i=1}^{N} B_r (u_i)$.
The disjointness of $\mathcal{C}$ gives $N \leq \frac{1}{\epsilon V_d r^d}$ and 
	\[ C \subset \bigcup\limits_{i=1}^{N} B_{2r} (u_i)  \subset\bigcup\limits_{i=1}^{N} B_{r'} (u_i). \]
Using this and $\mu^{s}(B_{r'} (u)) \leq I_{r'}(\mu^s) <\delta_1$ for all $u \in \Nb \times \Rb^d$, \eqref{eq: second part of J(mu)-J(nu), intermediate} can be continued as
\beq\label{eq: second part of J(mu)-J(nu)}
	 \int f_\epsilon \!\circ\! D_{r, \mu}  (u)  \mu^{s}(du) 
	 \leq \mu^{s} (C) \leq \mu^{s} \left(\bigcup\limits_{i=1}^{N} B_{r'} (u_i)\right) 
	 < N\delta_1 \leq \frac{\delta_1}{\epsilon V_d r^d}.
\eeq
Combining \eqref{eq: split of J(mu)- J(nu)}, \eqref{eq: first-1 part of J(mu)-J(nu)}, \eqref{eq: first-2 part of J(mu)-J(nu)}, 
and \eqref{eq: second part of J(mu)-J(nu)} completes the proof.
\end{proof}

\subsection{Asymptotic clustering of polymer endpoint distributions}
In this section, we prove the following theorem which is a reformulation of 
relations \eqref{eq: asymptotic clustering, intro} and \eqref{eq: endpoint is not asymptotically clusterized in high temperature, intro} 
in Theorem~\ref{thm: asymptotic clursterization, intro}.

\begin{thm}\label{thm: asymptotic clustering}
\mbox{}
\begin{enumerate}[label=\rm(\alph*), nosep]
	\item If $\beta > \beta_c$, then for all $r>0$, $(\rho_i)_{i \geq 0}$ is asymptotically clustered at level $r$.
	\item If $\beta \leq \beta_c$, then  for all $r>0$, $(\rho_i)_{i \geq 0}$ is not asymptotically clustered  at level $r$. 
Moreover, for any $r>0$, there is a sequence $(\epsilon_i)_{i \geq 0}$ tending to $0$ as $i \rightarrow \infty$, such that
\beq\label{eq: endpoint is not asymptotically clustered}
	\lim\limits_{n \rightarrow \infty} \frac{1}{n}\sum\limits_{i=0}^{n-1} \rho_i (\Ac_{i}^{\epsilon_i} (r) ) = 0 \qquad \mathbf{P} \text{-}a.s.
\eeq
\end{enumerate}
\end{thm}

\begin{proof}
(a) Suppose $\beta > \beta_c$. Our goal is to show~\eqref{eq: equivalent form of asymptotic clustering}. 
To this end, for any $r, \eps > 0 $, let us define a functional $\Jc_{r, \eps}$ on $\Pc(\Xct)$ as
	\[ \Jc_{r, \epsilon} (\xi) := \int J_{r, \epsilon} (\mu) \xi(d\mu). \]
Notice that the continuity of $J_{r, \epsilon}$ passes on to the continuity of $\Jc_{r, \epsilon}$. We also observe that
	\[ J_{r, \epsilon} (\rho_i) 
	= \int_{\{x : D_{r, \rho_i}  (x) >\epsilon\}}  f_\epsilon \!\circ\! D_{r, \rho_i}  (x)  \rho_i (dx) 
	\leq \int_{\Ac_{i}^{\epsilon} (r)}  f_\epsilon \!\circ\! D_{r, \mu}  (x)  \rho_i (dx) 
	\leq  \rho_i ( \Ac_{i}^{\epsilon} (r) ).\]
Hence,
\beq\label{eq: lower bound of clustered mass}
	\frac{1}{n} \sum\limits_{i=0}^{n-1}\rho_i(\Ac_{i}^{\epsilon}) 
	\geq \frac{1}{n} \sum\limits_{i=0}^{n-1} J_{r, \epsilon} (\rho_i) =\Jc_{r, \epsilon} (\psi_{n}).
\eeq
On the other hand, for any $\mu \in \Xct$ with $\norm{\mu} = 1$, we have
$\lim\limits_{\epsilon \, \downarrow \, 0} J_{r, \epsilon} (\mu) = 1$
because
$\mu \big(\{u: D_{r, \mu} (u) > 0\}\big) =1$ and $J_{r, \epsilon} (\mu) \geq \mu \big(\{u: D_{r, \mu} (u) > 2\epsilon \}\big)$.
Using this along with Theorem~\ref{thm: chracterstic of K}~(2), we obtain that
for every $\xi \in \Kc_0$, 
	\[ \lim\limits_{\epsilon \, \downarrow \, 0} \Jc_{r, \epsilon} (\xi) =1. \]
Since each $\Jc_{r, \epsilon}$ is continuous on the compact set $\Kc_0$ and $(\Jc_{r, \epsilon})_{\eps>0}$ is monotone increasing as $\eps \downarrow 0$,
the convergence above is uniform in $\xi$ by the Dini's theorem.

Let now $c>0$ be given. By the uniform convergence of $(\Jc_{r, \epsilon})_{\eps>0}$, we can choose $\epsilon = \epsilon(r, c) > 0$ such that
	\[	\Jc_{r, \epsilon} (\xi) >1- c \quad \text{for all} \,\, \xi \in \Kc_0 \]
and, for such $\epsilon$,  we can also find $\delta > 0$ such that
\beq\label{eq: lower bound of J(zeta)}
	 \Wc(\zeta, \Kc_0) < \delta \quad \Rightarrow \quad \Jc_{r, \epsilon} (\zeta) > 1-c.
\eeq
Combining Theorem~\ref{thm: convergence of psi to K0}, \eqref{eq: lower bound of clustered mass}, and \eqref{eq: lower bound of J(zeta)}, we complete the proof of (a).

(b)  Suppose $\beta \leq \beta_c$ and let $r>0$, $\eps >0$ be given. We claim that
\beq \label{eq: no clustering}
	\lim\limits_{n \rightarrow \infty} \frac{1}{n} \sum\limits_{i=0}^{n-1}\rho_i( \Ac_{i}^{\epsilon}(r) ) = 0.
\eeq
To see this, we observe that for any $\mu \in \Xct$ and any $u \in \Ac_\mu ^{\eps} (r) = \{ v \in \Nb \times \Rb^d : \mu (B_r (v)) > \epsilon V_d r^d\}$,
	\[ D_{2r, \mu} (u) = \frac{1}{V_d (2r)^d} \int \Big(1- \frac{\euc{u-v}}{2r}\Big)^+ \mu(dv) 
	\geq \frac{1}{V_d (2r)^d} \int \frac{ \one_{B_r(u)}(v)}{2} \mu(dv)
	= \frac{ \mu(B_r(u))}{2^{d+1} V_d r^d} > 2\eps', \]
where $\eps' = \eps/2^{d+2}$. Therefore, we have
\begin{align*}
	J_{2r, \eps'} (\mu) \geq \int_{\{D_{2r, \mu} (u) > 2\eps'\}} f_{\eps'} \!\circ D_{2r, \mu} (u) \mu(du)
	= \mu \big(\{D_{2r, \mu} (u) > 2\eps' \}\big) \geq \mu(\Ac_{\mu}^{\eps}(r)),
\end{align*}
which implies that
	\[ \frac{1}{n} \sum\limits_{i=0}^{n-1}\rho_i(\Ac_{i}^{\epsilon}(r)) 
	\leq \frac{1}{n} \sum\limits_{i=0}^{n-1} J_{2r, \eps'} (\rho_i) =\Jc_{2r, \eps'} (\psi_{n}). \]
By Theorem~\ref{thm: chracterstic of K}~(a) and the continuity of $\Jc_{2r, \eps'}$, we conclude
	\[ \lim\limits_{n \rightarrow \infty} \frac{1}{n} \sum\limits_{i=0}^{n-1}\rho_i(\Ac_{i}^{\epsilon}(r)) \leq \Jc_{2r, \eps'} (\delta_{\mathbf{0}}) =0.  \]

Fix $r>0$ and let us now construct a sequence $(\eps_i)_{i \geq 0}$ tending to $0$ and satisfying \eqref{eq: endpoint is not asymptotically clustered}.
By \eqref{eq: no clustering}, we see that for each $k \in \Nb$, there is $N_k$ such that
	\[ \frac{1}{n} \sum\limits_{i=0}^{n-1}\rho_i(\Ac_{i}^{1/k}(r)) <\frac{1}{k} \quad \text{for all}\,\, n \geq N_k. \]
We may assume $N_{k+1} > N_k$ for all $k$. Set $\eps_i = 1$ for $i < N_1$ and $\eps_i = 1/k$ for $N_{k} \leq i <N_{k+1}$.
Then, we see that for each $n \in \Nb$, there is $k = k(n)$ such that $N_k \leq n < N_{k+1}$ and hence
	\[ \frac{1}{n} \sum\limits_{i=0}^{n-1}\rho_i(\Ac_{i}^{\eps_i}(r)) 
	\leq \frac{1}{n} \sum\limits_{i=0}^{n-1}\rho_i(\Ac_{i}^{1/k}(r)) < \frac{1}{k}. \]
Since $\lim\limits_{n \rightarrow \infty} k(n) =\infty$, letting $n \rightarrow \infty$ on the both side above completes the proof of (b).
\end{proof}

\subsection{Asymptotic local clustering of the endpoint distribution} \label{subsection: asymptotic thick density}
In this section, we prove the following reformulation of relation~\eqref{eq: asymptotic local clustering, intro} in Theorem~\ref{thm: asymptotic clursterization, intro}:
\begin{thm}\label{thm: asymptotic local clustering}
If $\beta > \beta_c$, then $(\rho_i)_{i \geq 0}$ is asymptotically locally clustered.
In particular, for these values of $\beta$, clustering of densities (see Remark~\ref{rem:clustering-densities}) holds if the reference random walk step distribution
$\ker(dx)$ is absolutely continuous.
\end{thm}
Before we prove this, we recall the Besicovitch covering theorem and its related lemma which will be used later.
\begin{thm}{\rm\textbf{(Besicovitch covering theorem)}}\label{thm: Besicovitch covering}
There is a constant $N_d$, depending only on the dimension $d$, with the following property:\\
Let $\mathcal{F} = \{B_{r_\sigma}(x_\sigma):\sigma\in\Ic \}$ be any collection of open balls in $\Rb^d$ with $\sup\{ r_\sigma :  \sigma\in\Ic \} < \infty$.
Let us denote $A = \{x_\sigma:\sigma\in\Ic \}$. Then, there is a countable subcollection $\mathcal{G}$ of $\mathcal{F}$ such that $\mathcal{G}$ is a cover of $A$ and
every $x \in \bigcup\limits_{B \in \mathcal{G} } B$ belongs to at most $N_d$ different balls from the subcover $\mathcal{G}$.
\end{thm}
We remark that in \cite{FL94}, the lower bound and the upper bound for $N_d$(\textit{Besicovitch constant}) were provided:
	\[ (2.065 + o(1))^d \leq N_d \leq (2.691 + o(1))^d. \]
We now state a lemma which is based on the Besicovitch covering theorem.
\begin{lem}[Lemma 1.2. in \cite{EG15}]
Let $\alpha$, $\gamma$ be Radon measures on $\Rb^d$ and define
	\[ \underline{D}_{\mu} \nu (x) = 
	\begin{cases}
	\liminf\limits_{r \downarrow 0} \frac{\alpha(B_r(x))}{\gamma(B_r(x))} & \text{if} \,\, \gamma(B_r(x)) > 0 \,\, \text{for all} \,\,  r>0,\\
	+ \infty & otherwise.
	\end{cases} \]
Let $\epsilon>0$ be given. Then, for any Borel $A \subset \{ x \in \Rb^d : \underline{D}_{\gamma} \alpha (x) \leq \epsilon \}$, we have
	$\alpha(A)\leq \epsilon \gamma(A)$.
\end{lem}

Similarly to $\Ac_i ^\eps (r)$ and $\Ac_i ^\eps$, let us denote
\[ \Ac_\alpha ^{\epsilon} (r) = \{ x \in \Rb^d : \alpha (B_r (x)) > \epsilon V_d r^d\}, \quad 
\Ac_\alpha ^{\epsilon} = \{ x \in \Rb^d : \liminf\limits_{r \downarrow 0} \frac{\alpha \big(B_r (x)\big)}{V_d r^d} > \epsilon\}\]
for any $\alpha \in \Mc_1$, $\eps>0$ and $r>0$.
By substituting $\gamma = \mathfrak{m}$ (Leb esgue measure on $\Rb^d$) in the lemma above, we obtain
\beq\label{ineq: upper bound for alpha measure in low density region}
	\alpha (A) \leq \epsilon \mathfrak{m} (A), \quad  \forall \text{\rm\ Borel}\  A \subset (\Ac_{\alpha} ^{\epsilon})^c.
\eeq

\begin{prop}\label{prop: comparison between atomicity set  and density set}
Let $\epsilon, c, r >0$ be given and let us assume that $\alpha \in \Mc_{1}$ satisfies
\beq\label{eq: A_alpha_is_heavy}
	\alpha(\Ac_\alpha ^{\epsilon}(r)) > 1-c.
\eeq
Then, there is $\epsilon_1 = \epsilon_1 (\epsilon, c, r, d)>0$, independent of $\alpha$, such that 
	\[ \alpha (\Ac_{\alpha} ^{\epsilon_1} ) > 1-2c. \]
\end{prop}

\begin{proof}
Let $\eps >0$ be given. For any $ t \in(0,1)$, let us set $s=(1-t)/(V_dr^d)$ and 
\beq\label{eq: choice of eps_1}
	\epsilon_1=s\eps = \frac{1-t}{V_d r^d}\eps.
\eeq
We will determine the value of $t$ later.

Let $x\in \Ac_\alpha^{\eps}(r)$ and suppose $\alpha(\Ac_\alpha^{\epsilon_1} \cap B_r(x))\le t \alpha(B_r(x)) $. Then, we see that
\begin{align*}
	\alpha(B_r(x))&= \alpha(\Ac_\alpha^{\epsilon_1} \cap B_r(x)) + \alpha((\Ac_\alpha^{\epsilon_1})^c \cap B_r(x))
 	\\& \le t \alpha(B_r(x)) +  \eps_1 \mathfrak{m}((A_\alpha^{\epsilon_1})^c \cap B_r(x)) \le t \alpha(B_r(x)) + s\eps V_d r^d.
\end{align*}
In the first inequality above, we used \eqref{ineq: upper bound for alpha measure in low density region}.
By~\eqref{eq: choice of eps_1}, we have
	\[	\alpha(B_r(x))\le \frac{s\eps  V_d r^d}{1-t} = \eps, \] 
which contradicts $x\in \Ac_\alpha^{\eps}(r)$. Therefore, we obtain
\[
\alpha(\Ac_\alpha^{\eps_1} \cap B_r(x))> t \alpha(B_r(x)) \]
or, equivalently,
\begin{equation}
\label{eq: intersection-of D_alpha-with-B_r_x-ratio}
 \alpha((\Ac_\alpha^{\eps_1})^c \cap B_r(x))  \le  (1-t)\alpha(B_r(x)).
\end{equation}

Let us now apply Theorem~\ref{thm: Besicovitch covering} with $\mathcal{F} = \{B_r(x) : x \in \Ac_\alpha ^{\eps}(r)\}$ and $A= \Ac_\alpha ^{\eps}(r)$.
Then, we can find a countable subset $\Ac \subset \Ac_\alpha ^{\eps}(r)$ such that $\mathcal{G}=\{B_r(x)$ : $x\in\Ac \}$ is a cover of $\Ac_\alpha ^{\eps}(r)$ and
every $x \in \bigcup\limits_{y \in \Ac} B_r (y)$ is covered by at most $N_d$ balls from $\mathcal{G}$.
Therefore, due to~\eqref{eq: intersection-of D_alpha-with-B_r_x-ratio}, we have
\begin{align*}
	\alpha\left((\Ac_\alpha^{\eps_1})^c\cap \Ac_\alpha^{\eps}(r) \right) 
	&\le \alpha\left((\Ac_\alpha^{\eps_1})^c\cap \bigcup_{x\in \Ac} B_r(x)\right)
	\le  \sum_{x\in\Ac} \alpha((\Ac_\alpha^{\eps_1})^c \cap B_r(x)) 
	\le  (1-t)\sum_{x\in\Ac} \alpha(B_r(x))
	\\&\le N_d (1-t) \alpha\left(\bigcup_{x\in \Ac} B_r(x)\right)
	\le N_d(1-t).\numberthis\label{ineq: A_alpha c cap A_alpha,1  is small}
\end{align*}
Therefore, due to~\eqref{ineq: A_alpha c cap A_alpha,1  is small} and \eqref{eq: A_alpha_is_heavy},
\begin{align*}
	\alpha\left((\Ac_\alpha^{\eps_1})^c\right) 
	&= \alpha\left((\Ac_\alpha^{\eps_1})^c\cap \Ac_\alpha^{\eps}(r) \right)  + \alpha\left((\Ac_\alpha^{\eps_1})^c \cap (\Ac_\alpha^{\eps}(r))^c \right)
	\\& \le N_d(1-t)+\alpha((\Ac_\alpha^{\eps}(r))^c )
	< N_d(1-t)+c.
\end{align*}
Choosing $t = 1-\frac{c}{N_d}$ \Big($s=\frac{c}{V_d N_d r^d }$\Big) completes the proof.
\end{proof}

\begin{subproof}[Proof of Theorem~\ref{thm: asymptotic local clustering}]
Suppose $\beta > \beta_c$. Let $(\epsilon_i)_{i \geq 0}$ tending to $0$ be given. 
For any $c>0$, let us denote $s=\frac{c}{2  V_d N_d}$ and
	\[ F_c = \{ i \geq 0 : \rho_i(\Ac_{i} ^{\epsilon_i/s}(1)) > 1-c/2 \}, \quad F'_c = \{ i \geq 0 : \rho_i(\Ac_{i} ^{\epsilon_i}) > 1-c \}.\]
By Theorem~\ref{thm: asymptotic clustering} (a), $(\rho_i)_{i \geq 0}$ is asymptotically clustered at level $1$. This can be rewritten as
	\[ \lim\limits_{n \rightarrow \infty} \frac{1}{n} |F_{c} \cap [0, n-1]| = 1. \]
Since Proposition~\ref{prop: comparison between atomicity set  and density set} implies $F_c \subset F'_c$, the same relation holds for $F'_c$ and hence
	\[ 	\liminf\limits_{n\rightarrow \infty}  \frac{1}{n}  \sum_{i=0}^{n-1} \rho_i (\Ac_{i} ^{\epsilon_i}) \geq 1-c \quad \mathbf{P}\text{-a.s.}\]
Letting $c \downarrow 0$ completes the proof.
\end{subproof}

\section{Geometric localization} \label{sec: geometric localization}
Adapting the terminology from~\cite{BC16}, we say that the sequence $(\rho_n)_{n \ge 0}$ is \textit{geometrically localized with positive density} if for any $\delta >0$, 
there exist $K < \infty$ and $\theta > 0$ such that
\beq \label{def: geometric localization}
	\liminf\limits_{n \rightarrow \infty} \frac{1}{n} \sum\limits_{i=0}^{n-1} \one_{\{\rho_i \in \mathcal{G}_{\delta, K}\}} \geq \theta \quad \mathbf{P} \text{-}a.s.,
\eeq
where 
	\[ \mathcal{G}_{\delta, K} = \{\alpha \in \Mc_1 : \max\limits_{x \in \Rb^d} \alpha(B_K (x)) > 1 - \delta \}.\]
If $\theta$ can be taken equal to 1, then the sequence would be \textit{geometrically localized with full density}. 
A full density localization is an open question. In this section, we prove that
$(\rho_i)_{i\geq 0}$ is geometrically localized with positive density if and only if $\beta > \beta_c$.

\subsection{Useful functionals}
Given $\mu \in \Xct$, we choose its representative $\mu = (\alpha_i)_{i \in \Nb} \in \Xc$. 
To describe $\mathcal{G}_{\delta, K}$ in the language of $(\Xct, d)$, let us consider
	\[ W_\delta (\mu)  = \inf \{ r \geq 0 : I_r(\mu) > 1-\delta \}, \quad \mathcal{V}_{\delta, K} =\{ \mu \in \Xct : W_\delta (\mu) < K \}. \]
Using~\eqref{eq: relation between I_r and the heaviest ball} and the natural embedding of $\wt{\Mc}_1$ into $\Xct$, we obtain
\beq\label{eq: relation between set G and set V}
\mathcal{G}_{\delta, K} \subset \mathcal{V}_{\delta, K} \cap \{\mu \in \Xct : |S_\mu| =1 , \norm{\mu} = 1\} \subset \mathcal{G}_{\delta, K+1},
\eeq
where $S_\mu$ was defined in \eqref{eq: N-support}.
We also define functionals
	\[ G(\mu) = \max\limits_{i \in \Nb} \norm{\alpha_i}, \quad Q(\mu) = \sum\limits_{i \in \Nb} \frac{\norm{\alpha_i}}{1-\norm{\alpha_i}}.  \]
One can check that $ W_\delta, G$ and $Q$ are well-defined on $\Xct$.
\begin{prop}
\mbox{}
\begin{enumerate}[label=\rm(\alph*),nolistsep]
	\item $W_\delta$ is upper semi-continuous.
	\item $G$ is lower semi-continuous.
	\item $Q$ is lower semi-continuous.
\end{enumerate}
\end{prop}

\begin{proof}
(a) Let $\mu \in \Xct$ and $\epsilon > 0$ be given.
Then, $I_{W_\delta (\mu)+\epsilon} (\mu) > 1- \delta$. By Lemma~\ref{lem: continuity of I_r}, there is $\delta_1 > 0$ such that
	\[ d(\mu, \nu) < \delta_1 \quad  \Rightarrow \quad |I_{W_\delta (\mu)+\epsilon} (\mu) - I_{W_\delta (\mu)+\epsilon} (\nu)| < \epsilon', \]
where $\epsilon' = \frac{1}{2}\big( I_{W_\delta (\mu)+\epsilon} (\mu) - 1 + \delta \big)$. For such $\nu$, we obtain $I_{W_\delta (\mu)+\epsilon} (\nu) > 1 - \delta$, which implies
	\[ W_\delta(\nu) \le   W_\delta (\mu) + \epsilon. \]

(b) Fix $\mu \in \Xct$ and $0<\epsilon_2 <G(\mu).$ We must show that there is $\epsilon_1>0$ such that
	\[ d(\mu, \nu)  < \epsilon_1  \quad \Rightarrow \quad G(\nu) > G(\mu) - \epsilon_2. \]
Without loss of generality, we may assume $\norm{\alpha_1} = G(\mu)$. 
There is $R = R(\epsilon_2)>0$ such that $\mu(B_R (1, 0)) = \alpha_1(B_R(0)) > G(\mu) - \epsilon_2/2>\epsilon_2/2$.
Choose $\epsilon_1 = \min(\epsilon_2/2, 2^{-R})$. Then, assuming $d(\mu, \nu)  < \epsilon_1$, there is a triple $(r, \phi = \{(\mu_k, \nu_k)\}_{k=1}^{n}, \vec{x})$ such that
$d_{r, \phi, \vec{x}} (\mu, \nu) < \epsilon_1$.
Since $\sep(\phi)>2r>2R$, there is at most one $\mu_k$ whose support intersects with $B_R((1,0))$. If there is no $\mu_k$ whose support intersects with $B_R((1,0))$, then
$\one_{B_R(1, 0)} \mu \leq \mu- \sum\limits_{k=1}^{n} \mu_k$. It follows
	\[ d_{r, \phi ,\vec{x}} (\mu, \nu) > I_r \Big(\mu-\sum\limits_{k=1}^{n} \mu_k \Big) \geq \alpha_1\big(B_R(0)\big) > \epsilon_2/2 \geq \epsilon_1, \]
which is a contradiction. Therefore, we may assume $\mu_1$ is a unique submeasure of $\alpha_1$ whose support overlaps with $B_R(0)$.
Since $\one_{B_R(1, 0)} (\mu-\mu_1) \leq \mu- \sum\limits_{k=1}^{n} \mu_k$, we have
	\[ (\mu - \mu_1) (B_R(1, 0)) \leq I_R \Big(\mu - \sum\limits_{k=1}^{n} \mu_k\Big) \leq  I_r \Big(\mu - \sum\limits_{k=1}^{n} \mu_k \Big). \]
Therefore,
\begin{align*}
	& G(\nu) \geq \norm{\nu_1}	 = \norm{\mu_1} \geq \mu_1(B_R(1,0))
	\\& \geq \mu(B_R(1,0)) -  I_r \Big(\mu -\sum\limits_{k=1}^{n} \mu_k \Big) > \big(\norm{\alpha_1} - \frac{\epsilon_2}{2} \big) - \frac{\epsilon_2}{2} =G(\mu) -\epsilon_2.
	\numberthis\label{ineq: lower bound for m(nu)}
\end{align*}

(c)  Fix $\mu \in \Xct$ . If $Q(\mu) = \infty$ (i.e. $G(\mu) = 1$), then for any $L>0$, by part (b), we can find $\epsilon_1>0$ such that 
	\[ d(\mu, \nu)< \epsilon_1 \quad \Rightarrow \quad G(\nu) > 1 - \frac{1}{L+1} \quad \Rightarrow \quad  Q(\nu) > L. \]
Now consider the case $Q(\mu) < \infty$ and fix $\epsilon_2 >0$. First, we can find $N$ such that 
	\[ \sum\limits_{i > N} \norm{\alpha_i} < \frac{\epsilon_2}{2+\epsilon_2}. \]
Since for any nonnegative $x_1,x_2,\ldots$ satisfying $\sum_i x_i <1$, we see that
\[
\sum_i\frac{x_i}{1-x_i}\le \sum_i \frac{x_i}{1-\sum_j x_j} = \frac{\sum_i x_i}{1-\sum_i x_i},
\]
which implies $\sum\limits_{i >N} \frac{\norm{\alpha_i}}{1-\norm{\alpha_i}} <\epsilon_2/2$.\\
We may assume $\norm{\alpha_i} \geq \norm{\alpha_{i+1}}$ for $ i \leq N-1$. Let
	\[ N_1 = \sup \Big\{ i\in\{1,\ldots,N\}: \norm{\alpha_i} > \frac{\eps_2}{2N+\eps_2} \Big\} \vee 0. \]
We can choose $R = R(\epsilon_2)>0$ such that 
	\[ \max\limits_{1 \leq i \leq N_1} \alpha_i(B_R (0)^c) < \frac{\epsilon_2}{4N+2\eps_2} . \] 
Applying the argument that we used in \eqref{ineq: lower bound for m(nu)}, we can find $\eps_1>0$ such that
$d(\mu,\nu)<\eps_1$ guarantees the existence of $(\mu, \nu)$-matching $\{(\mu_k, \nu_k)\}_{k=1}^{n}$  
such that $n \geq N_1$ and 
	\[ \norm{\nu_k} \geq \norm{\alpha_k} - \frac{\epsilon_2}{2N+\eps_2} \]
for all $1 \leq k \leq N_1$. It follows that
	\[ \frac{\norm{\nu_k}}{1-\norm{\nu_k}} \geq \frac{\norm{\alpha_k}}{1-\norm{\alpha_k}} - \frac{\epsilon_2}{2N+\eps_2}
	> \frac{\norm{\alpha_k}}{1-\norm{\alpha_k}} - \frac{\epsilon_2}{2N} \qquad \text{for}\,\, 1 \leq  k \leq N_1 . \]
On the other hand, by the definition of $N_1$, we have
	\[ \frac{\norm{\alpha_k}}{1-\norm{\alpha_k}} \leq \frac{\eps_2}{2N}\qquad \text{for}\,\, N_1 < k \leq N. \]
Therefore, we conclude
\begin{align*}
	 Q(\nu) &\geq \sum\limits_{k=1}^{n} \frac{\norm{\nu_k}}{1-\norm{\nu_k}}
	\geq \sum\limits_{k=1}^{N_1} \left(\frac{\norm{\alpha_k}}{1-\norm{\alpha_k}} - \frac{\epsilon_2}{2N}\right)
	+ \sum\limits_{k=N_1 +1}^{N} \left(\frac{\norm{\alpha_k}}{1-\norm{\alpha_k}} - \frac{\epsilon_2}{2N} \right)
	\\& > \sum\limits_{k=1}^{N} \frac{\norm{\alpha_k}}{1-\norm{\alpha_k}} - \frac{\eps_2}{2} > Q(\mu) - \eps_2.
\end{align*}
\end{proof}

\begin{lem}\label{lem: Q infty}
Assume $\beta > \beta_c$. Then, for any $\xi \in \Kc_0$,
 \[ \int Q(\mu) \xi(d\mu) = \infty. \]
\end{lem}

\begin{proof}
Let $\xi \in \Kc_0$ be given. Suppose, towards a contradiction, that
	\[\int Q(\mu) \xi(d\mu) < \infty.\]
This implies $\xi\big(\{\mu \in \Xct: |S_\mu| = 1\}\big) = 0$. Let $\eta = [\wt{\alpha}_i]$ be a $\Xct$-valued random variable whose law is $\xi$ and 
$Y$ be a random field with the same law as $X$ and independent of $\eta$. 
Recalling $\xi \big(\{\mu \in \Xct: \norm{\mu} = 1\}\big)=1$ in Theorem~\ref{thm: chracterstic of K}~(b), we have $\norm{\eta} = 1$ almost surely.
Since $\Tc \eta = \eta$, the law of
 	\[ \hat{\eta} (du) = \frac{e^{\beta Y(u)}\, \eta * \ker (du)}
	{ \int_{\Nb\times \Rb^d} e^{\beta Y(w)} \eta *\ker (dw)} \]
is also $\xi$. We observe that
\begin{align*}
	\Eb [Q(\hat{\eta}) | \eta] &= \sum_{i \in\Nb}  \Eb \bigg[ \frac{\int e^{\beta Y(i, x)}\, \alpha_i * \ker (dx)}
	{ \sum_{j \neq i } \int_{\Rb^d} e^{\beta Y(j, x)} \alpha_j *\ker (dx)} \Big| \eta \bigg]
	\\& =  \sum_{i \in\Nb} \Eb \Big[ \int e^{\beta Y(i, x)}\, \alpha_i * \ker (dx) \Big| \eta \Big]
	\Eb \Big[ \frac{1}{\sum_{j \neq i } \int_{\Rb^d} e^{\beta Y(j, x)} \alpha_j *\ker (dx)} \Big| \eta \Big]
	\\& > \sum_{i \in\Nb} \frac{\Eb \Big[ \int e^{\beta Y(i, x)}\, \alpha_i * \ker (dx) \Big| \eta \Big]}
	{\Eb \Big[\sum_{j \neq i } \int_{\Rb^d} e^{\beta Y(j, x)} \alpha_j *\ker (dx) \Big| \eta \Big]}
	=  \sum_{i \in\Nb} \frac{\norm{\alpha_i}}{\sum_{j \neq i } \norm{\alpha_j}} = Q(\eta),
\end{align*}
where we used the independence between $Y(i, \cdot)$ and $(Y(j, \cdot))_{j \neq i}$ in the second line, Jensen's inequality in the third line. 
Integrating with respect to $\eta$ of the both sides leads to a contradiction, which completes the proof.
\end{proof}

The following result is a reformulation of Theorem~\ref{thm: geometric localization, intro}:
\begin{thm}\label{thm: localization}
\mbox{}
\begin{enumerate}[label=\rm(\alph*), nolistsep]
	\item If $\beta > \beta_c$, then $(\rho_i)_{i \geq 0}$ is geometrically localized with positive density.
	\item If $\beta \leq \beta_c$, then for any $\delta \in (0, 1)$ and any $K>0$,
\beq \label{eq: endpoint measure is not geometrically localized}
	\lim\limits_{n \rightarrow \infty} \frac{1}{n}\sum\limits_{i=0}^{n-1} \one_{\{\rho_i \in \mathcal{G}_{\delta, K}\}} = 0 \quad \mathbf{P} \text{-} a.s.
\eeq
\end{enumerate}
\end{thm}

\begin{proof}
(a) Let $\delta >0 $ be given. The left-hand side of \eqref{def: geometric localization} can be expressed in terms of the empirical measure $\psi_n$:
	 \[ \frac{1}{n} \sum\limits_{i=0}^{n-1} \one_{\{\rho_i \in \mathcal{G}_{\delta, K}\}}
	= \frac{1}{n} \sum\limits_{i=0}^{n-1} \delta_{\rho_i} (\mathcal{G}_{\delta, K})
	=\psi_{n} (\mathcal{G}_{\delta, K}). \]
Therefore, it suffices to show that there are $K>0$ and $\theta>0$ such that
\beq  \label{ineq: representation of geometric localization via empirical measure}
	\liminf\limits_{n \rightarrow \infty} \psi_{n} (\mathcal{G}_{\delta, K}) \geq \theta.
\eeq 
To see this, let us define 
\beq \label{eq:def of U_delta}
	\mathcal{U}_\delta := \{\mu \in \Xc : G(\mu) > 1 -\delta \} = \bigcup_{K=1}^{\infty} \mathcal{V}_{\delta, K}.
\eeq
By Lemma \ref{lem: Q infty}, for all $\xi \in \Kc_0$,
	\[\xi(\mathcal{U}_\delta)>0. \]
The lower semi-continuity of $G$ implies that $\mathcal{U}_\delta$ is an open set, so the map $\xi \rightarrow \xi (\mathcal{U}_\delta)$ is also lower semi-continuous. Together with the compactness of $\Kc_0$,
we have
	\[\theta := \inf\limits_{\xi \in \Kc_0} \xi(\mathcal{U}_\delta) >0. \]
For each $\xi \in \Kc_0$, we can use \eqref{eq:def of U_delta} and monotonicity of $\mathcal{V}_{\delta, K}$ in~$K$ to choose $K=K_\xi < \infty$ such that
	\[ \xi(\mathcal{V}_{\delta, K}) > (1-\epsilon)\theta. \]
The upper semi-continuity of $W_\delta$ implies that the map $\xi \rightarrow \xi (\mathcal{V}_{\delta, K})$ is lower semi-continuous. Hence, there is $r_\xi >0$ such that
	\[ \inf\limits_{\zeta \in \mathcal{B}(\xi, r_\xi)} \zeta(\mathcal{V}_{\delta, K_\xi}) >(1-\epsilon)\theta. \]
Since $\Kc_0$ is compact and $\{\mathcal{B}(\xi, r_{\xi}/2)\}_{\xi \in \Kc_0}$ is a open covering of $\Kc_0$, 
we can choose a finite sub-covering $\{\mathcal{B}(\xi_i, r_{\xi_i}/2)\}_{i=1}^{n}$.
Now let $K = \max\limits_{1\leq i \leq n} \{K_{\xi_i}\}$, $r = \min\limits_{1\leq i \leq n} \{r_{\xi_i}/2\}$. 
Using the finite open covering of $\Kc_0$ above and \eqref{eq: relation between set G and set V}, we have
	\[ \Wc(\xi, \Kc_0) < r \quad \Rightarrow \quad \xi (\mathcal{V}_{\delta, K}) > (1-\epsilon)\theta 
	\quad \Rightarrow \quad \xi (\mathcal{G}_{\delta, K+1}) > (1-\epsilon)\theta. \]
Notice that $\lim\limits_{n \rightarrow \infty} \Wc(\psi_n, \Kc_0) \rightarrow 0$ from Theorem~\ref{thm: convergence of psi to K0}.
Therefore, letting $\epsilon \downarrow 0$, we obtain \eqref{ineq: representation of geometric localization via empirical measure}. 
\medskip

(b) Suppose $\beta \leq \beta_c$ and let $\delta \in (0, 1)$, $K>0$, and $\epsilon>0$ be given. We write
	\[ G_{n} = \Big\{i\in \{0,\ldots,n-1\}: \max\limits_{x\in\Rb^d} \rho_i \big(B_K(x)\big) >1-\delta\Big\}. \]
Then, 	\eqref{eq: endpoint measure is not geometrically localized} is equivalent to
\beq \label{eq: equivalent form of geometric delocalization}
	\lim\limits_{n \rightarrow \infty} \frac{|G_n|}{n} = 0 \quad \mathbf{P} \text{-} a.s.
\eeq
Recalling Theorem~\ref{thm: localization of directed polymers}, we can write
	\[ \lim\limits_{n \rightarrow \infty} \frac{1}{n} \sum\limits_{i=0}^{n-1} \max\limits_{x\in\Rb^d} \rho_i \big(B_K (x) \big) = 0  \quad \mathbf{P} \text{-} a.s.\]
Therefore, there is $N \in \Nb$ such that
	\[  \frac{1}{n} \sum\limits_{i=0}^{n-1} \max\limits_{x \in \Rb^d} \rho_i \big(B_K(x)\big) < (1-\delta)\epsilon \quad \text{for all}\,\, n \geq N, \]
and for such $n$, we have $ |G_n|/n <\epsilon $. Letting $n \rightarrow \infty$ and then $\eps \downarrow 0$, we obtain \eqref{eq: equivalent form of geometric delocalization}.
\end{proof}

\appendix
\section{An auxiliary coupling lemma}

Here we formally state and prove a coupling lemma used in Section~\ref{sec: continuity-cond-update}, we used a statement 

Let us recall that the LU topology (topology of locally uniform convergence) on the space $C[\Rb^d, \Rb]$ 
of all continuous real-valued functions on $\Rb^d$
is defined by the following metric: 
\[ \rho(\omega_1, \omega_2) = \sum\limits_{n=1}^{\infty} \frac{1}{2^n}\Big(\sup\limits_{\sss{\euc{x} \leq n}} |\omega_1 (x) - \omega_2 (x)| \wedge 1 \Big) .\]
Let $\mathscr{H} = \mathscr{B}(C[\Rb^d, \Rb])$ be the Borel sigma-algebra on $C[\Rb^d, \Rb]$ equipped with LU topology.
Let $\mathbb{P}$ be the distribution of $X(\cdot)$ on $(\Omega, \mathscr{F})$. Under $\mathbb{P}$, the canonical process $Y_x(\omega) = \omega (x)$ is a  distributional copy of $X(x)$.
\begin{lem}
\label{lem:coupling-using-m-dep}
Suppose closed sets $U_1,  U_2, \cdots, U_n \subset \Rb^d$ satisfy $\min\limits_{i\neq j} \deuc(U_i, U_j) >\rdep$.
Then, there is an extended probability space $(\Omega', \mathscr{F}', \mathbb{P}')$ and stationary processes $Y^{(1)}, Y^{(2)}, \cdots Y^{(n)}$ defined on this space such that
\begin{enumerate}[label=(\arabic*), nolistsep]
\item	$Y^{(1)},\ldots, Y^{(n)}$ are mutually independent and have the same distribution as $X(\cdot)$;
\item $\mathbb{P}'\{Y^{(i)}_x = Y_x\ \textrm{\rm for all}\ x \in U_i\}=1$ for all $i=1,\ldots,n$.
\end{enumerate}
\end{lem}
\begin{rmk}\label{rem:regular_conditional_prob}
\rm Our proof of this lemma uses regular conditional probabilities.~Their existence is guaranteed 
by~our choice of $C[\Rb^d, \Rb]$ as the space of realizations but in principle we could impose weaker requirements on the potential than continuity in Section~\ref{sec: model}.
\end{rmk}

\begin{subproof}[Proof of Lemma~\ref{lem:coupling-using-m-dep}.]
It suffices to give a construction for $n=2$ since then one can iterate it to prove the lemma for general $n$.
Let us define $U_0=U_1\cup U_2$ and  $\mathscr{H}_k = \mathscr{B}\big(C[U_k, \Rb]\big)$ (the Borel sigma algebra on $C[U_k, \Rb]$ equipped with uniform topology), $k=0,1,2$.
For $k=0,1,2$, there is a regular conditional probability $Q_k$  defined on $\Omega  \times \mathscr{H}$ such that $Q_k (\omega, A) = \mathbb{P}(A\,| \mathscr{H}_k) (\omega)$. 
For any $A \in \mathscr{B}(\Omega)$, since $Q_k (\cdot, A)$ is $\mathscr{H}_k$-measurable, 
$Q_k$ depends only on  $u_k := \omega | _{U_k} \in C[U_k, \Rb]$. Therefore, $Q_k$ can be viewed as a function defined on $C[U_k, \Rb] \times \mathscr{H}$.
Let $T_k$ be the projection of $\Omega$ on $C[U_k, \Rb]$ and define $\mu_k(du_k) = \mathbb{P}\big(T_k ^{-1} (du_k)\big)$.\\
Note that by the $\rdep$-dependence of $Y$, $\mu_0=\mu_1\otimes \mu_2$.
Let us now take $\Omega' = \Omega^{\{0,1,2\}}, \mathscr{H}' = \mathscr{B}(\Omega')$ and define $\mathbb{P}'$ as
	\[ \mathbb{P}'(A \times B \times C) = \int Q_0(u_1, u_2, A) Q_1(u_1, B) Q_2(u_2, C) \mu_1(du_1) \mu_2(du_2),	\]
For $i=0,1,2,$ we denote the $i$-th marginal distribution of $\mathbb{P}'$ by $\mathbb{P}'_i$ 
and set $Y_x^{(i)}(\omega) = \omega_i (x)$ for $\omega =(\omega_0, \omega_1, \omega_2) \in \Omega'$. 
It is easy to check that $\mathbb{P}'_0 = \mathbb{P}'_1 = \mathbb{P}'_2 = \mathbb{P}$.
In addition, if we let $\mathbb{P}'_{ij}$ be the marginal distribution of $\mathbb{P}'$ with respect to the $i$-th and $j$-th arguments for $0 \leq i < j \leq 2$,
 then $\mathbb{P}'_{12} = \mathbb{P}'_1 \otimes \mathbb{P}'_2$, i.e., $Y^{(1)}$ and $Y^{(2)}$ are independent and the proof of part~(1) is completed.\\
Let us define $\bar{\mu}_{1}(du\,dv) = \mathbb{P}_{01} ' \big((T_1 \times T_1)^{-1} (du\,dv)\big)$ on $C[U_1, \Rb]^2$.
For $A, B \in \mathscr{B} (C[U_1, \Rb])$, we have
\begin{align*}
	&\bar{\mu}_1(A \times B) = \int Q_0\big(u_1, u_2, T_1^{-1}(A)\big) Q_1\big(u_1, T_1^{-1}(B)\big) \mu_1(du_1) \mu_2(du_2) 
	\\& = \int Q_1(u_1, T_1^{-1}(A)) Q_1(u_1, T_1^{-1}(B)) \mu_1(du_1) 
	= \int \one_A (u_1)  \one_B (u_1)  \mu_1(du_1) = \mu_1 (A\cap B),
\end{align*}
which implies that all the mass of $\bar{\mu}_{1}$ lie on the diagonal of $C[U_1, \Rb]^2$, i.e. 
	\[ \bar{\mu}_{1}\big(\{(u, v) \in C[U_1, \Rb]^2 : u=v\}\big) =1. \]
Therefore, $Y^{(0)}_x = Y^{(1)}_x $ $\mathbb{P}'$-a.s. for all $x \in U_1$.
Similarly, we also obtain $Y^{(0)}_x = Y^{(2)}_x $ $\mathbb{P}'$-a.s. for all $x \in U_2$. Identifying $Y^{(0)}$ with $Y$ completes the proof.
\end{subproof}

\bibliographystyle{alpha}
\bibliography{}
\end{document}